\RequirePackage{luatex85}

\documentclass[a4paper,reqno,10.5pt, oneside]{amsart}

\usepackage{amssymb}
\usepackage{amstext}
\usepackage{amsmath}
\usepackage{amscd}
\usepackage{amsthm}
\usepackage{amsfonts}
\usepackage{enumerate}
\usepackage{graphicx}
\usepackage{latexsym}
\usepackage{mathrsfs}
\usepackage{mathtools}
\usepackage{here}

\usepackage{tikz}
\usepackage{tikz-cd}
\usetikzlibrary{arrows}
\usetikzlibrary{calc,patterns}
\usetikzlibrary{decorations.markings}
\usetikzlibrary{patterns,decorations.pathreplacing}

\usepackage{hyperref}

\newtheorem{theorem}{Theorem}[section]

\newtheorem{corollary}[theorem]{Corollary}
\newtheorem{lemma}[theorem]{Lemma}
\newtheorem{proposition}[theorem]{Proposition}
\newtheorem{definition-proposition}[theorem]{Definition-Proposition}

\theoremstyle{definition}
\newtheorem{definition}[theorem]{Definition}

\newtheorem{example}[theorem]{Example}
\newtheorem{observation}[theorem]{Observation}

\newtheorem{remark}[theorem]{Remark}

\makeatletter
  
  \@addtoreset{equation}{section}
\makeatother

\newcommand{\rmb}{\mathrm{b}}

\newcommand{\rmD}{\mathrm{D}}

\newcommand{\rmH}{\mathrm{H}}

\newcommand{\rmL}{\mathrm{L}}

\newcommand{\rmR}{\mathrm{R}}

\newcommand{\rmZ}{\mathrm{Z}}

\newcommand{\calA}{\mathcal{A}}

\newcommand{\calC}{\mathcal{C}}
\newcommand{\calD}{\mathcal{D}}
\newcommand{\calE}{\mathcal{E}}

\newcommand{\calI}{\mathcal{I}}
\newcommand{\calJ}{\mathcal{J}}
\newcommand{\calK}{\mathcal{K}}

\newcommand{\calM}{\mathcal{M}}

\newcommand{\calO}{\mathcal{O}}
\newcommand{\calP}{\mathcal{P}}

\newcommand{\calT}{\mathcal{T}}

\newcommand{\calX}{\mathcal{X}}

\newcommand{\fkm}{\mathfrak{m}}

\newcommand{\fkp}{\mathfrak{p}}

\newcommand{\fkD}{\mathfrak{D}}

\newcommand{\fkR}{\mathfrak{R}}

\newcommand{\ZZ}{\mathbb{Z}}
\newcommand{\QQ}{\mathbb{Q}}
\newcommand{\RR}{\mathbb{R}}
\newcommand{\CC}{\mathbb{C}}

\newcommand{\VV}{\mathbb{V}}
\newcommand{\TT}{\mathbb{T}}
\newcommand{\kk}{\Bbbk}
\newcommand{\dd}{\underline{d}}

\newcommand{\sfD}{\mathsf{D}}
\newcommand{\sfE}{\mathsf{E}}

\newcommand{\sfR}{\mathsf{R}}

\newcommand{\Spec}{\operatorname{Spec}}

\newcommand{\depth}{\operatorname{depth}}
\newcommand{\degr}{\operatorname{deg}}

\newcommand{\Ext}{\operatorname{Ext}}

\newcommand{\Hom}{\operatorname{Hom}}
\newcommand{\RHom}{\operatorname{\mathbf{R}Hom}}

\newcommand{\Image}{\operatorname{Im}}
\newcommand{\Tr}{\operatorname{Tr}}

\newcommand{\GL}{\operatorname{GL}}
\newcommand{\SL}{\operatorname{SL}}

\newcommand{\Proj}{\operatorname{Proj}}

\newcommand{\End}{\operatorname{End}}
\newcommand{\gldim}{\mathrm{gl.dim}}
\newcommand{\projdim}{\mathrm{proj.dim}}

\newcommand{\add}{\mathsf{add}}
\newcommand{\CM}{\mathsf{CM}}
\newcommand{\sCM}{\underline{\mathsf{CM}}}

\newcommand{\Mod}{\mathsf{Mod}}
\newcommand{\mc}{\mathsf{mod}}

\newcommand{\coh}{\mathsf{coh}}

\newcommand{\per}{\mathsf{per}}

\newcommand{\hd}{\mathsf{hd}}
\newcommand{\tl}{\mathsf{tl}}

\begin{document}

\title[On $2$-representation infinite algebras arising from dimer models]{On $2$-representation infinite algebras arising from \\ dimer models}
\author[Y. Nakajima]{Yusuke Nakajima} 

\address[Y. Nakajima]{Department of Mathematics, Kyoto Sangyo University, Motoyama, Kamigamo, Kita-Ku, Kyoto, 603-8555, Japan}
\email{ynakaji@cc.kyoto-su.ac.jp}



\maketitle

\begin{abstract} 
The Jacobian algebra arising from a consistent dimer model is a bimodule $3$-Calabi-Yau algebra, and its center is a $3$-dimensional Gorenstein toric singularity. 
A perfect matching of a dimer model gives the degree making the Jacobian algebra $\ZZ$-graded. 
It is known that if the degree zero part of such an algebra is finite dimensional, then  it is a $2$-representation infinite algebra which is a generalization of a representation infinite hereditary algebra. 
Internal perfect matchings, which correspond to toric exceptional divisors on a crepant resolution of a $3$-dimensional Gorenstein toric singularity, 
characterize the property that the degree zero part of the Jacobian algebra is finite dimensional. 
Combining this characterization with the theorems due to Amiot-Iyama-Reiten, we show that  the stable category of graded maximal Cohen-Macaulay modules admits a tilting object for any $3$-dimensional Gorenstein toric isolated singularity. 
We then show that all internal perfect matchings corresponding to the same toric exceptional divisor are transformed into each other 
using the mutations of perfect matchings, and this induces derived equivalences of $2$-representation infinite algebras. 
\end{abstract}


\section{\bf Introduction} 

A \emph{dimer model} (or \emph{brane tiling}) is a finite bipartite graph on the real two-torus $\TT$, which induces the polygonal cell decomposition of $\TT$. 
It was introduced in the field of statistical mechanics, and from 2000s string theorists have been used it in the context of quiver gauge theories (see e.g., \cite{FHK,HK,HV}). 
Since then, the relationships between dimer models and many branches of mathematics have been discovered. 
In this paper, we observe the combinatorial structure on dimer models. 
Combining with some known results, it gives rise to several assertions concerning representation theory of algebras. 

\subsection{Background:\,$2$-representation infinite algebras}
\label{subsec_motivation1}

The notion of \emph{$n$-representation infinite algebras} was introduced in \cite{HIO} (see Subsection~\ref{subsec_nrepinfinite} for the precise definition). 
This algebra is a certain analogue of a representation infinite hereditary algebra, and is a distinguished class from the viewpoint of higher dimensional Auslander-Reiten theory. 
This class contains several important examples. 
For example the Beilinson algebra, which arises as the endomorphism ring of the tilting bundle $\bigoplus_{s=0}^n\calO(s)$ on $\mathbb{P}^n$, 
is an $n$-representation infinite algebra (see \cite[Example~2.15]{HIO}). 
Also, it is known that this algebra can be obtained as the degree zero part of a bimodule $(n+1)$-Calabi-Yau algebra of Gorenstein parameter $1$ (see Theorem~\ref{degzero_RI}). 
Some interesting examples of such a construction are given by dimer models as shown in \cite[Section~6]{AIR}. 

Let $\Gamma$ be a dimer model on the real two-torus $\TT$. When we consider the real two-torus $\TT$, we fix the fundamental domain and identify $\TT$ with $\RR^2/\ZZ^2$. 
Since $\Gamma$ is a bipartite graph, the set $\Gamma_0$ of nodes in $\Gamma$ is divided into two parts $\Gamma_0^+, \Gamma_0^-$, and the set $\Gamma_1$ of edges in $\Gamma$ consists of the ones connecting nodes in $\Gamma_0^+$ and those in $\Gamma_0^-$. 
In order to make the situation clear, we color nodes in $\Gamma_0^+$ white, and color nodes in $\Gamma_0^-$ black. 
A connected component of $\TT{\setminus}\Gamma_1$ is called a \emph{face} of $\Gamma$, and we denote by $\Gamma_2$ the set of faces. 
Then, as the dual of a dimer model $\Gamma$, we define the finite connected quiver $Q_\Gamma$ associated with $\Gamma$. 
(If the situation is clear, we simply denote the quiver $Q_\Gamma$ by $Q$.) 
That is, we assign a vertex of $Q$ dual to each face in $\Gamma_2$, an arrow of $Q$ dual to each edge in $\Gamma_1$. 
Here, the orientation of arrows is determined so that the white node is on the right of the arrow. 
For example, Figure~\ref{ex_quiver4a} is a dimer model and the associated quiver, where the outer frame is the fundamental domain of $\TT$. 
In addition, we can obtain the \emph{potential} $W_Q$, which is a linear combination of some cycles in $Q$. 
Using this quiver with potential $(Q,W_Q)$, we can define the \emph{Jacobian algebra} $\calP(Q,W_Q)$, 
which is the path algebra with relations associated to $(Q,W_Q)$ (see Subsection~\ref{subsec_dimer}). 
Under the \emph{consistency condition} (see Definition~\ref{def_consistent}), we see that $\calP(Q,W_Q)$ is a bimodule $3$-Calabi-Yau algebra (see \cite[Theorem~7.1]{Bro}). 

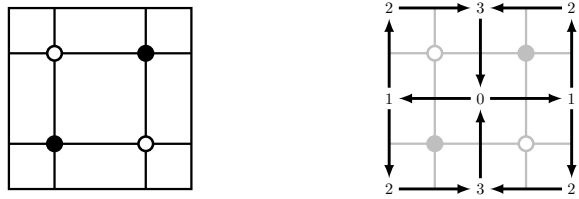
\begin{figure}[H]
\begin{center}
\begin{tikzpicture}
\node (DM) at (0,0) 
{\scalebox{0.6}{
\begin{tikzpicture}
\coordinate (P1) at (1,1); \coordinate (P2) at (3,1); 
\coordinate (P3) at (3,3); \coordinate (P4) at (1,3); 
\draw[line width=0.05cm]  (P1)--(P2)--(P3)--(P4)--(P1);\draw[line width=0.05cm] (0,1)--(P1)--(1,0); \draw[line width=0.05cm]  (4,1)--(P2)--(3,0);
\draw[line width=0.05cm]  (0,3)--(P4)--(1,4);\draw[line width=0.05cm]  (3,4)--(P3)--(4,3);
\filldraw  [ultra thick, fill=black] (P1) circle [radius=0.16] ;\filldraw  [ultra thick, fill=black] (P3) circle [radius=0.16] ;
\draw  [ultra thick,fill=white] (P2) circle [radius=0.16] ;\draw  [ultra thick, fill=white] (P4) circle [radius=0.16] ;

\draw[line width=0.05cm]  (0,0) rectangle (4,4);
\end{tikzpicture}
} }; 

\node (QV) at (5,0) 
{\scalebox{0.6}{
\begin{tikzpicture}[sarrow/.style={black, -latex, very thick}]

\node (Q1) at (2,2){$0$};\node (Q2a) at (0,2){$1$}; \node(Q2b) at (4,2){$1$};\node (Q3a) at (0,0){$2$};
\node(Q3c) at (4,4){$2$};\node(Q3b) at (4,0){$2$};\node(Q3d) at (0,4){$2$};\node (Q4a) at (2,0){$3$};
\node (Q4b) at (2,4){$3$};

\draw[lightgray, line width=0.05cm]  (P1)--(P2)--(P3)--(P4)--(P1);\draw[lightgray, line width=0.05cm] (0,1)--(P1)--(1,0); 
\draw[lightgray, line width=0.05cm]  (4,1)--(P2)--(3,0);
\draw[lightgray, line width=0.05cm]  (0,3)--(P4)--(1,4);\draw[lightgray, line width=0.05cm]  (3,4)--(P3)--(4,3);
\filldraw  [ultra thick, draw=lightgray, fill=lightgray] (P1) circle [radius=0.16] ;\filldraw  [ultra thick, draw=lightgray, fill=lightgray] (P3) circle [radius=0.16] ;
\draw  [ultra thick, draw=lightgray,fill=white] (P2) circle [radius=0.16] ;\draw  [ultra thick, draw=lightgray,fill=white] (P4) circle [radius=0.16] ;

\draw[sarrow, line width=0.064cm] (Q1)--(Q2a);\draw[sarrow, line width=0.064cm] (Q2a)--(Q3a);\draw[sarrow, line width=0.064cm] (Q3a)--(Q4a);
\draw[sarrow, line width=0.064cm] (Q4a)--(Q1);\draw[sarrow, line width=0.064cm] (Q2a)--(Q3d);\draw[sarrow, line width=0.064cm] (Q3d)--(Q4b);
\draw[sarrow, line width=0.064cm] (Q4b)--(Q1);\draw[sarrow, line width=0.064cm] (Q1)--(Q2b);\draw[sarrow, line width=0.064cm] (Q2b)--(Q3b);
\draw[sarrow, line width=0.064cm] (Q3b)--(Q4a);\draw[sarrow, line width=0.064cm] (Q2b)--(Q3c);\draw[sarrow, line width=0.064cm] (Q3c)--(Q4b);
\end{tikzpicture}
} }; 
\end{tikzpicture}
\caption{Dimer model and the associated quiver}
\label{ex_quiver4a}
\end{center}
\end{figure}

The notion of \emph{perfect matching} plays a crucial role in our context. 

\begin{definition}
\label{def_pm}
A \emph{perfect matching} (or \emph{dimer configuration}) of a dimer model $\Gamma$ is a subset $D$ of $\Gamma_1$ such that 
for any node $n\in\Gamma_0$ there is a unique edge in $D$ containing $n$ as the end point. 
For a perfect matching $D$ of $\Gamma$, we denote by $\sfD$ the subset of $Q_1$ obtained as the dual of $D$. We also say that $\sfD$ is a \emph{perfect matching} of $Q$. 
\end{definition}

For a perfect matching $\sfD$ of $Q$, we define the degree $d_{\sfD}$ on each arrow $a\in Q_1$ as 
\begin{equation}
\label{degree_pm}
d_{\sfD}(a)=\begin{cases}
1 \quad\text{if $a\in\sfD$}, \\ 
0 \quad\text{otherwise.}
\end{cases}
\end{equation}
This $d_\sfD$ induces the $\ZZ$-grading on $\calP(Q,W_Q)$, and this makes $\calP(Q,W_Q)$ bimodule $3$-Calabi-Yau with Gorenstein parameter $1$ if $\Gamma$ is consistent (see Subsection~\ref{subsec_2rep_dimer}). 
We define the \emph{truncated Jacobian algebra}, denoted by $\calP(Q,W_Q)_\sfD$, as the degree zero part of the $\ZZ$-graded Jacobian algebra $\calP(Q,W_Q)$ with respect to $d_\sfD$. 
The following is the motivating theorem in this paper. 

\begin{theorem}[{cf. \cite[Corollary~3.6]{AIR},\cite[Theorem~4.12]{MM}}]
\label{motivation_thm1}
Let $\calP(Q,W_Q)$ be the Jacobian algebra associated with a consistent dimer model, and $\sfD$ be a perfect matching of $Q$.  
Then, if the truncated Jacobian algebra $\calP(Q,W_Q)_\sfD$ is finite dimensional, then it is a $2$-representation infinite algebra. 
\end{theorem}

A perfect matching $\sfD$ making the truncated Jacobian algebra $\calA_\sfD\coloneqq\calP(Q,W_Q)_\sfD$ finite dimensional 
can be found using toric geometry. 
Precisely, using perfect matchings of a consistent dimer model, we define the perfect matching polygon $\Delta$ (see Subsection~\ref{subsec_pm}) 
and the $3$-dimensional Gorenstein toric ring $R$, which will be called \emph{toric singularity}, associated with the cone over $\Delta$. 
Also, we can assign a lattice point of $\Delta$ to each perfect matching. 
If $\Delta$ contains an interior lattice point (in other words, if the exceptional locus of a crepant resolution of $\Spec R$ has dimension two), 
then a perfect matching $D$ corresponding to that point, which is called an \emph{internal perfect matching}, gives us the desired property as follows. 

\begin{proposition}[{see Proposition~\ref{findim_internal}}]
\label{motivation_prop1_intro}
Let $Q$ be the quiver associated with a consistent dimer model $\Gamma$. 
Then, we see that $\sfD$ is an internal perfect matching of $Q$ if and only if $\calA_\sfD$ is a finite dimensional algebra, in which case 
$\calA_\sfD$ is a $2$-representation infinite algebra. 
\end{proposition}

Although this is the immediate consequence of toric geometry associated to dimer models, 
we give a brief explanation in Subsection~\ref{subsec_2rep_dimer} for the sake of completeness. 

\subsection{Background:\,Triangle equivalences obtained via dimer models} 
\label{subsec_motivation2}

It is known that the center of the Jacobian algebra associated with a consistent dimer model is a $3$-dimensional Gorenstein toric singularity, 
and such a toric singularity can be obtained in this way (see Theorem~\ref{exist_dimer}). 
When we study maximal Cohen-Macaulay (MCM) modules over such a toric singularity, a finite dimensional truncated Jacobian algebra is important as we will see below. 

Let $R$ be a $d$-dimensional Cohen-Macaulay commutative ring. 
We recall that an $R$-module $M$ is an MCM module if $\depth_{R_\fkp}M_\fkp=\dim R_\fkp$ for any $\fkp\in\Spec R$ or $M=0$. 
Let $\CM(R)$ be the category of MCM $R$-modules and $\sCM(R)$ be the stable category of MCM $R$-modules. 
It is known that if $R$ is a $d$-dimensional local Gorenstein isolated singularity, then $\sCM(R)$ is a $(d-1)$-Calabi-Yau triangulated category \cite{Aus}. 
Here, we say that a $\kk$-linear $\Hom$-finite triangulated category $\calT$ is \emph{$n$-Calabi-Yau} (\emph{$n$-CY}) if there is a natural isomorphism 
$$
\Hom_\calT(X,Y)\cong\rmD\Hom_\calT(Y,X[n])
$$
for any $X,Y\in\calT$ and $\rmD(-)=\Hom_\calT(-,\kk)$. 
We note that there exists a triangulated equivalence between $\sCM(R)$ and the \emph{singularity category} $\calD_{\rm sg}(R)$ of $R$ by \cite{Buc}. 
We then introduce a non-commutative crepant resolution in the sense of Van den Bergh \cite{VdB}. 

\begin{definition}
\label{def_NCCR}
Let $R$ be a Gorenstein normal domain, and $M$ be a non-zero reflexive $R$-module. Let $\Lambda\coloneqq\End_R(M)$. 
We say that $\Lambda$ is a \emph{non-commutative crepant resolution} (\emph{NCCR}) of $R$ if $\gldim \Lambda<\infty$ and $\Lambda$ is an MCM $R$-module.
\end{definition}

Let $M$ be an MCM $R$-module that contains no free direct summands. 
If $\End_R(R\oplus M)$ is an NCCR of $R$, then $M$ is a $(d-1)$-cluster tilting object in $\sCM(R)$ \cite[Theorem~5.2.1]{Iya2}. 
We here recall that an object $M$ in a triangulated category $\calT$ is called \emph{$n$-cluster tilting} (\emph{$n$-CT}) object if it satisfies 
\begin{align*}
\add M&=\{X\in\calT \mid \Hom_\calT(X, M[i])=0 \text{\, for all \,} i=1, \dots, n-1 \} \\
&=\{X\in\calT \mid \Hom_\calT(M, X[i])=0 \text{\, for all \,} i=1, \dots, n-1 \}.
\end{align*}

On the other hand, there is an important $n$-CY category admitting an $n$-CT object which is called a \emph{generalized} \emph{$n$-cluster category}. 
This category was introduced in \cite{BMRRT} (see also \cite{CCS}) for a finite dimensional hereditary algebra as the categorification of cluster algebras due to Fomin and Zelevinsky \cite{FZ}, and it was generalized in \cite{Ami,Guo}. 
In light of these facts, it is natural to ask a relationship between these CY categories admitting cluster tilting objects. 
Indeed, for some cases we can obtain a triangulated equivalence between the $(d-1)$-CY category $\sCM(R)$ 
and a generalized $(d-1)$-cluster category. 
An interesting family of such equivalences is also given by dimer models as we will see in Theorem~\ref{motivation_thm2} below. 
That is, using the grading $d_\sfD$ induced from a perfect matching $\sfD$, we have the following equivalences. 
Here, $Q_\sfD$ denotes the subquiver of $Q$ whose set of vertices is the same as $Q$ and set of arrows is $Q_1{\setminus} \sfD$, 
especially $Q_\sfD$ is also a finite connected quiver. 

\begin{theorem}
[{see \cite[Theorem~6.3]{AIR}}]
\label{motivation_thm2}
Let $\Gamma$ be a consistent dimer model, and $(Q,W_Q)$ be the associated quiver with potential. 
Let $\calP(Q,W_Q)$ be the Jacobian algebra, and we denote by $R$ the center of $\calP(Q,W_Q)$. 
If there exists a perfect matching $\sfD$ of $Q$ and a vertex $i$ of $Q$ satisfying the following properties {\rm :} 
\begin{itemize}
\item[\rm (P1)] the truncated Jacobian algebra $\calA_\sfD\coloneqq\calP(Q,W_Q)_\sfD$ is finite dimensional, 
\item[\rm (P2)] $i$ is a source of the quiver $Q_\sfD$, 
\item[\rm (P3)] $\calP(Q,W_Q)/\langle e_i\rangle$ is finite dimensional, where $e_i$ is the primitive idempotent corresponding to $i$, 
\end{itemize}
then $R$ is a $3$-dimensional Gorenstein toric isolated singularity, and for the algebra $\calA_{\sfD,e_i}\coloneqq \calA_\sfD/\langle e_i\rangle$ we have the following triangle equivalences. 
\begin{equation}
\label{motivated_equiv}
\begin{tikzcd}
  \calD^\rmb(\mc\calA_{\sfD,e_i}) \arrow[r, "\cong"] \arrow[d] & \sCM^\ZZ(R) \arrow[d] \\
  \calC_2(\calA_{\sfD,e_i}) \arrow[r,  "\cong" ] &\sCM(R)
\end{tikzcd}
\end{equation}
Here, $\calC_2(\calA_{\sfD,e_i})$ is the generalized $2$-cluster category, which is defined as the triangulated hull of the orbit category $\calD^\rmb(\mc\calA_{\sfD,e_i})/\mathbb{S}{\circ}[-2]$ where $\mathbb{S}$ is the Serre functor of $\calD^\rmb(\mc\calA_{\sfD,e_i})$, see \cite{Ami,AIR}. 
\end{theorem}

By Proposition~\ref{motivation_prop1_intro}, if $\sfD$ is an internal perfect matching, then $\calA_\sfD$ 
satisfies the condition (P1) in Theorem~\ref{motivation_thm2}. 
Moreover, if $R$ is an isolated singularity, we can find a vertex $i\in Q_0$ satisfying the conditions (P2), (P3). 
Thus, it follows from \cite[Theorem~4.1 and 6.3]{AIR} that we have the following statements. 

\begin{corollary}[{see Corollary~\ref{main_cor}}]
\label{cor_intro_clusterequiv}
Let $R$ be a non-regular $3$-dimensional Gorenstein toric isolated singularity that is not the $A_1$-singularity $($i.e., $R\not\cong\kk[x,y,z,w]/(xy-zw)$$)$. 
Then, there exists an internal perfect matching $\sfD$ and the primitive idempotent $e_i$ satisfying {\rm(P1)--(P3)} in Theorem~\ref{motivation_thm2}. 
Thus, we have equivalences {\rm (\ref{motivated_equiv})}. 
\end{corollary}

\begin{corollary}[{see Corollary~\ref{existence_tilting}}]
Let $R$ be a $3$-dimensional Gorenstein toric isolated singularity. Let $\Gamma$ be a consistent dimer model associated with $R$. 
Then, there exists a $\ZZ$-grading on $R$ induced from a perfect matching of $\Gamma$ via \eqref{degree_pm} such that 
the stable category $\sCM^\ZZ R$ of $\ZZ$-graded MCM $R$-modules admits a tilting object. 
\end{corollary}

Here, an object $X$ in a triangulated category $\calT$ is said to be \emph{tilting} if $\Hom_\calT(X,X[i])=0$ for all integers $i\neq0$ and 
the smallest thick subcategory containing $X$ is $\calT$. 

Concerning Corollary~\ref{cor_intro_clusterequiv}, 
we note that even if $R$ is the $A_1$-singularity, we have a similar equivalence, see Example~\ref{rem_conifold}. 

\subsection{Main results} 

As we saw in Subsection~\ref{subsec_motivation1} and \ref{subsec_motivation2}, an internal perfect matching plays an important role. 
In general, there are some internal perfect matchings corresponding to the same interior lattice point of the perfect matching polygon $\Delta$, 
thus it is also important to understand relationships among such internal perfect matchings. 
To do this, we consider the notion of the \emph{mutation of perfect matchings} (see Subsection~\ref{sec_mutation_pm}), which is the operation making a given perfect matching a different one corresponding to the same lattice point. Using the mutation, we can obtain the following theorem. 

\begin{theorem}[{see Theorem~\ref{pm_mutation_equiv}}]
\label{intro_mutation_equiv}
Let $\Gamma$ be a consistent dimer model. 
Then, we see that any pair of internal perfect matchings of $\Gamma$ are transformed into each other by repeating mutations of perfect matchings if and only if they correspond to the same interior lattice point of the perfect matching polygon. 
\end{theorem}

Even if $\sfD_i, \sfD_j$ are internal perfect matchings of $Q$ corresponding to the same interior lattice point, $2$-representation infinite algebras $\calA_{\sfD_i}$ and $\calA_{\sfD_j}$ are not isomorphic in general. 
However, the mutation of perfect matchings induces a certain tilting module by \cite{IO}, thus as an application of Theorem~\ref{intro_mutation_equiv}  we have derived equivalences between $2$-representation infinite algebras as in Theorem~\ref{intro_derived_equiv1}. 
We remark that this assertion also follows from \cite[Theorem~7.2 and Remark~7.3]{IU5}, which uses a tilting object in the derived category of coherent sheaves on a $2$-dimensional toric Deligne-Mumford stack. 
Whereas we will give another proof using Theorem~\ref{intro_mutation_equiv} because it is interesting from the viewpoint of representation theory of algebras.

\begin{theorem}[{see Theorem~\ref{derivedequ_pm}}]
\label{intro_derived_equiv1}
Let $\Gamma$ be a consistent dimer model and $Q$ be the associated quiver. 
Let $\sfD_i, \sfD_j$ be internal perfect matchings of $Q$ corresponding to the same interior lattice point of the perfect matching polygon $\Delta$. 
Then, we have an equivalence $\calD^\rmb(\mc\calA_{\sfD_i})\cong\calD^\rmb(\mc\calA_{\sfD_j})$. 
\end{theorem}

On the other hand, it is known that for any lattice polygon $\Delta$ there exists a consistent dimer model giving $\Delta$ as the perfect matching polygon \cite{Gul,IU3}, 
but such a dimer model is not unique. 
A way to understand a relationship among consistent dimer models having the same perfect matching polygon 
is to consider the \emph{mutation of dimer models}, 
which is obtained as the dual of the \emph{mutation of the quiver with potential} associated with a dimer model, 
because this operation does not change the associated perfect matching polygon (see Subsection~\ref{subsec_mutation_QP}). 
Using the mutation of dimer models and tilting theory, we also have derived equivalences of $2$-representation algebras as follows. 
We remark that this also follows from \cite[Theorem~7.2]{IU5} in the stronger sense (see Remark~\ref{rem_main_derived_eq}). 

\begin{theorem}[{see Theorem~\ref{main_thm_derived_eq}}]
\label{intro_derived_equiv2}
Let $\Gamma, \Gamma^\prime$ be consistent dimer models that can be obtained from each other using the mutations of dimer models $($see Definition~\ref{def_mut_QP}, \ref{def_mutation_dimer}$)$, 
and we assume that the perfect matching polygon $\Delta$ of $\Gamma$ $($and hence $\Gamma^\prime$$)$ contains an interior lattice point. 
Let $(Q,W_Q), (Q^\prime,W_{Q^\prime})$ be the quivers with potential associated with $\Gamma, \Gamma^\prime$ respectively. 
If $D$ and $D^\prime$ are respectively internal perfect matchings of $\Gamma, \Gamma^\prime$ corresponding to the same interior lattice point of $\Delta$, then we have an equivalence 
$$
\calD^\rmb(\mc \,\widehat{\calP}(Q,W_Q)_\sfD)\cong\calD^\rmb(\mc \,\widehat{\calP}(Q^\prime,W_{Q^\prime})_{\sfD^\prime}).
$$
where $\widehat{\calP}(Q,W_Q)$ is the complete Jacobian algebra. 
\end{theorem}

\medskip

The structure of this paper is as follows. In Section~\ref{sec_pre}, we prepare some notation and facts regarding dimer models. 
In Section~\ref{section_2RI}, we introduce $n$-representation infinite algebras (we are especially interested in the case of $n=2$) and discuss which perfect matching makes the degree zero part of the graded Jacobian algebra $2$-representation infinite. 
As a conclusion, we show that an internal perfect matching gives the desired answer (see Proposition~\ref{findim_internal}). 
In Section~\ref{sec_cluster_eq}, we consider the stable categories of (graded) MCM modules for $3$-dimensional Gorenstein toric isolated singularities, and show equivalences (\ref{motivated_equiv}) using the results in the previous section (see Corollary~\ref{main_cor}). 
Moreover, we show the existence of tilting object for this stable category (see Corollary~\ref{existence_tilting}). 
In Section~\ref{sec_mutation_pm}, we focus on internal perfect matchings and their mutations. 
In particular, we show that all internal perfect matchings corresponding to the same interior lattice point are successive mutations of one another (see Theorem~\ref{pm_mutation_equiv}). 
In Section~\ref{sec_derived_eq_2rep}, we turn our attention to the bounded derived categories of finitely generated modules over $2$-representation infinite algebras, and show their equivalences (see Theorem~\ref{derivedequ_pm} and \ref{main_thm_derived_eq}) using APR tilting modules and the mutations of (graded) quivers with potentials. 

\subsection{Notation and Conventions} 
Throughout this paper, $\kk$ is an algebraically closed field of characteristic zero. 
In this paper, all modules are left modules. 

Let $A$ be a $\kk$-algebra. 
We denote by $\Mod A$ the category of $A$-modules, by $\mc A$ the category of finitely generated $A$-modules. 
We also denote by $\add_AM$ the full subcategory consisting of direct summands of finite direct sums of some copies of $M\in \mc A$. 
We denote by $A^{\rm op}$ the opposite algebra of $A$, and let $A^e\coloneqq A\otimes_\kk A^{\rm op}$ be the \emph{enveloping algebra} of $A$. 
In particular, we can consider a left $A^e$-module as an $(A,A)$-bimodule. We denote by $\per A$ the thick subcategory generated by $A$. 
Let $B$ be a $\ZZ$-graded $\kk$-algebra. 
We denote by $\Mod^\ZZ B$ the category of $\ZZ$-graded $B$-modules, by $\mc^\ZZ B$ the category of finitely generated $\ZZ$-graded $B$-modules. 

When we consider a composition of morphism, $fg$ means we first apply $f$ then $g$. 
With this convention, $\Hom_A(M, X)$ is an $\End_A(M)$-module and $\Hom_A(X, M)$ is an $\End_A(M)^{\rm op}$-module. 
Similarly, when we consider a quiver, a path $ab$ means $a$ is followed by $b$. 

For an abelian category $\calC$, we denote by $\calK(\calC)$ the homotopy category and by $\calD(\calC)$ the derived category. 
We denote by $\calK^\rmb(\calC)$ the bounded homotopy category and by $\calD^\rmb(\calC)$ the bounded derived category. 
For $X,Y\in\calC$, let $\calP(X,Y)$ (resp. $\calI(X,Y)$) be the subset of $\Hom_\calC(X,Y)$ consisting of morphisms factoring through a projective (resp. injective) object in $\calC$. 
We denote the \emph{stable} (resp. \emph{costable}) \emph{category} of $\calC$ by $\underline{\calC}$ (resp. $\overline{\calC}$), 
that is, it has the same objects as $\calC$, and the morphism spaces are defined as $\Hom_\calC(X,Y)/\calP(X,Y)$ (resp. $\Hom_\calC(X,Y)/\calI(X,Y)$). 

\section{\bf Preliminaries on dimer models} 
\label{sec_pre}

We here introduce some notions related to dimer models. This section is mostly for a review and fixing notation. 
Some important ideas and concepts appearing in this section are derived from theoretical physics (see e.g., \cite{FHK,HK,HV}), 
and have been developed in the various references cited throughout this section.  

\subsection{Jacobian algebras associated with dimer models} 
\label{subsec_dimer}

Let $Q=(Q_0,Q_1)$ be the quiver obtained as the dual of a dimer model, where $Q_0$ is the set of vertices and $Q_1$ is the set of arrows. 
Let $\hd, \tl:Q_1\rightarrow Q_0$ be maps sending an arrow $a\in Q_1$ to the head of $a$ and the tail of $a$ respectively. 
A \emph{nontrivial path} is a finite sequence of arrows $a=a_1\cdots a_r$ with $\hd(a_\ell)=\tl(a_{\ell+1})$ for $\ell=1, \dots, r-1$. 
We say that a path $a$ is a \emph{cycle} if $\hd(a)=\tl(a)$. 
We define the \emph{length of path} $a=a_1\cdots a_r$ as $r \,(\ge 1)$, and denote by $Q_r$ the set of paths of length $r$. 
We consider each vertex $i\in Q_0$ as a trivial path $e_i$ with $\hd(e_i)=\tl(e_i)=i$, and define the length of $e_i$ as $0$. 
Extending the maps $\hd, \tl$, we define $\tl(a)=\tl(a_1), \hd(a)=\hd(a_r)$ for a path $a=a_1\cdots a_r$. 

Then, we consider the \emph{path algebra} $\kk Q$, which is the $\kk$-algebra whose $\kk$-basis consists of paths in $Q$. 
The multiplication of $\kk Q$ is defined as 
$a\cdot b=ab$ (resp. $a\cdot b=0$) if $\hd(a)=\tl(b)$ (resp. $\hd(a)\neq \tl(b)$) for paths $a, b$, and we extend this multiplication linearly. 
We then denote by $[\kk Q, \kk Q]$ the $\kk$-vector space generated by all commutators in $\kk Q$ and set the vector space $\kk Q_{\mathrm{cyc}}\coloneqq \kk Q/[\kk Q, \kk Q]$. 
The space $\kk Q_{\mathrm{cyc}}$ has a basis consisting of cycles in $Q$. 
By the definition of the quiver associated with a dimer model, for each node $n\in\Gamma_0$ on a dimer model, we have the cycle $\omega_n\in\kk Q_{\mathrm{cyc}}$ that is obtained as the product of all arrows around the node $n$. We call such a cycle \emph{small cycle}. 
In particular, small cycles dual to white (resp. black) nodes are oriented clockwise (resp. anti-clockwise). 
A \emph{potential} of $Q$ is a linear combination $W\in\kk Q_{\mathrm{cyc}}$ of cycles in $Q$ having the length at least $2$, and call a pair $(Q,W)$ a \emph{quiver with potential} (= \emph{QP}). 
For the quiver $Q$ associated with a dimer model, we define the potential $W_Q$ as 
\[
W_Q\coloneqq \sum_{n\in \Gamma^+_0}\omega_n-\sum_{n\in \Gamma^-_0}\omega_n. 
\]
For each small cycle $\omega_n$, we choose an arrow $a\in \omega_n$ and consider $\hd(a)$ as the starting point of $\omega_n$. 
Then, we may write $e_{\hd(a)}\omega_ne_{\hd(a)}\coloneqq a_1\cdots a_ra$ using some path $a_1\cdots a_r$. 
We define the partial derivative of $\omega_n$ with respect to $a$ by $\partial_a\omega_n\coloneqq a_1\cdots a_r$. 
We note that $\partial_a \omega_n = 0$ for an arrow $a$ that is not contained in $\omega_n$. 
Extending this derivative linearly, we also define $\partial_a W_Q$ for any $a\in Q_1$.
Then, we consider the two-sided ideal $\calJ_Q\coloneqq \langle\partial_aW_Q\mid a\in Q_1\rangle\subset\kk Q$. 
We define the \emph{Jacobian algebra} of a dimer model as $\calP(Q, W_Q)\coloneqq \kk Q/\calJ_Q$.
By construction, $\partial_a W_Q$ gives the relation in $Q$ for each arrow $a\in Q_1$. 
Here, a \emph{relation} in $Q$ is a $\kk$-linear combination of paths of length at least $2$ having the same head and tail. 
Namely, for each arrow $a\in Q_1$, there are precisely two oppositely oriented small cycles containing the arrow $a$ as a boundary. 
Let $p_a^{\pm}$ be the paths from $\hd(a)$ around the boundary of such small cycles to $\tl(a)$ (see Figure~\ref{relation}). 
Then, $\partial_a W_Q$ is described as $\partial_aW_Q=p_a^+-p_a^-$, and this is a relation in $Q$. 
Thus, for the set of relations $\fkR=\{p_a^+-p_a^- \mid a\in Q_1\}$ in $Q$, we have $\calJ_Q=\langle \fkR\rangle$. 

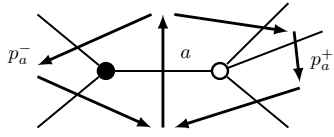
\begin{figure}[H] 
\[\scalebox{0.75}{
\begin{tikzpicture}[sarrow/.style={black, -latex, very thick}]
\node (N1) at (0,0){$$}; \node (N2) at (2,0){$$}; 
\draw[line width=0.035cm]  (N1)--(N2); \draw[line width=0.035cm]  (N1)--(-1.2,1); \draw[line width=0.035cm]  (N1)--(-1.2,-1);
\draw[line width=0.035cm]  (N2)--(3.2,1.2);  \draw[line width=0.035cm]  (N2)--(3.8,0.7); \draw[line width=0.035cm]  (N2)--(3.2,-1); 
\filldraw  [very thick, fill=black] (0,0) circle [radius=0.14] ;
\draw  [very thick, fill=white] (2,0) circle [radius=0.14] ;
\draw[sarrow, line width=0.05cm]  (1,-1)--(1,1); \draw[sarrow, line width=0.05cm]  (0.8,1)--(-1.2,0.1); 
\draw[sarrow, line width=0.05cm]  (-1.2,-0.1)--(0.8,-1); \draw[sarrow, line width=0.05cm]  (1.2,1)--(3.2,0.7); 
\draw[sarrow, line width=0.05cm]  (3.3,0.7)--(3.4, -0.2); 
\draw[sarrow, line width=0.05cm]  (3.4,-0.3)--(1.2,-1); 

\node  at (1.4,0.3) {$a$} ;\node  at (-1.5,0.3) {$p^-_a$} ;\node  at (3.8,0.2) {$p^+_a$} ;
\end{tikzpicture}
}\]
\caption{An example of $p_a^+$ and $p_a^-$}
\label{relation}
\end{figure}

\begin{remark}
\label{rem_bivalent}
We say that a node on a dimer model is \emph{$n$-valent} if the number of edges incident to that node is $n$. 
If a dimer model contains a $2$-valent node, then we may remove it using the join move which will be defined below, because this operation does not change the Jacobian algebra up to isomorphism. 
Thus, in the rest of this paper, we assume that dimer models do not contain $2$-valent nodes unless otherwise mentioned, 
and hence the length of a small cycle is at least three. 

Here, we define the join move and the split move. 
The \emph{join move} is an operation removing a $2$-valent node and joining two distinct nodes connected to it as shown in Figure~\ref{split_join}. 
On the other hand, the \emph{split move} is an operation inserting a $2$-valent node. 
In general, there are several way to insert a $2$-valent node. Figure~\ref{split_join} is an example of the split move. 

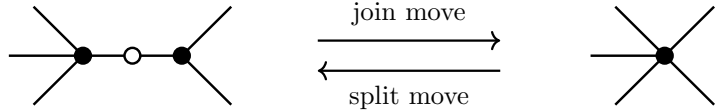
\begin{figure}[H]
\begin{center}
{\scalebox{1}{
\begin{tikzpicture} 

\node at (3.8,0.55) {join move} ; \node at (3.8,-0.55) {split move} ; 
\draw[->, line width=0.03cm] (2.6,0.2)--(5,0.2); 
\draw[<-, line width=0.03cm] (2.6,-0.2)--(5,-0.2); 

\node (Dimer_a) at (0,0)
{\scalebox{0.65}{
\begin{tikzpicture} 
\coordinate (B1) at (-1,0); \coordinate (B2) at (1,0); 

\coordinate (W1) at (-2,1); \coordinate (W2) at (-2.5,0); \coordinate (W3) at (-2,-1); 
\coordinate (W4) at (0,0); \coordinate (W5) at (2,1); \coordinate (W6) at (2,-1); 

\draw[line width=0.05cm] (W1)--(B1); \draw[line width=0.05cm] (W2)--(B1);  
\draw[line width=0.05cm] (W3)--(B1); \draw[line width=0.05cm] (B1)--(W4); \draw[line width=0.05cm] (W4)--(B2);   
\draw[line width=0.05cm] (W5)--(B2); \draw[line width=0.05cm] (W6)--(B2);

\filldraw  [line width=0.05cm, fill=black] (B1) circle [radius=0.16] ; \filldraw  [line width=0.05cm, fill=black] (B2) circle [radius=0.16] ;

\draw [line width=0.05cm, fill=white] (W4) circle [radius=0.16] ;
\end{tikzpicture} }} ;

\node (Dimer_b) at (7,0)
{\scalebox{0.65}{
\begin{tikzpicture} 
\coordinate (B1) at (0,0); 

\coordinate (W1) at (-1,1); \coordinate (W2) at (-1.5,0); \coordinate (W3) at (-1,-1); 
\coordinate (W5) at (1,1); \coordinate (W6) at (1,-1); 

\draw[line width=0.05cm] (W1)--(B1); \draw[line width=0.05cm] (W2)--(B1);  
\draw[line width=0.05cm] (W3)--(B1); \draw[line width=0.05cm] (W5)--(B1); \draw[line width=0.05cm] (W6)--(B1);

\filldraw  [line width=0.05cm, fill=black] (B1) circle [radius=0.16] ; 
\end{tikzpicture} }} ;

\end{tikzpicture} 
}}
\caption{Join move and split move}
\label{split_join}
\end{center}
\end{figure}
\end{remark}

In what follows, we will impose the extra condition called \emph{consistency condition} on a dime model. 
Under this condition, a dimer model gives a crepant resolution and an NCCR of a $3$-dimensional Gorenstein toric singularity 
(see Subsection~\ref{CCR_NCCR}). 
In the literature, there are several consistency conditions (see e.g., \cite{Boc_consist,Bro,Dav,Gul,HV,IU2,KS,MR}). 
The relationships between those conditions were studied in \cite{IU2,Boc_consist}, and almost all conditions are equivalent under mild assumptions. 
We thus note one of them. 

\begin{definition}[{cf.\cite{HV}}]
\label{def_consistent}
Let $Q$ be the quiver associated with a dimer model $\Gamma$. 
We say that a dimer model $\Gamma$ is \emph{consistent} if it admits a positive grading $\sfR:Q_1\rightarrow \RR_{>0}$ satisfying the following conditions: 
\begin{itemize}
\item [(1)] $\displaystyle\sum_{a\in \omega_n}\sfR(a)=2$ \;for any $n\in\Gamma_0$, 
\item [(2)] $\displaystyle\sum_{\hd(a)=i}(1-\sfR(a))+\displaystyle\sum_{\tl(a)=i}(1-\sfR(a))=2$ \;for any $i\in Q_0$,  
\end{itemize}
in which case $\sfR$ is called a \emph{consistent $\sfR$-charge}. 
\end{definition} 

\subsection{Perfect matchings of dimer models} 
\label{subsec_pm}

In this subsection, we focus on perfect matchings (see Definition~\ref{def_pm}). 
In general, every dimer model does not necessarily have a perfect matching. 
If a dimer model is consistent, then it has a perfect matching and every edge is contained in some perfect matchings (see e.g., \cite[Proposition~8.1]{IU2}). 

For each edge contained in a perfect matching on $\Gamma$, we give the orientation from a white node to a black node. 
We fix a perfect matching $D_0$, which will be called the \emph{reference perfect matching}. 
For any perfect matching $D$, the difference of two perfect matchings $D-D_0$ forms a $1$-cycle, 
and hence we consider it as an element in the homology group $\rmH_1(\TT)\cong\ZZ^2$. 
When we consider $D-D_0$ as the element of $\rmH_1(\TT)$, we denote it by $[D-D_0]$. 
We then define the lattice polygon 
$$
\Delta_\Gamma\coloneqq\text{\rm conv}\{[D-D_0]\in\ZZ^2 \mid D \text{ is a perfect matching of $\Gamma$}\}
$$
as the convex hull of lattice points associated to perfect matchings. 
We call $\Delta_\Gamma$ the \emph{perfect matching $($PM$)$ polygon} (or \emph{characteristic polygon}) of $\Gamma$, and sometimes we simply  denote this by $\Delta$. 
Although this lattice polygon depends on the reference perfect matching, it is determined up to translations. 

\begin{definition}
Fix a perfect matching $D_0$. We say that a perfect matching $D$ is 
\begin{itemize}
\item a \emph{corner} (or \emph{extremal}) \emph{perfect matching} if $[D-D_0]$ lies on a vertex of $\Delta$, 
\item a \emph{boundary} (or \emph{external}) \emph{perfect matching} if $[D-D_0]$ lies on the boundary of $\Delta$, 
\item an \emph{internal} \emph{perfect matching} if $[D-D_0]$ lies on an interior lattice point of $\Delta$. 
\end{itemize}
\end{definition}

Since $[D_i-D_j]=[D_i-D_k]-[D_j-D_k]$ for any perfect matchings $D_i, D_j, D_k$, we see that corner, boundary, and internal perfect matchings do not depend on a choice of the reference one. 
If a dimer model is consistent, 
then there exists a unique corner perfect matching corresponding to each vertex of $\Delta$ (see e.g., \cite[Corollary~4.27]{Bro}, \cite[Proposition~9.2]{IU2}). 
Thus, we can give a cyclic order to corner perfect matchings along the corresponding vertices of $\Delta$ in the anti-clockwise direction. 
In addition, we say that two corner perfect matchings are \emph{adjacent} if they are adjacent with respect to the given cyclic order. 
For example,  Figure~\ref{pm_4a} is the list of perfect matchings of the dimer model given in Figure~\ref{ex_quiver4a}. 

\newcommand{\basicdimerex}{

\newcommand{\edgewidth}{0.05cm} 
\newcommand{\nodewidth}{0.045cm} 
\newcommand{\noderad}{0.18} 

\coordinate (P1) at (1,1); \coordinate (P2) at (3,1); 
\coordinate (P3) at (3,3); \coordinate (P4) at (1,3); 
\draw[line width=\edgewidth]  (0,0) rectangle (4,4);
\draw[line width=\edgewidth]  (P1)--(P2)--(P3)--(P4)--(P1); \draw[line width=\edgewidth] (0,1)--(P1)--(1,0); \draw[line width=\edgewidth]  (4,1)--(P2)--(3,0);
\draw[line width=\edgewidth]  (0,3)--(P4)--(1,4); \draw[line width=\edgewidth]  (3,4)--(P3)--(4,3);
\draw  [line width=\nodewidth, fill=black] (P1) circle [radius=\noderad] ; \draw  [line width=\nodewidth, fill=black] (P3) circle [radius=\noderad] ;
\draw  [line width=\nodewidth, fill=white] (P2) circle [radius=\noderad] ; \draw  [line width=\nodewidth, fill=white] (P4) circle [radius=\noderad] ;
}

\begin{figure}[H]
\begin{center}
{\scalebox{0.8}{
\begin{tikzpicture} 

\newcommand{\pmwidth}{0.35cm} 
\newcommand{\pmcolor}{gray} 

\node at (0,-1.5) {$D_1$}; \node at (3.3,-1.5) {$D_2$}; \node at (6.6,-1.5) {$D_3$}; \node at (9.9,-1.5) {$D_4$};
\node at (0,-4.7) {$D_5$}; \node at (3.3,-4.7) {$D_6$}; \node at (6.6,-4.7) {$D_7$}; \node at (9.9,-4.7) {$D_8$}; 

\node (PM1) at (0,0) 
{\scalebox{0.55}{
\begin{tikzpicture}
\draw[line width=\pmwidth,color=\pmcolor] (P3)--(P4);\draw[line width=\pmwidth,color=\pmcolor] (P1)--(0,1);\draw[line width=\pmwidth,color=\pmcolor] (P2)--(4,1);
\basicdimerex
\end{tikzpicture} }}; 

\node (PM2) at (3.3,0) 
{\scalebox{0.55}{
\begin{tikzpicture}
\draw[line width=\pmwidth,color=\pmcolor] (P3)--(P2);\draw[line width=\pmwidth,color=\pmcolor] (P4)--(1,4);\draw[line width=\pmwidth,color=\pmcolor] (P1)--(1,0);
\basicdimerex
\end{tikzpicture} }} ;  

\node (PM3) at (6.6,0) 
{\scalebox{0.55}{
\begin{tikzpicture}
\draw[line width=\pmwidth,color=\pmcolor] (P2)--(P1);\draw[line width=\pmwidth,color=\pmcolor] (P4)--(0,3);\draw[line width=\pmwidth,color=\pmcolor] (P3)--(4,3);
\basicdimerex
\end{tikzpicture} }}; 

\node (PM4) at (9.9,0) 
{\scalebox{0.55}{
\begin{tikzpicture}
\draw[line width=\pmwidth,color=\pmcolor] (P4)--(P1);\draw[line width=\pmwidth,color=\pmcolor] (P3)--(3,4);\draw[line width=\pmwidth,color=\pmcolor] (P2)--(3,0);
\basicdimerex
\end{tikzpicture} }} ;

\node (PM5) at (0,-3.2) 
{\scalebox{0.55}{
\begin{tikzpicture}
\draw[line width=\pmwidth,color=\pmcolor] (P1)--(P4);\draw[line width=\pmwidth,color=\pmcolor] (P2)--(P3);
\basicdimerex
\end{tikzpicture}
 }}; 
 
 \node (PM6) at (3.3,-3.2) 
{\scalebox{0.55}{
\begin{tikzpicture}
\draw[line width=\pmwidth,color=\pmcolor] (P1)--(P2);\draw[line width=\pmwidth,color=\pmcolor] (P3)--(P4);
\basicdimerex
\end{tikzpicture}
 }}; 
 
  \node (PM7) at (6.6,-3.2) 
{\scalebox{0.55}{
\begin{tikzpicture}
\draw[line width=\pmwidth,color=\pmcolor] (P1)--(1,0);\draw[line width=\pmwidth,color=\pmcolor] (P2)--(3,0); 
\draw[line width=\pmwidth,color=\pmcolor] (P3)--(3,4);\draw[line width=\pmwidth,color=\pmcolor] (P4)--(1,4);
\basicdimerex
\end{tikzpicture}
 }}; 

 \node (PM8) at (9.9,-3.2) 
{\scalebox{0.55}{
\begin{tikzpicture}
\draw[line width=\pmwidth,color=\pmcolor] (P1)--(0,1);\draw[line width=\pmwidth,color=\pmcolor] (P2)--(4,1); 
\draw[line width=\pmwidth,color=\pmcolor] (P3)--(4,3);\draw[line width=\pmwidth,color=\pmcolor] (P4)--(0,3);
\basicdimerex
\end{tikzpicture}
 }}; 

\end{tikzpicture}
}}
\caption{Perfect matchings of the dimer model given in Figure~\ref{ex_quiver4a}}
\label{pm_4a}
\end{center}
\end{figure}
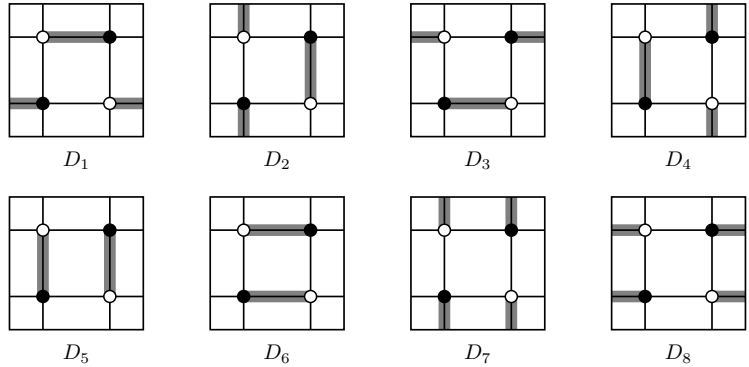

We fix $D_5$ as the reference perfect matching. 
We see that $D_1,\dots,D_4$ correspond to $(1,0),(0,1),(-1,0),(0,-1)$ respectively (see Figure~\ref{difference_4a}). 
We also see that  $D_5,\dots,D_8$ correspond to $(0,0)$, and hence the PM polygon is given as shown in the rightmost of Figure~\ref{difference_4a}. 
Therefore, $D_1,\dots,D_4$ are corner perfect matchings and $D_5,\dots,D_8$ are internal ones. 

\begin{figure}[H]
\begin{center}
{\scalebox{0.8}{
\begin{tikzpicture} 
\node at (0,-1.5) {$D_1-D_5$}; \node at (3.3,-1.5) {$D_2-D_5$}; \node at (6.6,-1.5) {$D_3-D_5$}; \node at (9.9,-1.5) {$D_4-D_5$};

\node (DF1) at (0,0) 
{\scalebox{0.55}{
\begin{tikzpicture}
\basicdimerex
\draw [->, line width=0.2cm, rounded corners, color=red] (0,1)--(P1)--(P4)--(P3)--(P2)--(4,1) ; 
\end{tikzpicture} }}; 

\node (DF2) at (3.3,0) 
{\scalebox{0.55}{
\begin{tikzpicture}
\basicdimerex
\draw [->, line width=0.18cm, rounded corners, color=red] (1,0)--(P1)--(P4)--(1,4) ; 
\end{tikzpicture} }} ;  

\node (DF3) at (6.6,0) 
{\scalebox{0.55}{
\begin{tikzpicture}
\basicdimerex
\draw [<-, line width=0.18cm, rounded corners, color=red] (0,3)--(P4)--(P1)--(P2)--(P3)--(4,3) ; 
\end{tikzpicture} }}; 

\node (DF4) at (9.9,0) 
{\scalebox{0.55}{
\begin{tikzpicture}
\basicdimerex
\draw [<-, line width=0.18cm, rounded corners, color=red] (3,0)--(P2)--(P3)--(3,4) ; 
\end{tikzpicture} }} ;

\node at (14,0) 
{\scalebox{0.85}{
\begin{tikzpicture} 
\draw [step=1,thin, gray] (-1.5,-1.5) grid (1.5,1.5); 
\filldraw [fill=black] (0,0) circle [radius=0.07] ; \filldraw [fill=black] (1,0) circle [radius=0.07] ; \filldraw [fill=black] (0,1) circle [radius=0.07] ; 
\filldraw [fill=black] (-1,0) circle [radius=0.07] ; \filldraw [fill=black] (0,-1) circle [radius=0.07] ; 
\draw [line width=0.05cm] (1,0)--(0,1)--(-1,0)--(0,-1)--(1,0) ; 
\node [red] at (1.3,0.3) {\small$D_1$} ; \node [red] at (-0.3,1.3) {\small$D_2$} ;
\node [red] at (-1.3,0.3) {\small$D_3$} ; \node [red] at (-0.3,-1.3) {\small$D_4$} ; 
\node [red] at (2.2,1.3) {\small$D_5, D_6, D_7, D_8$} ;

\draw[->, line width=0.03cm,red] (1.1,1.3) to [bend right] (0,0.15);  
\end{tikzpicture} 
}} ;

\end{tikzpicture}
}}
\caption{$1$-cycles obtained as the differences of perfect matchings, 
and the PM polygon of the dimer model given in Figure~\ref{ex_quiver4a}}
\label{difference_4a}
\end{center}
\end{figure}

Next, we define the cone $\sigma$ whose section on the hyperplane $z=1$ is $\Delta$. 
That is, let $v_1,\dots,v_n\in\ZZ^2$ be vertices of $\Delta$. We add the third coordinate $z=1$ to each vector $v_i$ 
and define the cone $\sigma$ as 
\[
\sigma\coloneqq\RR_{\ge0}(v_1,1)+\cdots +\RR_{\ge0}(v_n,1)\subset\RR^3.
\]
Then, we consider the dual cone 
\[
\sigma^\vee\coloneqq\{x\in\RR^3 \mid \langle x,(v_i,1)\rangle\ge0 \text{ for any } i=1,\dots,n \}, 
\]
where $\langle - ,- \rangle$ is the natural inner product. 
This $\sigma^\vee\cap\ZZ^3$ is a positive affine normal semigroup, and hence we can define the \emph{toric ring} (\emph{toric singularity}) $R$ as 
\[
R\coloneqq \kk[\sigma^\vee\cap\ZZ^3]=\kk[t_1^{a_1}t_2^{a_2} t_3^{a_3}\mid (a_1,a_2,a_3)\in\sigma^\vee\cap\ZZ^3]. 
\] 
By construction, $R$ is Gorenstein in dimension three. 
Also, this toric singularity $R$ is isomorphic to the center $\rmZ(\calP(Q,W_Q))$ of the Jacobian algebra of $\Gamma$ (see \cite[Chapter~5]{Bro}). 
Let $\tau$ be a strongly convex rational polyhedral cone in $\RR^3$ which defines a $3$-dimensional Gorenstein toric singularity. 
Then, it is known that, after appropriate unimodular transformations, we have the lattice polygon as the intersection of $\tau$ and the hyperplane at height one. 
Thus, we can assign a certain lattice polygon $\Delta$ to each $3$-dimensional Gorenstein toric singularity. 
We call this $\Delta$ the \emph{toric diagram} of $R$. 
By Theorem~\ref{exist_dimer} below, we see that any $3$-dimensional Gorenstein toric singularity $R$ can be obtained from a consistent dimer model $\Gamma$ using the above construction, in which case $R$ is said to be the \emph{toric singularity associated with $\Gamma$}. 
However, such a dimer model is not unique in general. 

\begin{theorem}[{see e.g., \cite{Gul,IU3}}] 
\label{exist_dimer}
For any $3$-dimensional Gorenstein toric singularity $R$, there exists a consistent dimer model whose PM polygon coincides with the toric diagram of $R$. 
\end{theorem}

\subsection{Zigzag paths of dimer models} 
\label{subsec_zigzag}

To investigate a dimer model and its perfect matchings, the notion of zigzag path is also important. 
We say that a path on a dimer model is a \emph{zigzag path} if it makes a maximum turn to the right on a white node and a maximum turn to the left on a black node. 
For example, Figure~\ref{zigzag_4a} shows the zigzag paths of the dimer model given in Figure~\ref{ex_quiver4a}. 

Considering the universal cover $\RR^2\rightarrow\TT$, we define the universal cover of a dimer model $\Gamma$, which will be denoted by $\widetilde{\Gamma}$. 
Also, we can lift a zigzag path to $\widetilde{\Gamma}$. 
Note that a zigzag path on the universal cover $\widetilde{\Gamma}$ is either periodic or infinite in both directions. 

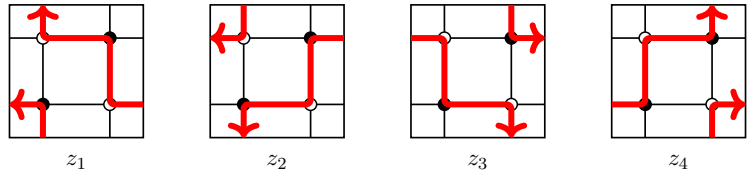
\begin{figure}[H]
\begin{center}
{\scalebox{0.8}{
\begin{tikzpicture} 
\node at (0,-1.5) {\large$z_1$}; \node at (3.3,-1.5) {\large$z_2$}; \node at (6.6,-1.5) {\large$z_3$}; \node at (9.9,-1.5) {\large$z_4$};

\node (ZZ1) at (0,0) 
{\scalebox{0.55}{
\begin{tikzpicture}
\basicdimerex
\draw [->, line width=0.18cm, rounded corners, color=red] (4,1)--(P2)--(P3)--(P4)--(1,4) ; 
\draw [->, line width=0.18cm, rounded corners, color=red] (1,0)--(P1)--(0,1) ; 
\end{tikzpicture} }}; 

\node (ZZ2) at (3.3,0) 
{\scalebox{0.55}{
\begin{tikzpicture}
\basicdimerex
\draw [->, line width=0.18cm, rounded corners, color=red] (4,3)--(P3)--(P2)--(P1)--(1,0) ; 
\draw [->, line width=0.18cm, rounded corners, color=red] (1,4)--(P4)--(0,3) ; 
\end{tikzpicture} }} ;  

\node (ZZ3) at (6.6,0) 
{\scalebox{0.55}{
\begin{tikzpicture}
\basicdimerex
\draw [->, line width=0.18cm, rounded corners, color=red] (0,3)--(P4)--(P1)--(P2)--(3,0) ; 
\draw [->, line width=0.18cm, rounded corners, color=red] (3,4)--(P3)--(4,3) ; 
\end{tikzpicture} }}; 

\node (ZZ4) at (9.9,0) 
{\scalebox{0.55}{
\begin{tikzpicture}
\basicdimerex
\draw [->, line width=0.18cm, rounded corners, color=red] (0,1)--(P1)--(P4)--(P3)--(3,4) ; 
\draw [->, line width=0.18cm, rounded corners, color=red] (3,0)--(P2)--(4,1) ; 
\end{tikzpicture} }} ;

\end{tikzpicture}
}}
\caption{Zigzag paths of the dimer model given in Figure~\ref{ex_quiver4a}}
\label{zigzag_4a}
\end{center}
\end{figure}

For a zigzag path $z$, we say that an edge of $z$ is a \emph{zig} (resp. \emph{zag}) of $z$ if $z$ passes through that edge in the direction from black to white (resp. white to black). 
Considering a zigzag path $z$ on $\Gamma$ as a $1$-cycle on $\TT$, we have the homology class $[z]\in\rmH_1(\TT)\cong\ZZ^2$. We call this element $[z]$ the \emph{slope} of $z$. 
We easily see that even if we apply the join moves and split moves to nodes contained in a zigzag path, the slope is preserved. 
Using zigzag paths, we also define another consistency condition as in \cite[Definition~3.5]{IU2}, and it is equivalent to Definition~\ref{def_consistent} (see \cite{Boc_consist,IU2}). 
In particular, we see that if $\Gamma$ is a consistent dimer model, then the slope $[z]$ of any zigzag path $z$ is not trivial, and any zigzag path does not have a self-intersection on the universal cover, in which case the slope is a primitive element. 
Considering the non-trivial slope $[z]=(a,b)\in\ZZ^2$ as the element of the unit circle, (i.e., $\frac{(a,b)}{\sqrt{a^2+b^2}}\in S^1$), we can give the cyclic order to such slopes in the anti-clockwise direction. 
We say that two zigzag paths are \emph{adjacent} if their slopes are adjacent with respect to this cyclic order. 
The following property of zigzag paths is important in the rest of this paper. 

\begin{proposition}[{see \cite[Theorem~3.3, Corollary~3.8]{Gul}, \cite[Proposition~9.2 (Step1), Corollary~9.3]{IU2}}]
\label{zigzag_sidepolygon}
Let $\Gamma$ be a consistent dimer model. 
Then, there exists a one-to-one correspondence between the set of slopes of zigzag paths on $\Gamma$ and 
the set of primitive side segments of the PM polygon of $\Gamma$. 

On the other hand, zigzag paths having the same slope arise as the difference of two corner perfect matchings that are adjacent. 
More precisely, if $z_1,\dots,z_n$ are all zigzag paths having the same slope, then there are corner perfect matchings $D_1,D_2$ 
such that they are adjacent and $D_1-D_2$ consists of $z_1,\dots,z_n$, in which case all zags $($resp. all zigs$)$ of $z_1,\dots,z_n$ are contained in $D_1$ $($resp. $D_2$$)$. 
\end{proposition}

\subsection{Crepant resolutions and NCCRs arising from dimer models}
\label{CCR_NCCR}

In this subsection, we will see that the quiver associated with a consistent dimer model gives a crepant resolution and an NCCR 
for a $3$-dimensional Gorenstein toric singularity. 

Let $Q$ be the quiver associated with a consistent dimer model. In particular, $Q$ is finite and connected. 
Also, this equips relations $p_a^+=p_a^-$ for any $a\in Q_1$ introduced in Subsection~\ref{subsec_dimer}, and we denote these relations by $\calJ_Q$. 
In the following, we denote this quiver with relations by $(Q,\calJ_Q)$. 
A \emph{representation of quiver} $(Q,\calJ_Q)$ consists of a $\kk$-vector space $V_i$ for each $i\in Q_0$ and 
a $\kk$-linear map $\varphi_a:V_{\tl(a)}\rightarrow V_{\hd(a)}$ for each $a\in Q_1$ that satisfies the relation corresponding to $\calJ_Q$, 
that is, $\varphi_{p_a^+}=\varphi_{p_a^-}$ for any $a\in Q_1$. 
Here, for a path $p=a_1\cdots a_r$, the map $\varphi_p$ denotes the composite $\varphi_{a_1}\cdots\varphi_{a_r}$ of $\kk$-linear maps. 
A representation $V=((V_i)_{i\in Q_0}, (\varphi_a)_{a\in Q_1})$ of $(Q,\calJ_Q)$ is called \emph{finite dimensional} 
if $\dim_\kk V_i$ is finite for all $i\in Q_0$. 
For a finite dimensional representation $V$, we define the \emph{dimension vector} of $V$ as ${\dd}\coloneqq(\dim_\kk V_i)_{i\in Q_0}$. 
We say that a representation $V$ is \emph{nilpotent} if there is $n\in\ZZ_{\ge0}$ such that any path with length greater than $n$ acts on $V$ by zero. 
Let $V,V^\prime$ be representations of $(Q,\calJ_Q)$. 
A morphism between $V$ and $V^\prime$ is a family of $\kk$-linear maps $\{f_i:V_i\rightarrow V_i^\prime\}_{i\in Q_0}$ such that 
$\varphi_af_{\hd(a)}=f_{\tl(a)}\varphi_a^\prime$ for any arrow $a:i\rightarrow j$. 
We say that representations $V$ and $V^\prime$ are \emph{isomorphic} if $f_i$ is an isomorphism of vector space for all $i\in Q_0$, 
in which case we denote $V\cong V^\prime$. 

In the rest, we consider a representation $V=((V_i)_{i\in Q_0}, (\varphi_a)_{a\in Q_1})$ with $\dd=(1,\dots,1)$. 
Taking a basis of each $1$-dimensional vector space $V_i$, we identify $V_i\cong\kk$. The representation space 
consists of $\kk$-valued points of 
$$
\VV(J_Q)=\Spec(\kk[x_a \mid a\in Q_1]/J_Q)
$$
where $J_Q\coloneqq\langle\,\prod_{a\in p_a^+}x_a- \prod_{a\in p_a^-}x_a \mid a\in Q_1\,\rangle$. 
The changes of basis give isomorphism classes of $V$, 
thus we can consider the isomorphism classes of representations as the orbits by the action of $(\kk^\times)^{Q_0}\cong\prod_{i\in Q_0}\GL(V_i)$. 
Here, an element $g=(g_i)_{i\in Q_0}$ acts on $\varphi=(\varphi_a)_{a\in Q_1}$ as $(g\cdot \varphi)_a=g_{\tl(a)}^{-1}\varphi_ag_{\hd(a)}$. 

\medskip

Since an element in the diagonal subgroup $\kk^\times$, which takes the form $(\alpha,\dots,\alpha)\in(\kk^\times)^{Q_0}$ 
acts trivially, $G\coloneqq (\kk^\times)^{Q_0}/\kk^\times$ acts faithfully on $\VV(J_Q)$. 
We then define a weight $\theta=(\theta_i)_{i\in Q_0}\in\ZZ^{Q_0}$. 
This $\theta$ induces the character 
\[
\widetilde{\chi_\theta}:(\kk^\times)^{Q_0}\rightarrow\kk^\times \quad \Big((g_i)_{i\in Q_0}\mapsto\prod_{i\in Q_0}g_i^{\theta_i}\Big),
\] 
and the character $\widetilde{\chi_\theta}$ descends to the character $\chi_\theta:G\rightarrow\kk^\times$ under 
the condition $\sum_{i\in Q_0}\theta_i=0$. 
In what follows, we consider the weight space 
$\Theta\coloneqq\{\theta\in\ZZ^{Q_0} \mid \sum_{i\in Q_0}\theta_i=0\}$ 
and let $\Theta_\RR\coloneqq\Theta\otimes_\ZZ\RR$. 
We call an element $\theta\in\Theta_\RR$ a \emph{stability parameter}. 
For a subrepresentation $W$ of $V$, we define $\theta(W)\coloneqq\sum_{i\in Q_0}\theta_i(\dim_\kk W_i)$, and hence $\theta(V)=0$ in particular. 
Using a stability parameter $\theta\in\Theta_\RR$, we introduce the notion of \emph{stability conditions} as follows. 

\begin{definition}[{see \cite{Kin}}]
For $\theta\in\Theta_\RR$, we say that a representation  
$V$ is \emph{$\theta$-semistable} if $\theta(W)\ge0$ for any subrepresentation $W$ of $V$, and 
$V$ is \emph{$\theta$-stable} if $\theta(W)>0$ for any non-zero proper subrepresentation $W$ of $V$. 
In addition, we say that $\theta$ is \emph{generic} if every $\theta$-semistable representation is $\theta$-stable. 
\end{definition}

For a character $\chi_\theta$, we say that $f\in S\coloneqq\kk[x_a \mid a\in Q_1]/J_Q$ is a \emph{$\chi_\theta$-semi-invariant} if $f(g\cdot x)=\chi_\theta(g)f(x)$ for any $g\in G$. 
Let $S_{\chi_\theta}$ be the vector space of all $\chi_\theta$-semi-invariants. 
Then, we define the \emph{GIT quotient} with respect to $\chi_\theta$ as 
\[
\calM_\theta(Q)\coloneqq \Proj\Big(\bigoplus_{j\ge0}S_{\chi_\theta^j}\Big). 
\] 
For the trivial character which is given by ${\bf 0}=(0,\dots,0)$, we see that any representation is $\theta$-semistable and 
$\calM_{\bf 0}(Q)$ is the categorical quotient $\Spec S^G$. We remark that $S^G$ is isomorphic to the $3$-dimensional Gorenstein toric singularity constructed in Subsection~\ref{subsec_pm} (see e.g., \cite[Proposition~6.3]{IU1}).
By \cite[Proposition~5.3]{Kin}, for a generic parameter $\theta$, $\calM_\theta(Q)$ is a fine moduli space parametrizing $\theta$-stable representations of $(Q, \calJ_Q)$. 
Moreover, this moduli space gives a crepant resolution. 

\begin{theorem}[{see \cite[Theorem~6.3 and 6.4]{IU1}}] 
\label{thm_crepant_dimer}
Let $\Gamma$ be a consistent dimer model, and $Q$ be the associated quiver. 
Let $R$ be the $3$-dimensional Gorenstein toric singularity associated with $\Gamma$. 
Then, $\calM_\theta(Q)\rightarrow\Spec R$ is a crepant resolution for a generic parameter $\theta$. 
\end{theorem}

\begin{remark}
\label{rem_crepant_resol}
We note that for a generic parameter $\theta$, $\calM_\theta(Q)$ is a smooth toric variety, thus this can be realized by using a smooth toric fan. 
More precisely, there is a certain smooth subdivision $\Sigma$ of the cone $\sigma$ defining $R=\kk[\sigma^\vee\cap\ZZ^3]$ such that the smooth toric variety $X_\Sigma$ associated to $\Sigma$ is isomorphic to $\calM_\theta(Q)$ (see e.g., \cite[Chapter~11]{CLS}). 
On the other hand, there is a certain way to correspond $\theta$-stable representations to cones in $\Sigma$  (see Subsection~\ref{subsec_cone_pm}). 
\end{remark}

Since $\calM_\theta(Q)$ is a fine moduli space for a generic parameter $\theta$, there is the universal vector bundle $\calE$, which is called \emph{tautological bundle}. 
This $\calE$ is a tilting bundle on $\calM_\theta(Q)$ and satisfies $\calP(Q,W_Q)\cong\End(\calE)$ (see \cite[Theorem~1.4]{IU2}), 
and hence we have an equivalence $\calD^\rmb(\coh\,\calM_\theta(Q))\cong\calD^\rmb(\mc\,\calP(Q,W_Q))$, see \cite{Bon,Ric}. 
Furthermore, this equivalence also implies that $\calP(Q,W_Q)$ is an NCCR for a $3$-dimensional Gorenstein toric singularity $R$ (see \cite[Corollary~4.15]{IW2}). 
This was also proved in \cite{Bro} using another method. 
Since any $3$-dimensional Gorenstein toric singularity can be constructed from a consistent dimer model (see Subsection~\ref{subsec_pm}), it admits an NCCR arising from a consistent dimer model.

\subsection{Perfect matchings corresponding to torus orbits}
\label{subsec_cone_pm}

Let $R$ be a $3$-dimensional Gorenstein toric singularity associated with a consistent dimer model $\Gamma$, 
in which case the toric diagram of $R$ coincides with the PM polygon $\Delta$ of $\Gamma$ (see Theorem~\ref{exist_dimer}). 
As we saw in Section~\ref{CCR_NCCR}, we can construct a crepant resolution $\calM_\theta(Q)$ of $\Spec R$, 
and $\calM_\theta(Q)$ parametrizes $\theta$-stable representations. 
We denote the $\theta$-stable representation corresponding to a point $y\in\calM_\theta(Q)$ by $V_y$. 

On the other hand, considering the intersection of the fan $\Sigma$ of $\calM_\theta(Q)$ (see Remark~\ref{rem_crepant_resol}) and the hyperplane at height one, we have the triangulation of $\Delta$. 
That is, we identify 
\begin{itemize}
\item one-dimensional cones (= \emph{rays}) in $\Sigma$ with lattice points in the triangulation of $\Delta$, 
\item two-dimensional cones in $\Sigma$ with line segments in the triangulation of $\Delta$, 
\item three-dimensional cones in $\Sigma$ with triangles in the triangulation of $\Delta$. 
\end{itemize}
We denote the set of $r$-dimensional cones in $\Sigma$ by $\Sigma(r)$ where $r=1,2,3$. 
By the Orbit-Cone correspondence (see e.g., \cite[Chapter~3]{CLS}), an $r$-dimensional cone $\sigma\in\Sigma(r)$ corresponds to a $3-r$ dimensional 
torus orbit in $X_\Sigma\cong\calM_\theta(Q)$, which we will denote by $\calO_\sigma$.  
Thus, we can assign $\theta$-stable representations to cones in the fan $\Sigma$ of $\calM_\theta(Q)$ (and hence to torus orbits). 
In particular, we will see in Proposition~\ref{corresp_pm} that each ray (and hence each lattice point in $\Delta$) corresponds to a $\theta$-stable representation arising from a perfect matching. 
In order to describe this correspondence, we define the \emph{support} of a representation $V=((V_i)_{i\in Q_0}, (\varphi_a)_{a\in Q_1})$, denoted by $\mathrm{Supp}V$, as the set $\{a\in Q_1 \mid \varphi_a\neq0\}$ of arrows whose corresponding linear maps are not zero. 
We also define the \emph{cosupport} of a representation $V$ as the complement of $\mathrm{Supp}V$. 

A perfect matching $\sfD$ of $Q$ is called \emph{$\theta$-stable} if the set of arrows in $\sfD$ is the cosupport of a $\theta$-stable representation. 
By the construction of the PM polygon $\Delta$ (see Subsection~\ref{subsec_pm}), 
each perfect matching of $\Gamma$ corresponds to a lattice point in $\Delta$. 
Then, we can obtain the following relationships between perfect matchings and $\theta$-stable representations. 

\begin{proposition}[{see \cite[Section~6]{IU1},\cite[Proposition~4.15]{Moz}}]
\label{corresp_pm}
Let the notation be the same as above. 
\begin{enumerate}[\rm (1)]
\item For a generic parameter $\theta$, we consider a two dimensional torus orbit $Z$ $($i.e., $Z=\calO_\sigma$ with $\sigma\in\Sigma(1)$$)$ of $\calM_\theta(Q)$. 
For any $y\in Z$, the cosupport of the $\theta$-stable representation $V_y$ is a perfect matching of $Q$, and this perfect matching does not depends on a choice of $y\in Z$. 
\item Let $V_\sfD$ be a representation of $Q$ whose cossuport is a perfect matching $\sfD$ of $Q$. 
Then, there exists a generic parameter $\theta$ such that $V_\sfD$ is $\theta$-stable. 
\end{enumerate}
\end{proposition}

Here, we remark that for any interior lattice point $p$ of $\Delta$, there is more than one perfect matching corresponding to $p$. 
(We will observe the relationship among those perfect matchings in Section~\ref{sec_mutation_pm}.) 
However, for a given generic parameter $\theta$, only one of those perfect matchings is $\theta$-stable. 
In other words, for a fixed generic parameter $\theta$, we have a bijection between lattice points of $\Delta$ and $\theta$-stable perfect matchings. 

\begin{remark}[{see \cite[the last part of Section~4]{Moz}}]
\label{method_make_triangulation}
Let $\theta$ be a generic parameter. 
There is a certain way to find the triangulation of $\Delta$ whose associated toric variety $X_\Sigma$ is isomorphic to $\calM_\theta(Q)$. 
To do this, we first assign a $\theta$-stable perfect matching to each lattice point of $\Delta$. 
Then, for any pair of $\theta$-stable perfect matchings $(\sfD,\sfD^\prime)$ of $Q$, we check whether $\sfD\cup \sfD^\prime$ is the cosupport of a $\theta$-stable representation or not. 
If it is so, we connect lattice points corresponding to $\sfD$ and $\sfD^\prime$, thus we have the line segment. 
Repeating these arguments, we have the triangulation of $\Delta$, and it gives the fan $\Sigma$ whose associated toric variety $X_\Sigma$ is isomorphic to $\calM_\theta(Q)$. 
\end{remark}

\section{\bf Perfect matchings giving $2$-representation infinite algebras} 
\label{section_2RI}

\subsection{$n$-representation infinite algebras} 
\label{subsec_nrepinfinite}

Let $n\ge 1$ be a positive integer. 
In this paper, the notion of $n$-representation infinite algebras plays an important role (especially the case of $n=2$). 
This is a certain generalization of representation infinite hereditary algebras. 
In order to define this algebra, we first introduce the Nakayama functor. 
Let $\Lambda$ be a finite dimensional $\kk$-algebra of $\gldim\Lambda\le n$. We define functors: 
$$\nu\coloneqq\rmD\RHom_\Lambda(-,\Lambda):\calD^\rmb(\mc\Lambda)\rightarrow\calD^\rmb(\mc\Lambda),$$
$$\nu^-\coloneqq\RHom_{\Lambda^{\rm op}}(\rmD-,\Lambda):\calD^\rmb(\mc\Lambda)\rightarrow\calD^\rmb(\mc\Lambda).$$
The functor $\nu$ is called the \emph{Nakayama functor}, and it is known that $\nu$ is the Serre functor of $\calD^\rmb(\mc\Lambda)$ (see \cite{BK}). 
We then define the autoequivalences 
$$
\nu_n\coloneqq\nu\circ[-n]:\calD^\rmb(\mc\Lambda)\rightarrow\calD^\rmb(\mc\Lambda), 
$$
$$
\nu_n^-\coloneqq\nu^-\circ[n]:\calD^\rmb(\mc\Lambda)\rightarrow\calD^\rmb(\mc\Lambda). 
$$

\begin{definition}[{see \cite[Definition~2.7 and Proposition~2.9]{HIO}}]
Let $n\ge 1$ be a positive integer, and $\Lambda$ be a finite dimensional algebra with $\gldim\Lambda\le n$. 
We say that $\Lambda$ is \emph{$n$-representation infinite} if $\nu_n^{-i}(\Lambda)\in\mc\Lambda$ for all $i\ge0$ (this means $\nu_n^{-i}(\Lambda)$ is quasi-isomorphic to a complex concentrated in the degree zero part), 
in which case its global dimension is precisely $n$, because $\Ext^n_\Lambda(\rmD\Lambda,\Lambda)=\nu_n^{-1}(\Lambda)\neq 0$. 
\end{definition}

If $n=1$, then $1$-representation infinite algebras coincide with representation infinite hereditary algebras. 
In the context of non-commutative algebraic geometry, $n$-representation infinite algebras are called \emph{quasi $n$-Fano algebras} and investigated in \cite{Min,MM}. In particular, the following results also hold in such a context using different methods. 

One of the methods to obtain a $n$-representation infinite algebra is to consider a bimodule $(n+1)$-Calabi-Yau algebra of Gorenstein parameter $1$ whose degree zero part is finite dimensional as we will see below. 
Here, we say that a graded algebra $B=\oplus_{i\ge 0}B_i$ with $\dim_\kk B_i<\infty$ for any $i\ge 0$ is \emph{bimodule $n$-Calabi-Yau of Gorenstein parameter} $1$ if $B\in\per B^e$ and 
$\RHom_{B^e}(B,B^e)[n](-1)\cong B$ in $\calD(\Mod^\ZZ B^e)$. 

\begin{theorem}[{see \cite[Corollary~3.6]{AIR}, \cite[Theorem~4.12]{MM}}] 
\label{degzero_RI}
Let $B=\oplus_{i\ge 0}B_i$ be a positively $\ZZ$-graded algebra that is bimodule $(n+1)$-Calabi-Yau of Gorenstein parameter $1$. 
If $\dim_\kk B_0$ is finite, then $B_0$ is an $n$-representation infinite algebra. 
\end{theorem}

On the other hand, a bimodule $(n+1)$-Calabi-Yau algebra of Gorenstein parameter $1$ can be obtained as the \emph{$(n+1)$-preprojective algebra} 
of an $n$-representation infinite algebra (see e.g., \cite[Theorem~4.36]{HIO}, \cite[Theorem~4.8]{Kel}). 

\subsection{$2$-representation infinite algebras arising from dimer models} 
\label{subsec_2rep_dimer}

In this subsection, we construct $2$-representation infinite algebras using consistent dimer models and their perfect matchings. 
In this construction, the degree $d_\sfD$ associated with a perfect matching $\sfD$ defined in (\ref{degree_pm}) plays a crucial role. 

\begin{proposition}[{see \cite[Proposition~6.1]{AIR}}]
\label{dimer_GP1}
Let $(Q,W_Q)$ be the QP associated with a consistent dimer model, and $\calP(Q,W_Q)$ be the Jacobian algebra. 
Then, for any perfect matching $\sfD$ of $Q$, the graded Jacobian algebra $\calP(Q,W_Q)$ with respect to $d_\sfD$ is a bimodule $3$-Calabi-Yau algebra of Gorenstein parameter $1$. 
\end{proposition}

Next, we consider the \emph{truncated Jacobian algebra} $\calA_\sfD\coloneqq\calP(Q,W_Q)_\sfD$, 
which is the degree zero part of the graded Jacobian algebra $\calP(Q,W_Q)$ with respect to $d_\sfD$. 
In particular, we can describe this algebra as 
$$
\calA_\sfD=\calP(Q,W_Q)/\langle\sfD\rangle=\kk Q_\sfD/\langle\partial_aW_Q \mid a\in \sfD\rangle. 
$$
Here, we recall that $Q_\sfD$ is the subquiver of $Q$ whose set of vertices is the same as $Q$ and set of arrows is $Q_1{\setminus} \sfD$, 
especially $Q_\sfD$ is finite and connected. 
By Theorem~\ref{degzero_RI} and Proposition~\ref{dimer_GP1}, $\calA_\sfD$ will be a $2$-representation infinite algebra if it is finite dimensional. 
However, the degree $d_\sfD$ does not necessarily make $\calA_\sfD$ finite dimensional. 
Thus, in the rest we will find perfect matchings making $\calA_\sfD$ finite dimensional. 
Let $R$ be a $3$-dimensional Gorenstein toric singularity associated with a consistent dimer model $\Gamma$, 
and $Q$ be the quiver associated with $\Gamma$. 
By Theorem~\ref{thm_crepant_dimer}, we have a crepant resolution $\pi_\theta :\calM_\theta(Q)\rightarrow\Spec R$ for a generic parameter $\theta$. 
Since $\calM_\theta(Q)$ is a fine moduli space parametrizing $\theta$-stable representations, 
we have a $\theta$-stable representation $V_y$ parametrized by a point $y\in\calM_\theta(Q)$. 
Then, we have the following lemma, which is well-known for experts. 

\begin{lemma}
\label{condition_nilpotent}
Let $V_y=((V_i)_{i\in Q_0}, (\varphi_a)_{a\in Q_1})$ be the $\theta$-stable representation parametrized by $y\in\calM_\theta(Q)$. 
For the unique torus invariant point $x_0\in\Spec R$, we see that $y\in\calM_\theta(Q)$ lies in $\pi^{-1}_\theta(x_0)$ if and only if $V_y$ is a nilpotent representaion. 
\end{lemma}

\begin{proof}
Since $R\cong e_i\calP(Q,W_Q)e_i$ for any $i\in Q_0$ (see e.g., \cite[Chapter~5]{Bro}), we see that $\pi_\theta(y)=x_0$ if and only if for any nontrivial element in 
$e_i\calP(Q,W_Q)e_i$ the corresponding cycle $\omega$ satisfies $\varphi_\omega=0$. 
Also, since $Q$ is a finite connected quiver, a path of length at least $|Q_0|$ contains a cycle. 
Therefore, if $\pi_\theta(y)=x_0$ then $V_y$ is nilpotent. On the other hand, if $\pi_\theta(y)\neq x_0$, then there exists a nontrivial cycle $c$ such that 
$\varphi_c\neq 0$. Thus, the action of $c^n$ is not zero for any $n\ge 1$, and hence $V_y$ is not nilpotent. 
\end{proof}

Then, we can obtain the following characterization, which was also discussed in \cite[Lemma~1.44]{Boc_ABC}, and this plays an important role in this paper. 

\begin{proposition}
\label{findim_internal}
Let $Q$ be the quiver associated with a consistent dimer model $\Gamma$. 
Then, for a perfect matching $\sfD$ of $Q$, the following conditions are equivalent. 
\begin{enumerate}[\rm (1)]
\item $\sfD$ is an internal perfect matching. 
\item $Q_\sfD$ is an acyclic quiver. 
\item $\calA_\sfD$ is a finite dimensional algebra. 
\end{enumerate}
When this is the case, $\calA_\sfD$ is a $2$-representation infinite algebra. 
\end{proposition}

\begin{proof}
Fix a generic parameter $\theta$ such that $\sfD$ is $\theta$-stable (see Proposition~\ref{corresp_pm}(2)).  
Then, there is a point $y\in\calM_\theta(Q)$ such that the cosupport of $\theta$-stable representation $V_y$ coincides with $\sfD$ by Proposition~\ref{corresp_pm}(1), 
in which case $y$ lies in a toric divisor. 
Therefore, $\sfD$ is internal if and only if $y$ lies in an exceptional one. 
This is equivalent to the condition that $V_y$ is nilpotent by Lemma~\ref{condition_nilpotent}. 
Since the dimension vector of $V_y$ is $(1,\dots,1)$, each non-zero map is given by an element in $\kk^\times$. 
Since $Q_\sfD$ is a finite connected quiver, $Q_\sfD$ is acyclic if and only if $V_y$ is nilpotent, and hence we have $(1)\Leftrightarrow(2)$. 
$(2)\Rightarrow(3)$ is trivial. 
On the other hand, we assume that $Q_\sfD$ contains a cycle. 
Then $\calA_\sfD$ is not finite dimensional, since the relations on $Q$ only make two paths equal, and hence they do not make a path to be zero. 
Thus we have $(3)\Rightarrow(2)$. 

By Theorem~\ref{degzero_RI} and Proposition~\ref{dimer_GP1}, 
$\calA_\sfD$ is a $2$-representation infinite algebra if $\sfD$ is internal. 
\end{proof}

An $n$-representation infinite algebra is called \emph{$n$-representation tame} if its $(n+1)$-preprojective algebra $\Lambda$ is a Noetherian $A$-algebra, that is, 
$\Lambda$ is a finitely generated module over a commutative ring $A$ (see \cite[Definition~6.10]{HIO}). 
In our situation, the $3$-preprojective algebra of a $2$-representation infinite algebra $\calA_\sfD=\calP(Q,W_Q)_\sfD$ associated with an internal perfect matching $\sfD$ is the Jacobian algebra $\calP(Q,W_Q)$ (see e.g., \cite[Theorem~4.36]{HIO}). 
Since $\calP(Q,W_Q)$ is an NCCR of a $3$-dimensional Gorenstein toric singularity $R=\rmZ(\calP(Q,W_Q))$, there exists a reflexive $R$-module $M$ such that $\calP(Q,W_Q)\cong\End_R(M)$, and hence this is finitely generated as an $R$-module. 
As a conclusion, $\calA_\sfD$ is $2$-representation tame.  

\section{\bf Stable categories of (graded) MCM modules and tilting objects}
\label{sec_cluster_eq}

One of the purposes in this section is to show equivalences as in Theorem~\ref{motivation_thm2} 
for a $3$-dimensional Gorenstein isolated singularity $R$ that is not the $A_1$-singularity (see Corollary~\ref{main_cor}). 
Here, we recall sufficient conditions for giving such equivalences: 
\begin{itemize}
\item[\rm (P1)] the truncated Jacobian algebra $\calA_\sfD$ is finite dimensional, 
\item[\rm (P2)] $i$ is a source of the quiver $Q_\sfD$, 
\item[\rm (P3)] the algebra $\calP(Q,W_Q)/\langle e_i\rangle$ is finite dimensional. 
\end{itemize}
By Proposition~\ref{findim_internal}, if $\sfD$ is an internal perfect matching, then $\calA_\sfD$ satisfies the condition (P1). 
Thus, in what follows, we discuss other conditions. 

\begin{lemma}
\label{isolated_condition}
Let $(Q,W_Q)$ be the QP associated with a consistent dimer model $\Gamma$. 
Then, the following conditions are equivalent. 
\begin{enumerate}[\rm(1)] 
\item The $3$-dimensional Gorenstein toric singularity $R=\rmZ(\calP(Q,W_Q))$ is isolated. 
\item $\calP(Q,W_Q)/\langle e_i\rangle$ is finite dimensional for any vertex $i\in Q_0$. 
\item $\calP(Q,W_Q)/\langle e_i\rangle$ is finite dimensional for some vertex $i\in Q_0$. 
\item Any edge of the PM polygon $\Delta$ does not have an interior lattice point. 
\end{enumerate}
\end{lemma}

\begin{proof}
First, $(1)\Rightarrow(2)$ follows from \cite[Lemma~6.19(3) and Remark~6.15]{IW}, and $(2)\Rightarrow(3)$ is trivial. 

Let $E$ be an edge of $\Delta$.  To show $(3)\Rightarrow(4)$, we assume that the number of lattice line segments of $E$ is $r\ge 2$. 
By Proposition~\ref{zigzag_sidepolygon}, there exist zigzag paths $z_1,\dots,z_r$ corresponding to lattice line segments of $E$, and they satisfy $[z_1]=\cdots=[z_r]$. 
Since $r\ge 2$, we consider two of them and denote by $z,z^\prime$. 
Then, these $z$ and $z^\prime$ do not share a common node (see the conditions of a \emph{properly ordered dimer model} defined in \cite[Section~3.1]{Gul}, which is equivalent to Definition~\ref{def_consistent}). 
We then consider the paths $p_z,p_{z^\prime}$ on $Q$ going along the left side of $z, z^\prime$ respectively, and these are cycles on $\TT$ in particular. 
Also, we see that $p_z$ and $p_{z^\prime}$ do not factor through a common vertex of $Q$. 
Thus, even if we divide $\calP(Q,W_Q)$ by the ideal $\langle e_i\rangle$, the underlying quiver contains at least one of these cycles $p_z,p_{z^\prime}$, and hence $\calP(Q,W_Q)/\langle e_i\rangle$ is not finite dimensional for any $i\in Q_0$. 

Next, we show $(4)\Rightarrow(1)$. 
As we mentioned, the toric singularity $R$ is constructed from the strongly convex rational polyhedral cone $\sigma$ in $\RR^3$ defined by putting the PM polygon $\Delta$ on the hyperplane at height one. 
Since $R$ is normal, it is regular in codimension one. 
Thus, we consider the toric ring $R_\tau=\kk[\tau^\vee\cap\ZZ^3]$ associated to a facet $\tau$ of $\sigma$. 
If we assume the condition $(4)$, then $\tau$ is a smooth cone and hence $R_\tau$ is regular. Therefore, we have the assertion. 
\end{proof}

\begin{lemma}
\label{interior_isolated}
Let $R$ be a non-regular $3$-dimensional Gorenstein toric isolated singularity. 
If $R$ is not the $A_1$-singularity $($i.e., $R\not\cong\kk[x,y,z,w]/(xy-zw)$$)$, then the toric diagram of $R$ contains an interior lattice point. 
\end{lemma}

\begin{proof}
Let $\Delta$ be the toric diagram of $R$. Since $R$ is an isolated singularity, $\Delta$ satisfies the condition given in Lemma~\ref{isolated_condition}(4). 
Then, we see that $\Delta$ does not contain an interior lattice point if and only if $\Delta$ is unimodular equivalent to 
\raisebox{-0.2cm}{\scalebox{0.5}{\begin{tikzpicture} \coordinate (00) at (0,0); \coordinate (10) at (1,0); \coordinate (01) at (0,1); 
\draw [line width=0.05cm, fill=black] (00) circle [radius=0.08] ; \draw [line width=0.05cm, fill=black] (10) circle [radius=0.08] ; \draw [line width=0.05cm, fill=black] (01) circle [radius=0.08] ;
\draw [line width=0.05cm] (00)--(10)--(01)--(00); \end{tikzpicture}}} 
or 
\raisebox{-0.2cm}{\scalebox{0.5}{\begin{tikzpicture} \coordinate (11) at (1,1); 
\draw [line width=0.05cm, fill=black] (00) circle [radius=0.08] ; \draw [line width=0.05cm, fill=black] (10) circle [radius=0.08] ; \draw [line width=0.05cm, fill=black] (11) circle [radius=0.08] ;  \draw [line width=0.05cm, fill=black] (01) circle [radius=0.08] ;
\draw [line width=0.05cm] (00)--(10)--(11)--(01)--(00); \end{tikzpicture}}} (see e.g., \cite[Theorem~1]{Rab}). 
The former case gives a regular toric ring, and we easily see that the later one gives the $A_1$-singularity $\kk[x,y,z,w]/(xy-zw)$. 

As a conclusion, we see that a non-regular $3$-dimensional Gorenstein toric isolated singularity not having an interior lattice point is isomorphic to the $A_1$-singularity. 
\end{proof}

\begin{corollary}
\label{main_cor}
Let $R$ be a non-regular $3$-dimensional Gorenstein toric isolated singularity that is not the $A_1$-singularity. 
Let $\Gamma$ be a consistent dimer model associated with $R$, and $(Q,W_Q)$ be the QP obtained as the dual of $\Gamma$. 
Let $\sfD$ be an internal perfect matching of $Q$, and let $\calA_\sfD=\calP(Q,W_Q)_\sfD$. 
Then, there exists a triangle equivalences\,$:$ 
$$
\begin{tikzcd}
  \calD^\rmb(\mc\calA_{\sfD,e_i}) \arrow[r, "\cong"] \arrow[d] & \sCM^\ZZ(R) \arrow[d] \\
  \calC_2(\calA_{\sfD,e_i}) \arrow[r,  "\cong" ] &\sCM(R)
\end{tikzcd}
$$
where $\calA_{\sfD,e_i}\coloneqq \calA_\sfD/\langle e_i\rangle$ and $i\in Q_0$ is a source of $Q_\sfD$. 
\end{corollary}

\begin{proof}
By Theorem~\ref{motivation_thm2}, we only check that there actually exists a perfect matching $\sfD$ and a vertex $i\in Q_0$ satisfying the conditions (P1)--(P3). 

Let $\Delta$ be the toric diagram of $R$, which contains an interior lattice point (see Lemma~\ref{interior_isolated}). 
Then, for each interior lattice point $p\in\Delta$, there exists an internal perfect matching corresponding to $p$. 
Indeed, taking a generic parameter $\theta$, we have a crepant resolution $\calM_\theta(Q)$ of $\Spec R$ (see Theorem~\ref{thm_crepant_dimer}). 
For the two dimensional torus orbit $Z_p$ in $\calM_\theta(Q)$ that corresponds to the lattice point $p$ (see Subsection~\ref{subsec_cone_pm}), 
we have a $\theta$-stable representation contained in $Z_p$, 
and the cosupport of this representation is our desired perfect matching of $Q$ (see Proposition~\ref{corresp_pm}). 
We denote this perfect matching by $\sfD$. 
Then, by Proposition~\ref{findim_internal}, we have the acyclic quiver $Q_\sfD$ and the finite dimensional algebra $\calA_\sfD$. 
For the idempotent $e_i$ corresponding to a source $i$ of $Q_\sfD$, we see that $\calP(Q,W_Q)/\langle e_i\rangle$ is finite dimensional by Lemma~\ref{isolated_condition}. 
\end{proof}

\begin{remark}
By Corollary~\ref{main_cor}, for any internal perfect matching $\sfD$, the generalized cluster category $\calC_2(\calA_{\sfD,e_i})$ is triangle equivalent to $\sCM(R)$. Thus, for internal perfect matchings $\sfD,\sfD^\prime$, the generalized cluster categories $\calC_2(\calA_{\sfD,e_i})$ and $\calC_2(\calA_{\sfD^\prime,e_j})$ are triangle equivalent where $i\in Q_0$ (resp. $j\in Q_0$) is a source of $Q_\sfD$ (resp. $Q_{\sfD^\prime}$). 
In other words, they are \emph{cluster equivalent} in the sense of \cite{AO}. 
However, as shown in \cite[5.4]{AIR}, $\calD^\rmb(\mc\calA_{\sfD,e_i})$ and $\calD^\rmb(\mc\calA_{\sfD^\prime,e_j})$ are not derived equivalent in general. 
\end{remark}

Even if $R$ is the $A_1$-singularity, we have similar equivalences as follows. 

\begin{example}
\label{rem_conifold}
Let $R=\kk[x,y,z,w]/(xy-zw)$. A consistent dimer model associated with $R$ and the associated quiver $Q$ (with potential) takes the following form. 

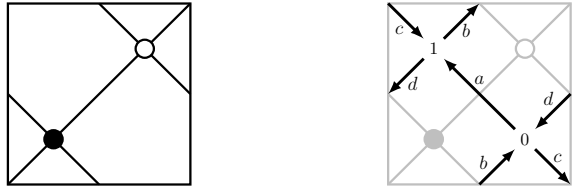
\begin{center}
\begin{tikzpicture}

\node at (0,0) {
\scalebox{0.6}{
\begin{tikzpicture}
\coordinate (B1) at (1,1); \coordinate (W1) at (3,3); 
\draw[line width=0.05cm] (0,0) rectangle (4,4);
\draw[line width=0.05cm]  (B1)--(W1); \draw[line width=0.05cm]  (B1)--(0,0);  \draw[line width=0.05cm]  (B1)--(2,0);  \draw[line width=0.05cm]  (B1)--(0,2); 
\draw[line width=0.05cm]  (W1)--(4,4);  \draw[line width=0.05cm]  (W1)--(2,4);  \draw[line width=0.05cm]  (W1)--(4,2); 
\draw [line width=0.05cm, fill=black] (B1) circle [radius=0.2] ;
\draw  [line width=0.05cm,fill=white] (W1) circle [radius=0.2] ;
\end{tikzpicture} }}; 

\node at (5,0) {
\scalebox{0.6}{
\begin{tikzpicture}[sarrow/.style={black, -latex}]
\coordinate (B1) at (1,1); \coordinate (W1) at (3,3); 
\draw[line width=0.05cm, lightgray] (0,0) rectangle (4,4);
\draw[line width=0.05cm, lightgray] (B1)--(W1); \draw[line width=0.05cm, lightgray] (B1)--(0,0);  \draw[line width=0.05cm, lightgray] (B1)--(2,0);  
\draw[line width=0.05cm, lightgray] (B1)--(0,2); \draw[line width=0.05cm, lightgray] (W1)--(4,4);  \draw[line width=0.05cm, lightgray] (W1)--(2,4);  
\draw[line width=0.05cm, lightgray] (W1)--(4,2); 
\draw [line width=0.05cm, lightgray, fill=lightgray] (B1) circle [radius=0.2] ;
\draw  [line width=0.05cm, lightgray, fill=white] (W1) circle [radius=0.2] ; 
\node (V0) at (3,1) {\large$0$}; \node (V1) at (1,3) {\large$1$}; 
\draw[sarrow, line width=0.064cm] (V0)--(V1) node[midway,xshift=0cm,yshift=0.3cm] {\Large $a$}; 
\draw[sarrow, line width=0.064cm] (V1)--(2,4) node[midway,xshift=0.1cm,yshift=-0.2cm] {\Large $b$}; 
\draw[sarrow, line width=0.064cm] (2,0)--(V0) node[midway,xshift=-0.3cm,yshift=0.1cm] {\Large $b$}; 
\draw[sarrow, line width=0.064cm] (V0)--(4,0) node[midway,xshift=0.1cm,yshift=0.2cm] {\Large $c$}; 
\draw[sarrow, line width=0.064cm] (0,4)--(V1) node[midway,xshift=-0.15cm,yshift=-0.2cm] {\Large $c$}; 
\draw[sarrow, line width=0.064cm] (V1)--(0,2) node[midway,xshift=0.15cm,yshift=-0.15cm] {\Large $d$}; 
\draw[sarrow, line width=0.064cm] (4,2)--(V0) node[midway,xshift=-0.1cm,yshift=0.3cm] {\Large $d$}; 
\end{tikzpicture} }}; 

\end{tikzpicture}
\end{center}
Indeed, the center of the Jacobian algebra, which is generated by cycles $x\coloneqq ab, y\coloneqq cd, z\coloneqq cb, w\coloneqq ad$, is isomorphic to the $A_1$-singularity $R=\kk[x,y,z,w]/(xy-zw)\cong \kk[ab,cd,cb,ad]$. 
We then consider the degree $d_\sfD$ induced by the perfect matching $\sfD=\{d\}$ of $Q$.
Then, this gives the degree $\degr\,x=\degr\,z=0$, $\degr\,y=\degr\,w=1$ which makes $R$ a $\ZZ$-graded ring. 
We note that $\sfD$ is a corner perfect matching, and hence the degree zero part of the Jacobian algebra is not finite dimensional (see Proposition~\ref{findim_internal}). 

Let $\fkm$ be the irrelevant homogeneous ideal of $R$. 
It is well-known that the completion $\widehat{R}$ at $\fkm$ has only finitely many indecomposable objects in $\CM(\widehat{R})$ up to isomorphism (see e.g., \cite[Chapter~9 and 12]{Yos}). 
Therefore, $R$ has only finitely many indecomposable objects in $\CM^\ZZ(R)$ up to isomorphism and the degree shift \cite{AR}, and 
the indecomposable objects are $R$, $M\coloneqq(x,z)$, $N\coloneqq(x,w)$ and their degree shifts. 
Thus, we see that the Auslander-Reiten quiver of $\CM^\ZZ(R)$ can be described as the repetition of the quiver shown in Figure~\ref{ARquiver_A1}, where the dotted arrow stands for the Auslander-Reiten translation $\tau$. 

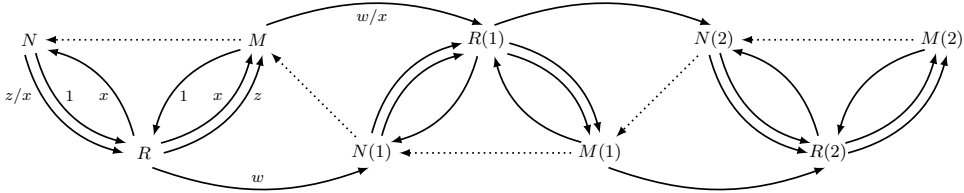
\begin{figure}[H]
\begin{center}
{\scalebox{0.75}{
\begin{tikzpicture}[sarrow/.style={black, -latex},tarrow/.style={black, latex-}]
\node (R0) at (0,0) {\small$R$}; \node (M0) at (2,2) {\small$M$}; \node (N0) at (-2,2) {\small$N$}; 
\node (R+) at (6,2) {\small$R(1)$}; \node (M+) at (8,0) {\small$M(1)$}; \node (N+) at (4,0) {\small$N(1)$}; 
\node (R++) at (12,0) {\small$R(2)$}; \node (M++) at (14,2) {\small$M(2)$}; \node (N++) at (10,2) {\small$N(2)$}; 
\path (R0) ++(180:0.35cm) coordinate (R01); \path (R0) ++(150:0.35cm) coordinate (R02); 
\path (R0) ++(30:0.35cm) coordinate (R03); \path (R0) ++(0:0.35cm) coordinate (R04); 
\path (R0) ++(120:0.35cm) coordinate (R05); \path (R0) ++(60:0.35cm) coordinate (R06); 
\path[tarrow, line width=0.03cm] (R01) edge[bend left] ++(135:2.45cm); 
\path[tarrow, line width=0.03cm] (R02) edge[bend left] ++(135:2.25cm); 
\path[sarrow, line width=0.03cm] (R03) edge[bend right] ++(45:2.25cm); 
\path[sarrow, line width=0.03cm] (R04) edge[bend right] ++(45:2.45cm); 
\path[sarrow, line width=0.03cm] (R05) edge[bend right] ++(135:2.15cm); 
\path[tarrow, line width=0.03cm] (R06) edge[bend left] ++(45:2.15cm); 
\node at (-1.3,1) {\footnotesize $1$}; \node at (-2.2,1) {\footnotesize $z/x$}; \node at (-0.7,1) {\footnotesize $x$}; 
\node at (0.7,1) {\footnotesize $1$}; \node at (1.3,1) {\footnotesize $x$}; \node at (2,1) {\footnotesize $z$}; 

\path (N+) ++(90:0.35cm) coordinate (N+1); \path (N+) ++(60:0.35cm) coordinate (N+2); \path (N+) ++(30:0.38cm) coordinate (N+3); 
\path (M+) ++(90:0.35cm) coordinate (M+1); \path (M+) ++(120:0.35cm) coordinate (M+2); \path (M+) ++(150:0.38cm) coordinate (M+3); 
\path[tarrow, line width=0.03cm] (M+1) edge[bend right] ++(135:2.25cm); 
\path[tarrow, line width=0.03cm] (M+2) edge[bend right] ++(135:2.05cm); 
\path[sarrow, line width=0.03cm] (M+3) edge[bend left] ++(135:2.15cm); 
\path[sarrow, line width=0.03cm] (N+1) edge[bend left] ++(45:2.25cm); 
\path[sarrow, line width=0.03cm] (N+2) edge[bend left] ++(45:2.05cm); 
\path[tarrow, line width=0.03cm] (N+3) edge[bend right] ++(45:2.15cm); 

\path (R0) ++(300:0.3cm) coordinate (R07); 
\path[sarrow, line width=0.03cm] (R07) edge[bend right=20] ++(0:3.8cm); \node at (2,-0.45) {\footnotesize $w$}; 
\path (M0) ++(60:0.3cm) coordinate (M01); 
\path[sarrow, line width=0.03cm] (M01) edge[bend left=20] ++(0:3.8cm); \node at (4,2.4) {\footnotesize $w/x$}; 

\path (R++) ++(180:0.35cm) coordinate (R++1); \path (R++) ++(150:0.35cm) coordinate (R++2); 
\path (R++) ++(30:0.35cm) coordinate (R++3); \path (R++) ++(0:0.35cm) coordinate (R++4); 
\path (R++) ++(120:0.35cm) coordinate (R++5); \path (R++) ++(60:0.35cm) coordinate (R++6); 
\path[tarrow, line width=0.03cm] (R++1) edge[bend left] ++(135:2.45cm); 
\path[tarrow, line width=0.03cm] (R++2) edge[bend left] ++(135:2.25cm); 
\path[sarrow, line width=0.03cm] (R++3) edge[bend right] ++(45:2.25cm); 
\path[sarrow, line width=0.03cm] (R++4) edge[bend right] ++(45:2.45cm); 
\path[sarrow, line width=0.03cm] (R++5) edge[bend right] ++(135:2.15cm); 
\path[tarrow, line width=0.03cm] (R++6) edge[bend left] ++(45:2.15cm); 

\path (R+) ++(60:0.3cm) coordinate (R+7); 
\path[sarrow, line width=0.03cm] (R+7) edge[bend left=20] ++(0:3.8cm); 
\path (M+) ++(300:0.3cm) coordinate (M+1); 
\path[sarrow, line width=0.03cm] (M+1) edge[bend right=20] ++(0:3.8cm); 

\draw[sarrow, line width=0.03cm, dotted] (M0)--(N0); \draw[sarrow, line width=0.03cm, dotted] (N+)--(M0); 
\draw[sarrow, line width=0.03cm, dotted] (M+)--(N+); 
\draw[sarrow, line width=0.03cm, dotted] (M++)--(N++); \draw[sarrow, line width=0.03cm, dotted] (N++)--(M+); 
\end{tikzpicture}
} }
\end{center}
\caption{The Auslander-Reiten quiver of $\CM^\ZZ(R)$ for the $A_1$-singularity}
\label{ARquiver_A1}
\end{figure}

In particular, any arrow factors through a projective object. 
Thus, the stable category $\sCM^\ZZ(R)$, which is a triangulated category whose shift functor is the cosyzygy functor $\Omega^{-1}$, can be described by deleting all solid arrows from Figure~\ref{ARquiver_A1}. 
We note that $\tau$ coincides with the syzygy functor $\Omega$ in our situation. 
Therefore, we can see that $M$ is a tilting object in $\sCM^\ZZ(R)$, and hence it induces an equivalence $\sCM^\ZZ(R)\cong\calD^\rmb(\mc\underline{\End}_R(M))\cong\calD^\rmb(\mc\,\kk)$. 

On the other hand, forgetting the grading on $\sCM^\ZZ(R)$ we have the stable category $\sCM(R)$ which can be considered as the orbit category $\sCM^\ZZ(R)/(1)$. 
Here, the degree shift $(1)$ coincides with $\tau^{-1}{\circ}[1]=\Omega^{-2}$, which is the shift functor $[2]$ in $\sCM^\ZZ(R)\cong\calD^\rmb(\mc\,\kk)$. 
Thus, we have an equivalence $\sCM(R)\cong\calC_2(\kk)$ (see also \cite[Example~5.2.2.]{TV}). 
\end{example}

Using these observations, we have the following corollary. 

\begin{corollary}
\label{existence_tilting}
Let $R$ be a $3$-dimensional Gorenstein toric isolated singularity. Let $\Gamma$ be a consistent dimer model associated with $R$. 
Then, there exists a $\ZZ$-grading on $R$ induced from a perfect matching of $\Gamma$ via \eqref{degree_pm} such that 
the stable category $\sCM^\ZZ R$ of $\ZZ$-graded MCM $R$-modules admits a tilting object. 
\end{corollary}

\begin{proof}
The case where $R$ is the $A_1$-singularity follows from Example~\ref{rem_conifold}. 
If $R$ is not the $A_1$-singularity, then we have the tilting object $e_i\calP(Q,W_Q)$ in $\sCM^\ZZ R$ by \cite[Theorem~4.1(a)]{AIR}, 
because there exists a perfect matching $\sfD$ and an idempotent $e_i$ satisfying (P1)--(P3) in this situation. 
\end{proof}

\section{\bf Mutations of perfect matchings} 
\label{sec_mutation_pm}

In Section~\ref{section_2RI} and \ref{sec_cluster_eq}, we saw that an internal perfect matching gives a $2$-representation infinite algebra (see Proposition~\ref{findim_internal}) and several triangle equivalences (see Corollary~\ref{main_cor}). 
In general, there are some internal perfect matchings corresponding to the same interior lattice point. 
Thus, in this section we investigate the relationship among such internal perfect matchings. 

First, we note that a perfect matching of $Q$ can be considered as a \emph{cut} in the sense of \cite{HI,IO}. 
In particular, by Proposition~\ref{findim_internal} we see that any internal perfect matching is \emph{algebraic} (see \cite[Definition~3.2]{HI}). 
To understand the relationship among cuts, the \emph{mutation of cuts}, which was also introduced in \cite{HI,IO}, is important. 
In the following, we introduce this notion in terms of perfect matchings, and call it the \emph{mutation of perfect matchings}. 

We say that a vertex $k\in Q_0$ is a \emph{strict source} (resp. \emph{strict sink}) of $(Q,\sfD)$ if all arrows ending (resp. starting) at $k$ belong to $\sfD$ and all arrows starting (resp. ending) at $k$ do not belong to $\sfD$. 
Namely, a strict source (resp. strict sink) is a source (resp. sink) of the quiver $Q_\sfD$. 

\begin{definition}
\label{def_mutation_pm}
Let $Q$ be the quiver associated with a dimer model, and $\sfD$ be a perfect matching of $Q$. 
\begin{enumerate}[(1)]
\item We assume that $k\in Q_0$ is a strict source of $(Q,\sfD)$. We define a subset $\lambda^+_k(\sfD)$ of $Q_1$ by removing all arrows in $Q$ ending at $k$ from $\sfD$ and adding all arrows in $Q$ starting at $k$ to $\sfD$. 
\item Dually, we assume that $k\in Q_0$ is a strict sink of $(Q,\sfD)$, and define a subset $\lambda^-_k(\sfD)$ of $Q_1$ by removing all arrows in $Q$ starting at $k$ from $\sfD$ and adding all arrows in $Q$ ending at $k$ to $\sfD$. 
\end{enumerate}
\end{definition}

The following properties follow from the definition. 

\begin{lemma}
\label{basic_mutation_pm}
Let $\sfD$ be a perfect matching of $Q$. For a strict source $($resp. strict sink$)$ $k\in Q_0$ of $(Q,\sfD)$, we have the following\,$:$ 
\begin{enumerate}[\rm(a)]
\item $\lambda^+_k(\sfD)$ $($resp. $\lambda^-_k(\sfD)$$)$ is a perfect matching of $Q$. 
\item $k$ is a strict sink of $(Q,\lambda^+_k(\sfD))$ $($resp. a strict source of $(Q,\lambda^-_k(\sfD))$$)$. 
\item We have $\lambda^-_k(\lambda^+_k(\sfD))=\sfD$ $($resp. $\lambda^+_k(\lambda^-_k(\sfD))=\sfD$$)$. 
\end{enumerate}
\end{lemma}

Since $\lambda^\pm_k(\sfD)$ are again perfect matchings, we call these the \emph{mutations of a perfect matching} $\sfD$ of $Q$ at $k\in Q_0$. 
We also denote by $\lambda^+_k(D)$ (resp. $\lambda^-_k(D)$) the perfect matching of a dimer model $\Gamma$ obtained as the dual of $\lambda^+_k(\sfD)$ (resp. $\lambda^-_k(\sfD)$), and call this the \emph{mutation of a perfect matching} $D$ of $\Gamma$ at $k\in \Gamma_2$. 
We say that two perfect matchings are \emph{mutation equivalent} if they are connected by repeating the mutations of perfect matchings. 
We here note the behavior of the associated quiver under these mutations. 

\begin{observation}
\label{obs_mutationPM}
Let $Q=(Q_0,Q_1)$ be the quiver associated with a dimer model, and $\fkR=\{\partial_aW_Q\mid a\in Q_1\}$ be the set of relations in $Q$ (see Subsection~\ref{subsec_dimer}). 
Recall that for a perfect matching $\sfD_i$, the quiver $Q_{\sfD_i}$ is defined as $(Q_{\sfD_i})_0=Q_0$ and $(Q_{\sfD_i})_1=Q_1{\setminus}\{a\in Q_1\mid a\in\sfD_i\}$. 
Considering the set $\fkR_{\sfD_i}=\{\partial_aW_Q \mid a\in\sfD_i\}$ in $Q_{\sfD_i}$, we have $\calA_{\sfD_i}=\kk Q_{\sfD_i}/\langle\fkR_{\sfD_i} \rangle$. 

Now, we assume that for a perfect matching $\sfD_i$, there is a strict source $k$ of $(Q,\sfD_i)$. 
Let $\lambda_k^+(\sfD_i)=\sfD_j$. Then, by definition we see that the quiver $Q_{\sfD_j}$ is given by $(Q_{\sfD_j})_0=Q_0$ and 
$$
(Q_{\sfD_j})_1=\{a\in(Q_{\sfD_i})_1 \mid \tl(a)\neq k\}\sqcup\{r^*:\hd(r)\rightarrow k \mid r\in\fkR_{\sfD_i}, \tl(r)=k\}, 
$$
where $r^*$ is the new arrow associated to $r\in\fkR_{\sfD_i}$ with $\tl(r)=k$, and especially we have 
$\{r^*:\hd(r)\rightarrow k \mid r\in\fkR_{\sfD_i}, \tl(r)=k\}=\{ a\in\sfD_i \mid \hd(a)=k\}$. 
Then, the set 
$$
\{r\in\fkR_{\sfD_i} \mid \tl(r)\neq k\}\sqcup\{\partial_aW_Q:\hd(a)\rightarrow k \mid a\in (Q_{\sfD_i})_1, \tl(a)=k\} 
$$
of relations in $Q_{\sfD_j}$ coincides with $\fkR_{\sfD_j}=\{\partial_aW_Q \mid a\in\sfD_j\}$ and $\calA_{\sfD_j}=\kk Q_{\sfD_j}/\langle\fkR_{\sfD_j} \rangle$. 

We have a similar observation for the case $\lambda_k^-(\sfD_i)$ with a strict sink $k$ of $(Q,\sfD_i)$. 
\end{observation}

Let $Q$ be the quiver associated with a consistent dimer model. 
If $\sfD$ is an internal perfect matching of $Q$, then $Q_\sfD$ is acyclic (see Proposition~\ref{findim_internal}), 
thus there exists a strict source and sink of $(Q,\sfD)$. 
Thus, we can apply these operations $\lambda^\pm_k$ to all internal perfect matchings. 
In particular, we have the following. 

\begin{lemma}
\label{basic_mutation_pm2}
Let $\Gamma$ be a consistent dimer model, and $Q$ be the associated quiver. 
For an internal perfect matching $D$ of $\Gamma$, we assume that $k$ is a strict sink of $(Q,\sfD)$ $($resp. strict source of $(Q,\sfD)$$)$. 
Then, $\lambda^+_k(D)$ $($resp. $\lambda^-_k(D)$$)$ is also an internal perfect matching corresponding to the same interior lattice point of 
the PM polygon $\Delta$ of $\Gamma$. 
\end{lemma}

\begin{proof}
We set $D^\prime\coloneqq\lambda_k^+(D)$. 
Then, by definition we see that $D-D^\prime$ is a homologically trivial cycle on $\Gamma$, and hence $[D-D^\prime]=(0,0)$. 
Fixing the reference perfect matching $D_0$ of $\Gamma$, we have $(0,0)=[D-D^\prime]=[D-D_0]-[D^\prime-D_0]$.
Therefore, they determine the same lattice point of the PM polygon, and hence $D^\prime$ is internal. 
The case of $\lambda_k^-(D)$ is similar. 
\end{proof}

\begin{remark}
\label{rem_mut_pm}
For a perfect matching $\sfD$ of $Q$, there is a generic parameter $\theta$ such that $\sfD$ is a $\theta$-stable perfect matching (see Proposition~\ref{corresp_pm}(2)). 
However, the mutated ones $\lambda_k^\pm(\sfD)$ are not $\theta$-stable. 
Indeed, let $\theta=(\theta_i)_{i\in Q_0}\in\Theta_\RR$. 
We assume that an internal perfect matching $\sfD$ is $\theta$-stable and $k$ is a strict source of $(Q,\sfD)$. 
Thus, a representation $V=(V_i)_{i\in Q_0}$ such that $V_k=0$ and $V_i\cong\kk$ ($i\neq k$) is a subrepresentation of $V_\sfD$. 
Then we have $\theta(V)=\sum_{i\neq k}\theta_i=-\theta_k>0$. 
On the other hand, since $k$ is a strict sink of $(Q,\lambda_k^+(\sfD))$, 
the simple representation $S_k$ corresponding to $k$ is a subrepresentation of $V_{\lambda_k^+(\sfD)}$. 
Since $\theta(S_k)=\theta_k<0$, we see that $\lambda_k^+(\sfD)$ is not $\theta$-stable. 
We also see that $\lambda_k^-(\sfD)$ is not $\theta$-stable by the argument dual to the above one. 
However, for each of $\lambda_k^+(\sfD), \lambda_k^-(\sfD)$, there respectively exists a generic parameter making them stable (see Proposition~\ref{corresp_pm}(2)). 
\end{remark}

In Figure~\ref{pm_4a}, the perfect matchings $D_5,D_6,D_7,D_8$ are internal, and they correspond to the same interior lattice point. 
By considering the mutations at appropriate faces, we see that these are mutation equivalent. 
This property holds for a more general situation as follows. 

\begin{theorem}
\label{pm_mutation_equiv}
Let $\Gamma$ be a consistent dimer model and $Q$ be the associated quiver. 
Let $D, D^\prime$ be internal perfect matchings of $\Gamma$. 
Then, $D$ and $D^\prime$ are mutation equivalent if and only if $D$ and $D^\prime$ correspond to the same interior lattice point of the PM polygon of $\Gamma$. 
\end{theorem}

\begin{proof}
Let $\sfD,\sfD^\prime$ be perfect matchings of $Q$ corresponding to $D, D^\prime$ respectively. 
The ``only if" part follows from Lemma~\ref{basic_mutation_pm2}. 
Thus, we show the other direction. 

\medskip

\noindent{\bf(Step1)} 
If $D$ and $D^\prime$ correspond to the same lattice point, then we have $[D-D^\prime]=(0,0)$. 
Then, $D-D^\prime$ consists of cycles that can be divided into the following two types: 
\begin{itemize}
\item cycles $C_1,\dots, C_s$ that are homologically trivial, that is, $[C_i]=(0,0)$ for any $i=1,\dots,s$, 
\item cycles $E_1,\dots,E_t,E_{t+1},\dots, E_{2t}$ that are not homologically trivial, and $[E_j]=-[E_{t+j}]=v$ 
for any $j=1,\dots,t$, where $v$ is a certain primitive vector. 
\end{itemize}
Indeed, the former case is trivial, and the later case can be verified as follows. 
Let us consider the cycles $E_1,\dots,E_u$ in $D-D^\prime$ that are not homologically trivial. 
Since they are obtained as the difference of two perfect matchings, 
they do not have self-intersections, and any pair of these cycles do not have an intersection. 
Thus, the slopes of any pair of these cycles are linearly dependent. 
Let $v\coloneqq[E_1]$. Since the slopes of these cycles are primitive, we see that each slope is either $v$ or $-v$. 
Since $[D-D^\prime]=(0,0)$, the number of cycles whose slope is $v$ coincides with that of cycles whose slope is $-v$. 
Thus, we have $u=2t$ for some integer $t>0$, and by renumbering subscripts (if necessary) we have the assertion. 

\medskip

\noindent{\bf(Step2)} 
In this step, we fix the notation. 
Let $C$ be a region on the two-torus $\TT$ that takes the form either 
\begin{enumerate}[\rm (a)]
\setlength{\leftskip}{0.5cm}
\item the region determined as the inside of some homologically trivial cycle $C_i$ 
(assume that $C_i$ is contained in $C$ as the boundary, and there is no homologically trivial cycle in $C$ except $C_i$), or 
\item the region $C_{\ell m}$ determined as follows. 
\end{enumerate}
We consider the cycles $E_1,\dots,E_t,E_{t+1},\dots, E_{2t}$. 
Since $[E_j]=-[E_{t+j}]=v$, there are cycles $E_\ell$ and $E_{t+m}$ with $v=[E_\ell]=-[E_{t+m}]$ such that 
there is no cycle whose homology class is either $v$ or $-v$ in the region determined as the intersection of 
the right of $E_\ell$ and the right of $E_{t+m}$, and we denote this region by $C_{\ell m}$. 
Furthermore, we assume that there is no homologically trivial cycle in $C_{\ell m}$, 
indeed our problem will be reduced to such a situation in (Step5) below. 

Let $C^\circ$ be the strict interior of $C$, and $\partial C$ be the boundary of $C$. 
In particular, $\partial C$ is either $C_i$ or $E_\ell\sqcup E_{t+m}$ in our situation. 
By the assumption, we see that edges appearing on $\partial C$ are alternately contained in $D$ and $D^\prime$, and 
edges appearing in $C^\circ$ are contained in both $D$ and $D^\prime$ or contained in neither $D$ nor $D^\prime$. 
Let $Q_\sfD\cap C^\circ$ be the subquiver of $Q_\sfD$ in which vertices are the ones corresponding to the faces of $\Gamma$ contained in $C^\circ$, and arrows  are the ones of $Q_\sfD$ appearing in $C^\circ$ (especially, we do not contain the arrows crossing over $\partial C$). 
We note that $Q_\sfD\cap C^\circ=Q_{\sfD^\prime}\cap C^\circ$ holds. 
Since $D$ is internal, $Q_\sfD$ is acyclic (see Proposition~\ref{findim_internal}) and so is the subquiver $Q_\sfD\cap C^\circ$.  
Also, let $Q_\sfD\cap \partial C$ be the subquiver of $Q_\sfD$ consisting of the arrows of $Q_\sfD$ crossing over $\partial C$. 

\medskip

\noindent{\bf(Step3)} 
Let us consider the case $\partial C=C_i$. 
We then show that we can take $D$ so that a source of $Q_\sfD\cap C^\circ$ is also a source of $Q_\sfD$. 
Indeed, since the edges appearing on $\partial C$ are alternately contained in $D$ and $D^\prime$, 
we see that all the arrows in $Q_\sfD\cap \partial C$ point in the same direction 
(from the inside of $C$ to the outside, or the opposite direction) and 
all the arrows in $Q_{\sfD^\prime}\cap \partial C$ point in the direction opposite to those in $Q_\sfD\cap \partial C$. 
We choose $D$ so that the arrows in $Q_\sfD\cap \partial C$ point from the inside of $C$ to the outside, 
in which case the edges in $D\cap\partial C$ are the dual of arrows pointing from the outside of $C$ to the inside. 
By this choice of $D$, we see that a source of $Q_\sfD\cap C^\circ$ is also a source of $Q_\sfD$. 

For the case $\partial C=E_\ell\sqcup E_{t+m}$, we also see that any source of $Q_\sfD\cap C^\circ$ is a source of $Q_\sfD$. 
In this case, all the edges in $D\cap\partial C$ (resp. $D^\prime\cap\partial C$) are directed from white to black (resp. from black to white). 
Since $[E_\ell]=-[E_{t+m}]$ and $C^\circ$ is determined as the intersection of the right hand side of $E_\ell$ and that of $E_{t+m}$, 
we see that the edges in $D\cap\partial C$ are the dual of arrows pointing from the outside of $C$ to the inside. 
Thus, we have the assertion. 

\medskip

\noindent{\bf(Step4)} 
We consider the perfect matching $D$ chosen in (Step3). 
Let $V$ be the set of vertices of the quiver $Q_\sfD\cap C^\circ$. Since $Q_\sfD\cap C^\circ$ is acyclic, 
we can define a total order $<_\sfD$ on $V$ such that $\tl(a)<_\sfD\hd(a)$ for any arrow $a$ of $Q_\sfD\cap C^\circ$. 
We will call such an order a \emph{total order compatible with $\sfD$}. 
For simplicity, we let $V\coloneqq\{k_1,\dots, k_n\}$ and suppose that $k_1<_\sfD\cdots <_\sfD k_n$. 
Since the vertex $k_1$ is the minimal element of $V$ with respect to the order $<_\sfD$, it is a source of $Q_\sfD\cap C^\circ$. 
By (Step3), $k_1$ is also a source of $Q_\sfD$, and hence we can apply the mutation $\lambda_{k_1}^+$ to $D$. 
Let $D_1\coloneqq \lambda_{k_1}^+(D)$. Note that the quiver $Q_{\sfD_1}\cap C^\circ$ is acyclic and the set of vertices of this quiver coincides with $V$. 
We define a total order $<_{\sfD_1}$ on $V$ such that $k_2<_{\sfD_1}\cdots <_{\sfD_1} k_n<_{\sfD_1} k_1$. 
Since applying the mutation $\lambda_{k_1}^+$ makes $k_1$ a sink of $Q_{\sfD_1}\cap C^\circ$ (see Lemma~\ref{basic_mutation_pm}) and 
preserves the arrows of $Q_\sfD\cap C^\circ$ not incident to $k_1$, we see that $<_{\sfD_1}$ is a total order compatible with $\sfD_1$. 
By the same argument as above, we see that the vertex $k_2$, which is the minimal element  of $V$ with respect to the order $<_{\sfD_1}$, 
is a source of $Q_{\sfD_1}$. 
Thus, we can define $D_2\coloneqq \lambda_{k_2}^+(D_1)=\lambda_{k_2}^+\lambda_{k_1}^+(D)$ 
and a total order $<_{\sfD_2}$ on $V$ compatible with $\sfD_2$ 
such that $k_3<_{\sfD_2}\cdots <_{\sfD_2} k_n<_{\sfD_2} k_1<_{\sfD_2} k_2$. 

Repeating the same argument, we have the perfect matching $D_r\coloneqq \lambda_{k_r}^+\cdots\lambda_{k_1}^+(D)$ 
and a total order $<_{\sfD_r}$ on $V$ compatible with $\sfD_r$ for $r\ge 3$. 
Then, we finally have the perfect matching $D_n\coloneqq \lambda_{k_n}^+\cdots\lambda_{k_1}^+(D)$ and stop these procedures here. 
We note that we apply mutations of perfect matchings to $D$ exactly once at each vertex in $V$ through the sequence of mutations producing $D_n$. 

Let $a$ be an arrow of $Q_\sfD\cap C^\circ$, and hence we have $a\not\in\sfD$ and $a\not\in\sfD^\prime$. 
Let $\fkD\coloneqq\{D_1,\dots,D_n\}$. 
We suppose that $D_p\in\fkD$ is the perfect matching obtained after applying the mutation at $\tl(a)$, 
which means $\tl(a)$ is a source of $Q_{\sfD_{p-1}}\cap C^\circ$. 
Since $\tl(a)$ is a sink of $Q_{\sfD_p}\cap C^\circ$, we have $a\in\sfD_p$. 
We then suppose that $D_q\in\fkD$ is the perfect matching obtained after applying the mutation at $\hd(a)$, especially $p<q$ by construction. 
Since $\hd(a)$ is a sink of $Q_{\sfD_q}\cap C^\circ$, we have $a\not\in\sfD_q$. 
After this, we do not apply the mutation at $\tl(a), \hd(a)$, thus we have $a\not\in\sfD_n$. 
We then let $b$ be an arrow of $Q_\sfD\cap \partial C$. 
If $b\not\in\sfD$ (equivalently $b\in\sfD^\prime$), then $\tl(b)$ is a vertex of $Q_\sfD\cap C^\circ$, but $\hd(b)$ is not a vertex of $Q_\sfD\cap C^\circ$. 
We suppose that $D_{p^\prime}\in\fkD$ is the perfect matching obtained after applying the mutation at $\tl(b)$. 
Since $\tl(b)$ is a sink of $Q_{\sfD_{p^\prime}}\cap C^\circ$, we have $b\in\sfD_{p^\prime}$. 
After this, we do not apply the mutation at $\tl(b)$, thus we have $b\in\sfD_n$. 
By a similar argument, if $b\in\sfD$ (equivalently $b\not\in\sfD^\prime$), then we have $b\not\in\sfD_n$. 
Thus, we see that $D_n$ coincides with $D^\prime$ in the region $C$, thus $C$ vanishes in $D_n-D^\prime$.

\medskip

\noindent{\bf(Step5)} 
We consider the cycles $C_1,\dots, C_s$, $E_1,\dots,E_t,E_{t+1},\dots, E_{2t}$ discussed in (Step1). 
Applying the arguments in (Step4) to $C_1,\dots, C_s$, 
we have an internal perfect matching mutation equivalent to $D$ and the difference from $D^\prime$ 
consists of the cycles $E_1,\dots, E_{2t}$. 
Thus, we may assume that $D-D^\prime$ consists of cycles $E_1, \dots, E_{2t}$, and we can consider the region $C_{\ell m}$ as in (b) of (Step2). 
By applying the arguments in (Step4) to $C_{\ell m}$, we erase $E_\ell$ and $E_{t+m}$ in the difference of two perfect matchings. 
Thus, repeating this argument, we eventually have the perfect matching that coincide with $D^\prime$, 
which means $D$ and $D^\prime$ are mutation equivalent. 
\end{proof}

\begin{remark}
\label{combinatorics_PM}
In many contexts of mathematics and physics, finding all perfect matchings of a dimer model is the fundamental problem. 
For a consistent dimer model, boundary perfect matchings can be determined by using zigzag paths (see e.g., \cite[Subsection~4.6,4.7]{Bro}, \cite[Subsection~3.4]{Gul}), especially corner ones are determined uniquely. 
On the other hand, by Theorem~\ref{pm_mutation_equiv} if we find an internal perfect matching $\sfD$, 
then we can obtain all internal perfect matchings corresponding to the same interior lattice point as $\sfD$ using the mutations. 
Thus, the mutations give an effective way to list all perfect matchings. 
\end{remark}

\begin{example}
\label{ex_mut_pm1}
We consider the following consistent dimer model and the associated quiver $Q$. 
Here, the rightmost figure is the PM polygon $\Delta$, and the triangulation given in this figure induces a crepant resolution of the associated toric singularity. 
In particular, using the method in Remark~\ref{method_make_triangulation}, we see that this crepant resolution is isomorphic to $\calM_\theta(Q)$ with the generic parameter $\theta=(-7,1, \dots,1)$. 

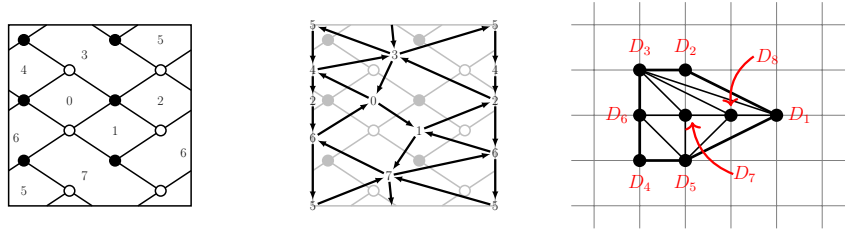
\newcommand{\edgewidth}{0.1cm} 
\newcommand{\nodewidth}{0.1cm} 
\newcommand{\noderad}{0.35} 

\newcommand{\dimerexampleB}{
\coordinate (B1) at (1,3); \coordinate (B2) at (7,3); 
\coordinate (B3) at (1,7); \coordinate (B4) at (7,7); 
\coordinate (B5) at (1,11); \coordinate (B6) at (7,11); 

\coordinate (W1) at (4,1); \coordinate (W2) at (10,1); 
\coordinate (W3) at (4,5); \coordinate (W4) at (10,5); 
\coordinate (W5) at (4,9); \coordinate (W6) at (10,9); 
\draw[line width=0.1cm]  (0,0) rectangle (12,12);

\node [font=\fontsize{40pt}{0pt}\selectfont] (V0) at (4,7) {\scalebox{1.2}{$0$}} ; 
\node [font=\fontsize{40pt}{0pt}\selectfont] (V1) at (7,5) {\scalebox{1.2}{$1$}} ; 
\node [font=\fontsize{40pt}{0pt}\selectfont] (V2) at (10,7) {\scalebox{1.2}{$2$}} ; 
\node [font=\fontsize{40pt}{0pt}\selectfont] (V3) at (5,10) {\scalebox{1.2}{$3$}} ; 
\node [font=\fontsize{40pt}{0pt}\selectfont] (V4) at (1,9) {\scalebox{1.2}{$4$}} ; 
\node [font=\fontsize{40pt}{0pt}\selectfont] (V5a) at (1,1) {\scalebox{1.2}{$5$}} ; 
\node [font=\fontsize{40pt}{0pt}\selectfont] (V5b) at (10,11) {\scalebox{1.2}{$5$}} ; 
\node [font=\fontsize{40pt}{0pt}\selectfont] (V6a) at (0.5,4.5) {\scalebox{1.2}{$6$}} ; 
\node [font=\fontsize{40pt}{0pt}\selectfont] (V6b) at (11.5,3.5) {\scalebox{1.2}{$6$}} ; 
\node [font=\fontsize{40pt}{0pt}\selectfont] (V7) at (5,2) {\scalebox{1.2}{$7$}} ; 

\draw[line width=\edgewidth] (B1)--(W1); \draw[line width=\edgewidth] (B1)--(W3); 
\draw[line width=\edgewidth] (B2)--(W2); \draw[line width=\edgewidth] (B2)--(W3); \draw[line width=\edgewidth] (B2)--(W4); 
\draw[line width=\edgewidth] (B3)--(W3); \draw[line width=\edgewidth] (B3)--(W5); 
\draw[line width=\edgewidth] (B4)--(W3); \draw[line width=\edgewidth] (B4)--(W4); 
\draw[line width=\edgewidth] (B4)--(W5); \draw[line width=\edgewidth] (B4)--(W6); 
\draw[line width=\edgewidth] (B5)--(W5); \draw[line width=\edgewidth] (B6)--(W6); 

\draw[line width=\edgewidth] (W1)--(2.5,0); \draw[line width=\edgewidth] (B5)--(2.5,12); 
\draw[line width=\edgewidth] (W1)--(5.5,0); \draw[line width=\edgewidth] (B6)--(5.5,12); 
\draw[line width=\edgewidth] (W2)--(12,2.3334); \draw[line width=\edgewidth] (B1)--(0,2.3334); 
\draw[line width=\edgewidth] (W2)--(8.5,0); \draw[line width=\edgewidth] (B6)--(8.5,12); 
\draw[line width=\edgewidth] (W4)--(12,6.3334); \draw[line width=\edgewidth] (B3)--(0,6.3334);
\draw[line width=\edgewidth] (W6)--(12,7.6666); \draw[line width=\edgewidth] (B3)--(0,7.6666); 
\draw[line width=\edgewidth] (W6)--(12,10.3334); \draw[line width=\edgewidth] (B5)--(0,10.3334); 

\draw  [line width=\nodewidth, fill=black] (B1) circle [radius=\noderad] ; \draw  [line width=\nodewidth, fill=black] (B2) circle [radius=\noderad] ;
\draw  [line width=\nodewidth, fill=black] (B3) circle [radius=\noderad] ; \draw  [line width=\nodewidth, fill=black] (B4) circle [radius=\noderad] ;
\draw  [line width=\nodewidth, fill=black] (B5) circle [radius=\noderad] ; \draw  [line width=\nodewidth, fill=black] (B6) circle [radius=\noderad] ;
\draw [line width=\nodewidth, fill=white] (W1) circle [radius=\noderad] ; \draw [line width=\nodewidth, fill=white] (W2) circle [radius=\noderad] ;
\draw [line width=\nodewidth, fill=white] (W3) circle [radius=\noderad] ; \draw [line width=\nodewidth, fill=white] (W4) circle [radius=\noderad] ;
\draw [line width=\nodewidth, fill=white] (W5) circle [radius=\noderad] ; \draw [line width=\nodewidth, fill=white] (W6) circle [radius=\noderad] ;
}

\begin{center}
\begin{tikzpicture}

\node (DM1) at (0,0) 
{\scalebox{0.2}{
\begin{tikzpicture}
\dimerexampleB
\end{tikzpicture}
} }; 

\node (DM2) at (4,0) 
{\scalebox{0.2}{
\begin{tikzpicture}[sarrow/.style={black, -latex, very thick}, ssarrow/.style={black, latex-, very thick}]

\newcommand{\nodecolor}{lightgray} 
\newcommand{\arrowwidth}{0.15cm} 
\draw[line width=0.1cm, \nodecolor]  (0,0) rectangle (12,12); 

\node [font=\fontsize{40pt}{0pt}\selectfont] (V0) at (4,7) {\scalebox{1.2}{$0$}} ; 
\node [font=\fontsize{40pt}{0pt}\selectfont] (V1) at (7,5) {\scalebox{1.2}{$1$}} ; 
\node [font=\fontsize{40pt}{0pt}\selectfont] (V2a) at (0,7) {\scalebox{1.2}{$2$}} ; 
\node [font=\fontsize{40pt}{0pt}\selectfont] (V2b) at (12,7) {\scalebox{1.2}{$2$}} ; 
\node [font=\fontsize{40pt}{0pt}\selectfont] (V3) at (5.4,10) {\scalebox{1.2}{$3$}} ; 
\node [font=\fontsize{40pt}{0pt}\selectfont] (V4a) at (0,9) {\scalebox{1.2}{$4$}} ; 
\node [font=\fontsize{40pt}{0pt}\selectfont] (V4b) at (12,9) {\scalebox{1.2}{$4$}} ; 
\node [font=\fontsize{40pt}{0pt}\selectfont] (V5a) at (0,0) {\scalebox{1.2}{$5$}} ; 
\node [font=\fontsize{40pt}{0pt}\selectfont] (V5b) at (12,0) {\scalebox{1.2}{$5$}} ; 
\node [font=\fontsize{40pt}{0pt}\selectfont] (V5c) at (12,12) {\scalebox{1.2}{$5$}} ; 
\node [font=\fontsize{40pt}{0pt}\selectfont] (V5d) at (0,12) {\scalebox{1.2}{$5$}} ; 
\node [font=\fontsize{40pt}{0pt}\selectfont] (V6a) at (0,4.5) {\scalebox{1.2}{$6$}} ; 
\node [font=\fontsize{40pt}{0pt}\selectfont] (V6b) at (12,3.5) {\scalebox{1.2}{$6$}} ; 
\node [font=\fontsize{40pt}{0pt}\selectfont] (V7) at (5,2) {\scalebox{1.2}{$7$}} ; 

\draw[line width=\edgewidth, \nodecolor] (B1)--(W1); \draw[line width=\edgewidth, \nodecolor] (B1)--(W3); 
\draw[line width=\edgewidth, \nodecolor] (B2)--(W2); \draw[line width=\edgewidth, \nodecolor] (B2)--(W3); 
\draw[line width=\edgewidth, \nodecolor] (B2)--(W4); 
\draw[line width=\edgewidth, \nodecolor] (B3)--(W3); \draw[line width=\edgewidth, \nodecolor] (B3)--(W5); 
\draw[line width=\edgewidth, \nodecolor] (B4)--(W3); \draw[line width=\edgewidth, \nodecolor] (B4)--(W4); 
\draw[line width=\edgewidth, \nodecolor] (B4)--(W5); \draw[line width=\edgewidth, \nodecolor] (B4)--(W6); 
\draw[line width=\edgewidth, \nodecolor] (B5)--(W5); \draw[line width=\edgewidth, \nodecolor] (B6)--(W6); 

\draw[line width=\edgewidth, \nodecolor] (W1)--(2.5,0); \draw[line width=\edgewidth, \nodecolor] (B5)--(2.5,12); 
\draw[line width=\edgewidth, \nodecolor] (W1)--(5.5,0); \draw[line width=\edgewidth, \nodecolor] (B6)--(5.5,12); 
\draw[line width=\edgewidth, \nodecolor] (W2)--(12,2.3334); \draw[line width=\edgewidth, \nodecolor] (B1)--(0,2.3334); 
\draw[line width=\edgewidth, \nodecolor] (W2)--(8.5,0); \draw[line width=\edgewidth, \nodecolor] (B6)--(8.5,12); 
\draw[line width=\edgewidth, \nodecolor] (W4)--(12,6.3334); \draw[line width=\edgewidth, \nodecolor] (B3)--(0,6.3334);
\draw[line width=\edgewidth, \nodecolor] (W6)--(12,7.6666); \draw[line width=\edgewidth, \nodecolor] (B3)--(0,7.6666); 
\draw[line width=\edgewidth, \nodecolor] (W6)--(12,10.3334); \draw[line width=\edgewidth, \nodecolor] (B5)--(0,10.3334); 

\draw  [line width=\nodewidth, fill=\nodecolor, \nodecolor] (B1) circle [radius=\noderad] ; 
\draw  [line width=\nodewidth, fill=\nodecolor, \nodecolor] (B2) circle [radius=\noderad] ;
\draw  [line width=\nodewidth, fill=\nodecolor, \nodecolor] (B3) circle [radius=\noderad] ; 
\draw  [line width=\nodewidth, fill=\nodecolor, \nodecolor] (B4) circle [radius=\noderad] ;
\draw  [line width=\nodewidth, fill=\nodecolor, \nodecolor] (B5) circle [radius=\noderad] ; 
\draw  [line width=\nodewidth, fill=\nodecolor, \nodecolor] (B6) circle [radius=\noderad] ;
\draw [line width=\nodewidth, \nodecolor, fill=white] (W1) circle [radius=\noderad] ; 
\draw [line width=\nodewidth, \nodecolor, fill=white] (W2) circle [radius=\noderad] ;
\draw [line width=\nodewidth, \nodecolor, fill=white] (W3) circle [radius=\noderad] ; 
\draw [line width=\nodewidth, \nodecolor, fill=white] (W4) circle [radius=\noderad] ;
\draw [line width=\nodewidth, \nodecolor, fill=white] (W5) circle [radius=\noderad] ; 
\draw [line width=\nodewidth, \nodecolor, fill=white] (W6) circle [radius=\noderad] ;

\draw [sarrow, line width=\arrowwidth] (V0)--(V1); \draw [sarrow, line width=\arrowwidth] (V0)--(V4a); 
\draw [sarrow, line width=\arrowwidth] (V3)--(V0); \draw [sarrow, line width=\arrowwidth] (V6a)--(V0); 
\draw [sarrow, line width=\arrowwidth] (V1)--(V2b); \draw [sarrow, line width=\arrowwidth] (V1)--(V7); 
\draw [sarrow, line width=\arrowwidth] (V6b)--(V1); \draw [sarrow, line width=\arrowwidth] (V2b)--(V3); 
\draw [sarrow, line width=\arrowwidth] (V2b)--(V6b); \draw [sarrow, line width=\arrowwidth] (V4a)--(V3); 
\draw [sarrow, line width=\arrowwidth] (V3)--(V5c); \draw [sarrow, line width=\arrowwidth] (V5a)--(V7); 
\draw [sarrow, line width=\arrowwidth] (V6a)--(V5a); \draw [sarrow, line width=\arrowwidth] (V7)--(V6a); 
\draw [sarrow, line width=\arrowwidth] (V7)--(V6b); \draw [sarrow, line width=\arrowwidth] (V2a)--(V6a); 
\draw [sarrow, line width=\arrowwidth] (V4a)--(V2a); \draw [sarrow, line width=\arrowwidth] (V4b)--(V2b); 
\draw [sarrow, line width=\arrowwidth] (V6b)--(V5b); 
\draw [sarrow, line width=\arrowwidth] (V5d)--(V4a); \draw [sarrow, line width=\arrowwidth] (V5c)--(V4b); 
\draw [sarrow, line width=\arrowwidth] (V3)--(V5d); \draw [sarrow, line width=\arrowwidth] (V5b)--(V7); 
\draw [sarrow, line width=\arrowwidth] (V7)--(5.2,0); \draw [sarrow, line width=\arrowwidth] (5.2,12)--(V3); 
\end{tikzpicture}
} }; 

\node at (8,0) {
\scalebox{0.6}{
\begin{tikzpicture}
\coordinate (00) at (0,0); \coordinate (10) at (1,0); \coordinate (01) at (0,1); \coordinate (-10) at (-1,0); \coordinate (-11) at (-1,1);  
\coordinate (20) at (2,0); \coordinate (-1-1) at (-1,-1); \coordinate (0-1) at (0,-1);   

\draw [step=1, gray] (-2.5,-2.5) grid (3.5,2.5);
\draw [line width=0.06cm] (20)--(01)--(-11)--(-1-1)--(0-1)--(20) ;  

\draw [line width=0.05cm, fill=black] (00) circle [radius=0.08] ; \draw [line width=0.05cm, fill=black] (10) circle [radius=0.08] ; 
\draw [line width=0.05cm, fill=black] (01) circle [radius=0.08] ; \draw [line width=0.05cm, fill=black] (-10) circle [radius=0.08] ; 
\draw [line width=0.05cm, fill=black] (-11) circle [radius=0.08] ; \draw [line width=0.05cm, fill=black] (20) circle [radius=0.08] ; 
\draw [line width=0.05cm, fill=black] (-1-1) circle [radius=0.08] ; \draw [line width=0.05cm, fill=black] (0-1) circle [radius=0.08] ; 

\draw [line width=0.035cm] (00)--(-10) ; \draw [line width=0.035cm] (00)--(-11) ; \draw [line width=0.035cm] (00)--(0-1) ; 
\draw [line width=0.035cm] (00)--(10) ; \draw [line width=0.035cm] (-10)--(0-1) ; \draw [line width=0.035cm] (0-1)--(10) ; 
\draw [line width=0.035cm] (-11)--(10) ; \draw [line width=0.035cm] (-11)--(20) ; \draw [line width=0.035cm] (10)--(20) ; 

\node [red] at (2.5,0) {\Large$D_1$} ; 
\node [red] at (0,1.5) {\Large$D_2$} ; \node [red] at (-1,1.5) {\Large$D_3$} ; 
\node [red] at (-1,-1.5) {\Large$D_4$} ; \node [red] at (0,-1.5) {\Large$D_5$} ; 
\node [red] at (-1.5,0) {\Large$D_6$} ; 
\node [red] at (1.3,-1.3) {\Large$D_7$} ; \draw[->, line width=0.05cm,red] (1.05,-1.3) to [bend left] (0.15,-0.15);  
\node [red] at (1.8,1.3) {\Large$D_8$} ; \draw[->, line width=0.05cm,red] (1.5,1.3) to [bend right] (1,0.2); 

\end{tikzpicture}
}} ;

\end{tikzpicture}
\end{center}

The following figures are $\theta$-stable perfect matchings. 
In particular, $D_1, \dots, D_5$ are corner perfect matchings, $D_6$ is a boundary one, and $D_7,D_8$ are internal ones. 
These correspond to lattice points in $\Delta$ as shown in the above figure. 

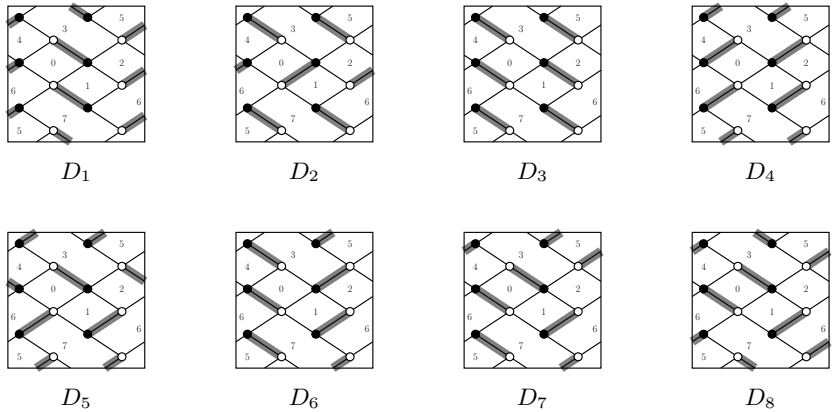
\begin{center}
\begin{tikzpicture}
\newcommand{\pmwidth}{0.8cm} 
\newcommand{\pmcolor}{gray} 

\node at (0,-1.3) {\small$D_1$}; \node at (3,-1.3) {\small$D_2$}; \node at (6,-1.3) {\small$D_3$}; \node at (9,-1.3) {\small$D_4$}; 
\node at (0,-4.3) {\small$D_5$}; \node at (3,-4.3) {\small$D_6$}; \node at (6,-4.3) {\small$D_7$}; \node at (9,-4.3) {\small$D_8$}; 

\node (PM1) at (0,0) 
{\scalebox{0.15}{
\begin{tikzpicture}
\draw[line width=\pmwidth, color=\pmcolor] (W1)--(5.5,0); \draw[line width=\pmwidth, color=\pmcolor] (B6)--(5.5,12); 
\draw[line width=\pmwidth, color=\pmcolor] (B2)--(W3); \draw[line width=\pmwidth, color=\pmcolor] (B4)--(W5); 
\draw[line width=\pmwidth, color=\pmcolor] (W2)--(12,2.3334); \draw[line width=\pmwidth, color=\pmcolor] (B1)--(0,2.3334); 
\draw[line width=\pmwidth, color=\pmcolor] (B3)--(0,6.3334); \draw[line width=\pmwidth, color=\pmcolor] (W4)--(12,6.3334); 
\draw[line width=\pmwidth, color=\pmcolor] (B5)--(0,10.3334); \draw[line width=\pmwidth, color=\pmcolor] (W6)--(12,10.3334); 
\dimerexampleB
\end{tikzpicture}
} }; 

\node (PM2) at (3,0) 
{\scalebox{0.15}{
\begin{tikzpicture}
\draw[line width=\pmwidth, color=\pmcolor] (W1)--(B1); \draw[line width=\pmwidth, color=\pmcolor] (W2)--(B2); 
\draw[line width=\pmwidth, color=\pmcolor] (W3)--(B4); 
\draw[line width=\pmwidth, color=\pmcolor] (W5)--(B5); \draw[line width=\pmwidth, color=\pmcolor] (W6)--(B6); 
\draw[line width=\pmwidth, color=\pmcolor] (B3)--(0,6.3334); \draw[line width=\pmwidth, color=\pmcolor] (W4)--(12,6.3334); 
\dimerexampleB
\end{tikzpicture}
} }; 

\node (PM3) at (6,0) 
{\scalebox{0.15}{
\begin{tikzpicture}
\draw[line width=\pmwidth, color=\pmcolor] (W1)--(B1); \draw[line width=\pmwidth, color=\pmcolor] (W2)--(B2); 
\draw[line width=\pmwidth, color=\pmcolor] (W3)--(B3); \draw[line width=\pmwidth, color=\pmcolor] (W4)--(B4); 
\draw[line width=\pmwidth, color=\pmcolor] (W5)--(B5); \draw[line width=\pmwidth, color=\pmcolor] (W6)--(B6); 
\dimerexampleB
\end{tikzpicture}
} }; 

\node (PM4) at (9,0) 
{\scalebox{0.15}{
\begin{tikzpicture}
\draw[line width=\pmwidth, color=\pmcolor] (W3)--(B1); \draw[line width=\pmwidth, color=\pmcolor] (W4)--(B2); 
\draw[line width=\pmwidth, color=\pmcolor] (W5)--(B3); \draw[line width=\pmwidth, color=\pmcolor] (W6)--(B4); 
\draw[line width=\pmwidth, color=\pmcolor] (W1)--(2.5,0); \draw[line width=\pmwidth, color=\pmcolor] (B5)--(2.5,12); 
\draw[line width=\pmwidth, color=\pmcolor] (W2)--(8.5,0); \draw[line width=\pmwidth, color=\pmcolor] (B6)--(8.5,12); 
\dimerexampleB
\end{tikzpicture}
} }; 

\node (PM5) at (0,-3) 
{\scalebox{0.15}{
\begin{tikzpicture}
\draw[line width=\pmwidth, color=\pmcolor] (W3)--(B1); \draw[line width=\pmwidth, color=\pmcolor] (W4)--(B2); 
\draw[line width=\pmwidth, color=\pmcolor] (W5)--(B4); 
\draw[line width=\pmwidth, color=\pmcolor] (B3)--(0,7.6666); \draw[line width=\pmwidth, color=\pmcolor] (W6)--(12,7.6666); 
\draw[line width=\pmwidth, color=\pmcolor] (W1)--(2.5,0); \draw[line width=\pmwidth, color=\pmcolor] (B5)--(2.5,12); 
\draw[line width=\pmwidth, color=\pmcolor] (W2)--(8.5,0); \draw[line width=\pmwidth, color=\pmcolor] (B6)--(8.5,12); 
\dimerexampleB
\end{tikzpicture}
} }; 

\node (PM6) at (3,-3) 
{\scalebox{0.15}{
\begin{tikzpicture}
\draw[line width=\pmwidth, color=\pmcolor] (W1)--(B1); \draw[line width=\pmwidth, color=\pmcolor] (W4)--(B2); 
\draw[line width=\pmwidth, color=\pmcolor] (W3)--(B3); \draw[line width=\pmwidth, color=\pmcolor] (W6)--(B4); 
\draw[line width=\pmwidth, color=\pmcolor] (W5)--(B5); 
\draw[line width=\pmwidth, color=\pmcolor] (W2)--(8.5,0); \draw[line width=\pmwidth, color=\pmcolor] (B6)--(8.5,12); 
\dimerexampleB
\end{tikzpicture}
} }; 

\node (PM7) at (6,-3) 
{\scalebox{0.15}{
\begin{tikzpicture}
\draw[line width=\pmwidth, color=\pmcolor] (B1)--(W1); \draw[line width=\pmwidth, color=\pmcolor] (B2)--(W4); 
\draw[line width=\pmwidth, color=\pmcolor] (B3)--(W3); \draw[line width=\pmwidth, color=\pmcolor] (B4)--(W5); 
\draw[line width=\pmwidth, color=\pmcolor] (W2)--(8.5,0); \draw[line width=\pmwidth, color=\pmcolor] (B6)--(8.5,12); 
\draw[line width=\pmwidth, color=\pmcolor] (B5)--(0,10.3334); \draw[line width=\pmwidth, color=\pmcolor] (W6)--(12,10.3334); 
\dimerexampleB
\end{tikzpicture}
} }; 

\node (PM8) at (9,-3) 
{\scalebox{0.15}{
\begin{tikzpicture}
\draw[line width=\pmwidth, color=\pmcolor] (W1)--(5.5,0); \draw[line width=\pmwidth, color=\pmcolor] (B6)--(5.5,12); 
\draw[line width=\pmwidth, color=\pmcolor] (B2)--(W4); 
\draw[line width=\pmwidth, color=\pmcolor] (B3)--(W3); \draw[line width=\pmwidth, color=\pmcolor] (B4)--(W5); 
\draw[line width=\pmwidth, color=\pmcolor] (W2)--(12,2.3334); \draw[line width=\pmwidth, color=\pmcolor] (B1)--(0,2.3334); 
\draw[line width=\pmwidth, color=\pmcolor] (B5)--(0,10.3334); \draw[line width=\pmwidth, color=\pmcolor] (W6)--(12,10.3334); 
\dimerexampleB
\end{tikzpicture}
} }; 
\end{tikzpicture}
\end{center}

Then, we consider the mutations of the perfect matching $D_7$. 
Since $0\in Q_0$ is a strict source of $(Q,\sfD_7)$, and $5\in Q_0$ is a strict sink of $(Q,\sfD_7)$, we can apply the mutations $\lambda^+_0$ and $\lambda^-_5$ to $D_7$. 
Repeating these arguments, we have the exchange graph of mutations as shown in Figure~\ref{ex_graph_pm}. 
In this figure, each edge is indexed by the mutable vertex. 
We note that the mutated perfect matchings correspond to the same interior lattice point in $\Delta$ by Theorem~\ref{pm_mutation_equiv}. 
However, the mutated ones are not $\theta$-stable (see Remark~\ref{rem_mut_pm}). 
For example, the bottom perfect matching in Figure~\ref{ex_graph_pm} is not $\theta$-stable, but $\theta^\prime$-stable where 
$\theta^\prime=(1,1,1,-7,1,1,1,1)$. 
Similarly, by applying the mutations to $D_8$, we can obtain perfect matchings mutation equivalent to $D_8$ 
which correspond to the same interior lattice point as $D_8$. 

We further remark that as we mentioned $D_7$ and $D_8$ are $\theta$-stable perfect matchings, and $5\in Q_0$ (resp. $6\in Q_0$) is a strict sink of $(Q,\sfD_7)$ (resp. $(Q,\sfD_8)$). 
These vertices are \emph{essential vertices} in the sense of \cite{CIK}, and these are related to ``Reid's recipe" (see e.g., \cite{BCQV,Cra,Tap}). 

\begin{figure}[H]
\begin{center}
\scalebox{0.85}{
\begin{tikzpicture}

\newcommand{\pmwidth}{0.8cm} 
\newcommand{\pmcolor}{gray} 

\node (PM15) at (2.5,2) 
{\scalebox{0.12}{
\begin{tikzpicture}
\draw[line width=\pmwidth, color=\pmcolor] (B1)--(W1); \draw[line width=\pmwidth, color=\pmcolor] (B2)--(W3); 
\draw[line width=\pmwidth, color=\pmcolor] (B3)--(0,6.3334); \draw[line width=\pmwidth, color=\pmcolor] (W4)--(12,6.3334); 
\draw[line width=\pmwidth, color=\pmcolor] (B5)--(W5); \draw[line width=\pmwidth, color=\pmcolor] (W2)--(8.5,0);  
\draw[line width=\pmwidth, color=\pmcolor] (B6)--(8.5,12); \draw[line width=\pmwidth, color=\pmcolor] (W6)--(B4); 
\dimerexampleB
\end{tikzpicture}
} }; 

\node (PM13) at (7.5,2) 
{\scalebox{0.12}{
\begin{tikzpicture}
\draw[line width=\pmwidth, color=\pmcolor] (B1)--(W1); \draw[line width=\pmwidth, color=\pmcolor] (B2)--(W4); 
\draw[line width=\pmwidth, color=\pmcolor] (B4)--(W3); \draw[line width=\pmwidth, color=\pmcolor] (B5)--(W5); 
\draw[line width=\pmwidth, color=\pmcolor] (W2)--(8.5,0); \draw[line width=\pmwidth, color=\pmcolor] (B6)--(8.5,12); 
\draw[line width=\pmwidth, color=\pmcolor] (B3)--(0,7.6666); \draw[line width=\pmwidth, color=\pmcolor] (W6)--(12,7.6666); 
\dimerexampleB
\end{tikzpicture}
} };

\node (PM7) at (12.5,2) 
{\scalebox{0.12}{
\begin{tikzpicture}
\draw[line width=\pmwidth, color=\pmcolor] (B1)--(W1); \draw[line width=\pmwidth, color=\pmcolor] (B2)--(W4); 
\draw[line width=\pmwidth, color=\pmcolor] (B3)--(W3); \draw[line width=\pmwidth, color=\pmcolor] (B4)--(W5); 
\draw[line width=\pmwidth, color=\pmcolor] (W2)--(8.5,0); \draw[line width=\pmwidth, color=\pmcolor] (B6)--(8.5,12); 
\draw[line width=\pmwidth, color=\pmcolor] (B5)--(0,10.3334); \draw[line width=\pmwidth, color=\pmcolor] (W6)--(12,10.3334); 
\dimerexampleB
\end{tikzpicture}
} }; 

\node (PM19) at (0,0) 
{\scalebox{0.12}{
\begin{tikzpicture}
\draw[line width=\pmwidth, color=\pmcolor] (B1)--(W3); \draw[line width=\pmwidth, color=\pmcolor] (B5)--(W5); 
\draw[line width=\pmwidth, color=\pmcolor] (B2)--(W2); \draw[line width=\pmwidth, color=\pmcolor] (B4)--(W6); 
\draw[line width=\pmwidth, color=\pmcolor] (W1)--(5.5,0); \draw[line width=\pmwidth, color=\pmcolor] (B6)--(5.5,12); 
\draw[line width=\pmwidth, color=\pmcolor] (B3)--(0,6.3334); \draw[line width=\pmwidth, color=\pmcolor] (W4)--(12,6.3334); 
\dimerexampleB
\end{tikzpicture}
} }; 

\node (PM14) at (5,0) 
{\scalebox{0.12}{
\begin{tikzpicture}
\draw[line width=\pmwidth, color=\pmcolor] (B1)--(W1); \draw[line width=\pmwidth, color=\pmcolor] (B2)--(W3); 
\draw[line width=\pmwidth, color=\pmcolor] (B4)--(W4); \draw[line width=\pmwidth, color=\pmcolor] (B5)--(W5); 
\draw[line width=\pmwidth, color=\pmcolor] (W2)--(8.5,0); \draw[line width=\pmwidth, color=\pmcolor] (B6)--(8.5,12); 
\draw[line width=\pmwidth, color=\pmcolor] (B3)--(0,7.6666); \draw[line width=\pmwidth, color=\pmcolor] (W6)--(12,7.6666); 
\dimerexampleB
\end{tikzpicture}
} }; 

\node (PM9) at (10,0) 
{\scalebox{0.12}{
\begin{tikzpicture}
\draw[line width=\pmwidth, color=\pmcolor] (B1)--(W1); \draw[line width=\pmwidth, color=\pmcolor] (B2)--(W4); 
\draw[line width=\pmwidth, color=\pmcolor] (B4)--(W3); \draw[line width=\pmwidth, color=\pmcolor] (B3)--(W5); 
\draw[line width=\pmwidth, color=\pmcolor] (W2)--(8.5,0); \draw[line width=\pmwidth, color=\pmcolor] (B6)--(8.5,12); 
\draw[line width=\pmwidth, color=\pmcolor] (B5)--(0,10.3334); \draw[line width=\pmwidth, color=\pmcolor] (W6)--(12,10.3334); 
\dimerexampleB
\end{tikzpicture}
} }; 

\node (PM10) at (15,0) 
{\scalebox{0.12}{
\begin{tikzpicture}
\draw[line width=\pmwidth, color=\pmcolor] (B2)--(W4); \draw[line width=\pmwidth, color=\pmcolor] (B3)--(W3); 
\draw[line width=\pmwidth, color=\pmcolor] (B4)--(W5); \draw[line width=\pmwidth, color=\pmcolor] (B6)--(W6); 
\draw[line width=\pmwidth, color=\pmcolor] (W1)--(2.5,0); \draw[line width=\pmwidth, color=\pmcolor] (B5)--(2.5,12); 
\draw[line width=\pmwidth, color=\pmcolor] (W2)--(12,2.3334); \draw[line width=\pmwidth, color=\pmcolor] (B1)--(0,2.3334); 
\dimerexampleB
\end{tikzpicture}
} }; 

\node (PM18) at (2.5,-2) 
{\scalebox{0.12}{
\begin{tikzpicture}
\draw[line width=\pmwidth, color=\pmcolor] (B1)--(W3); \draw[line width=\pmwidth, color=\pmcolor] (B5)--(W5); 
\draw[line width=\pmwidth, color=\pmcolor] (B2)--(W2); \draw[line width=\pmwidth, color=\pmcolor] (B4)--(W4); 
\draw[line width=\pmwidth, color=\pmcolor] (W1)--(5.5,0); \draw[line width=\pmwidth, color=\pmcolor] (B6)--(5.5,12); 
\draw[line width=\pmwidth, color=\pmcolor] (B3)--(0,7.6666); \draw[line width=\pmwidth, color=\pmcolor] (W6)--(12,7.6666); 
\dimerexampleB
\end{tikzpicture}
} }; 

\node (PM11) at (7.5,-2) 
{\scalebox{0.12}{
\begin{tikzpicture}
\draw[line width=\pmwidth, color=\pmcolor] (B1)--(W1); \draw[line width=\pmwidth, color=\pmcolor] (B3)--(W5); 
\draw[line width=\pmwidth, color=\pmcolor] (B2)--(W3); \draw[line width=\pmwidth, color=\pmcolor] (B4)--(W4); 
\draw[line width=\pmwidth, color=\pmcolor] (W2)--(8.5,0); \draw[line width=\pmwidth, color=\pmcolor] (B6)--(8.5,12); 
\draw[line width=\pmwidth, color=\pmcolor] (B5)--(0,10.3334); \draw[line width=\pmwidth, color=\pmcolor] (W6)--(12,10.3334); 
\dimerexampleB
\end{tikzpicture}
} }; 

\node (PM12) at (12.5,-2) 
{\scalebox{0.12}{
\begin{tikzpicture}
\draw[line width=\pmwidth, color=\pmcolor] (B2)--(W4); \draw[line width=\pmwidth, color=\pmcolor] (B4)--(W3); 
\draw[line width=\pmwidth, color=\pmcolor] (B3)--(W5); \draw[line width=\pmwidth, color=\pmcolor] (B6)--(W6); 
\draw[line width=\pmwidth, color=\pmcolor] (W1)--(2.5,0); \draw[line width=\pmwidth, color=\pmcolor] (B5)--(2.5,12); 
\draw[line width=\pmwidth, color=\pmcolor] (W2)--(12,2.3334); \draw[line width=\pmwidth, color=\pmcolor] (B1)--(0,2.3334); 
\dimerexampleB
\end{tikzpicture}
} };

\node (PM16) at (10,-4) 
{\scalebox{0.12}{
\begin{tikzpicture}
\draw[line width=\pmwidth, color=\pmcolor] (B4)--(W4); \draw[line width=\pmwidth, color=\pmcolor] (B2)--(W3); 
\draw[line width=\pmwidth, color=\pmcolor] (B3)--(W5); \draw[line width=\pmwidth, color=\pmcolor] (B6)--(W6); 
\draw[line width=\pmwidth, color=\pmcolor] (W1)--(2.5,0); \draw[line width=\pmwidth, color=\pmcolor] (B5)--(2.5,12); 
\draw[line width=\pmwidth, color=\pmcolor] (W2)--(12,2.3334); \draw[line width=\pmwidth, color=\pmcolor] (B1)--(0,2.3334); 
\dimerexampleB
\end{tikzpicture}
} }; 

\node (PM17) at (5,-4) 
{\scalebox{0.12}{
\begin{tikzpicture}
\draw[line width=\pmwidth, color=\pmcolor] (B1)--(W3); \draw[line width=\pmwidth, color=\pmcolor] (B3)--(W5); 
\draw[line width=\pmwidth, color=\pmcolor] (B2)--(W2); \draw[line width=\pmwidth, color=\pmcolor] (B4)--(W4); 
\draw[line width=\pmwidth, color=\pmcolor] (W1)--(5.5,0); \draw[line width=\pmwidth, color=\pmcolor] (B6)--(5.5,12); 
\draw[line width=\pmwidth, color=\pmcolor] (B5)--(0,10.3334); \draw[line width=\pmwidth, color=\pmcolor] (W6)--(12,10.3334); 
\dimerexampleB
\end{tikzpicture}
} };

\node (PM20) at (7.5,5.7) 
{\scalebox{0.12}{
\begin{tikzpicture}
\draw[line width=\pmwidth, color=\pmcolor] (B1)--(W3); \draw[line width=\pmwidth, color=\pmcolor] (B2)--(W2); 
\draw[line width=\pmwidth, color=\pmcolor] (B4)--(W5); \draw[line width=\pmwidth, color=\pmcolor] (B6)--(W6); 
\draw[line width=\pmwidth, color=\pmcolor] (W1)--(2.5,0); \draw[line width=\pmwidth, color=\pmcolor] (B5)--(2.5,12); 
\draw[line width=\pmwidth, color=\pmcolor] (W4)--(12,6.3334); \draw[line width=\pmwidth, color=\pmcolor] (B3)--(0,6.3334); 
\dimerexampleB
\end{tikzpicture}
} }; 

\node (PM21) at (7.5,-5.7) 
{\scalebox{0.12}{
\begin{tikzpicture}
\draw[line width=\pmwidth, color=\pmcolor] (B3)--(W3); \draw[line width=\pmwidth, color=\pmcolor] (B5)--(W5); 
\draw[line width=\pmwidth, color=\pmcolor] (B2)--(W4); \draw[line width=\pmwidth, color=\pmcolor] (B4)--(W6); 
\draw[line width=\pmwidth, color=\pmcolor] (W1)--(5.5,0); \draw[line width=\pmwidth, color=\pmcolor] (B6)--(5.5,12); 
\draw[line width=\pmwidth, color=\pmcolor] (W2)--(12,2.3334); \draw[line width=\pmwidth, color=\pmcolor] (B1)--(0,2.3334); 
\dimerexampleB
\end{tikzpicture}
} }; 

\draw[line width=0.02cm] (PM7)--(PM9) node[fill=white,inner sep=0.5pt, circle, midway,xshift=0cm,yshift=0cm] {\small $0$} ; 
\draw[line width=0.02cm] (PM7)--(PM10) node[fill=white,inner sep=0.5pt, circle, midway,xshift=0cm,yshift=0cm] {\small $5$} ; 
\draw[line width=0.02cm] (PM9)--(PM12) node[fill=white,inner sep=0.5pt, circle, midway,xshift=0cm,yshift=0cm] {\small $5$} ; 
\draw[line width=0.02cm] (PM10)--(PM12) node[fill=white,inner sep=0.5pt, circle, midway,xshift=0cm,yshift=0cm] {\small $0$} ; 
\draw[line width=0.02cm] (PM9)--(PM13) node[fill=white,inner sep=0.5pt, circle, midway,xshift=0cm,yshift=0cm] {\small $4$} ; 
\draw[line width=0.02cm] (PM9)--(PM11) node[fill=white,inner sep=0.5pt, circle, midway,xshift=0cm,yshift=0cm] {\small $1$} ; 
\draw[line width=0.02cm] (PM12)--(PM16) node[fill=white,inner sep=0.5pt, circle, midway,xshift=0cm,yshift=0cm] {\small $1$} ; 
\draw[line width=0.02cm] (PM11)--(PM16) node[fill=white,inner sep=0.5pt, circle, midway,xshift=0cm,yshift=0cm] {\small $5$} ; 
\draw[line width=0.02cm] (PM14)--(PM13) node[fill=white,inner sep=0.5pt, circle, midway,xshift=0cm,yshift=0cm] {\small $1$} ; 
\draw[line width=0.02cm] (PM14)--(PM11) node[fill=white,inner sep=0.5pt, circle, midway,xshift=0cm,yshift=0cm] {\small $4$} ; 
\draw[line width=0.02cm] (PM14)--(PM15) node[fill=white,inner sep=0.5pt, circle, midway,xshift=0cm,yshift=0cm] {\small $2$} ; 
\draw[line width=0.02cm] (PM14)--(PM18) node[fill=white,inner sep=0.5pt, circle, midway,xshift=0cm,yshift=0cm] {\small $7$} ; 
\draw[line width=0.02cm] (PM17)--(PM11) node[fill=white,inner sep=0.5pt, circle, midway,xshift=0cm,yshift=0cm] {\small $7$} ; 
\draw[line width=0.02cm] (PM17)--(PM18) node[fill=white,inner sep=0.5pt, circle, midway,xshift=0cm,yshift=0cm] {\small $4$} ; 
\draw[line width=0.02cm] (PM19)--(PM15) node[fill=white,inner sep=0.5pt, circle, midway,xshift=0cm,yshift=0cm] {\small $7$} ; 
\draw[line width=0.02cm] (PM19)--(PM18) node[fill=white,inner sep=0.5pt, circle, midway,xshift=0cm,yshift=0cm] {\small $2$} ; 

\draw[line width=0.02cm, bend left] (PM10) to node[fill=white,inner sep=0.5pt, circle, midway,xshift=0cm,yshift=0cm] {\small $3$} (PM21) ; 
\draw[line width=0.02cm, bend right] (PM19) to node[fill=white,inner sep=0.5pt, circle, midway,xshift=0cm,yshift=0cm] {\small $6$} (PM21) ; 
\draw[line width=0.02cm, bend left=35] (PM19) to node[fill=white,inner sep=0.5pt, circle, midway,xshift=0cm,yshift=0cm] {\small $3$} (PM20) ; 
\draw[line width=0.02cm, bend right=35] (PM10) to node[fill=white,inner sep=0.5pt, circle, midway,xshift=0cm,yshift=0cm] {\small $6$} (PM20) ; 
\end{tikzpicture}
}
\end{center}
\caption{The exchange graph of mutations of perfect matchings}
\label{ex_graph_pm}
\end{figure}
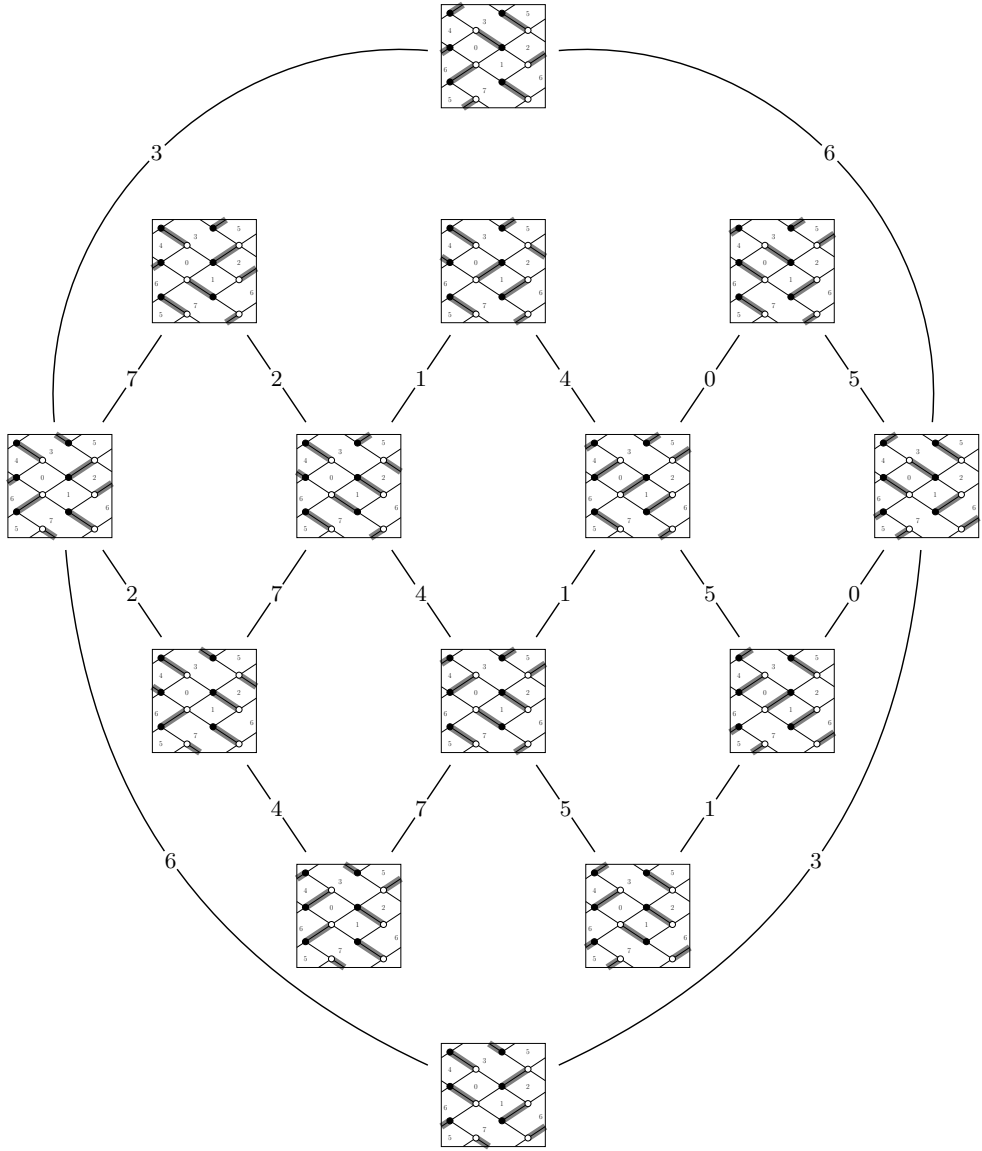
\end{example}

\section{\bf Derived equivalences of $2$-representation infinite algebras} 
\label{sec_derived_eq_2rep}

\subsection{Derived equivalences obtained via mutations of perfect matchings} 

We then turn our attention to the bounded derived category $\calD^\rmb(\mc\calA_{\sfD})$ arising from an internal perfect matching $\sfD$. 
If the PM polygon $\Delta$ of a consistent dimer model contains an interior lattice point, then we can define the toric weak Fano Deligne-Mumford stack $\calX_{\bf \Sigma}$ associated to the stacky fan ${\bf \Sigma}$ arising from $\Delta$ (see e.g., \cite{IU5}). 
Moreover, thanks to \cite[Theorem~7.2]{IU5}, we see that the $2$-representation infinite algebra $\calA_{\sfD}$ is isomorphic to the endomorphism ring of a tilting object in $\calD^\rmb(\coh\calX_{\bf \Sigma})$, where $\sfD$ is an internal perfect matching corresponding to the interior lattice point of $\Delta$ fixed as the origin. 
Thus, such a tilting object induces an equivalence $\calD^\rmb(\coh\calX_{\bf \Sigma})\cong\calD^\rmb(\mc\calA_{\sfD})$, and we have the following theorem 
which is pointed out in \cite[Remark~7.3]{IU5}.

\begin{theorem}
\label{derivedequ_pm}
Let $\Gamma$ be a consistent  dimer model, and $\Delta$ be the PM polygon of $\Gamma$. 
Let $D_i,D_j$ be internal perfect matchings of $\Gamma$ corresponding to the same interior lattice point of $\Delta$. 
Then, we have an equivalence $\calD^\rmb(\mc\calA_{\sfD_i})\cong\calD^\rmb(\mc\calA_{\sfD_j})$. 
\end{theorem}

For example, since all perfect matchings appearing in Figure~\ref{ex_graph_pm} correspond to the same interior lattice point, the associated $2$-representation infinite algebras are derived equivalent. 

In the rest of this section, we give another proof of this theorem using Theorem~\ref{pm_mutation_equiv} and $2$-APR tilting modules. 
We first fix a positive integer $n$. In order to control equivalences of derived categories, tilting theory is so important. 
Thus, we first recall the definition of a tilting module. 

\begin{definition}
\label{def_tilting}
For an algebra $\Lambda$, we say that a $\Lambda$-module $T$ is \emph{tilting} if it satisfies the conditions: 
\begin{itemize}
\item[(T1)] $\projdim_\Lambda T\le n$, 
\item[(T2)] $\Ext^i_\Lambda(T,T)=0$ for any $i>0$, 
\item[(T3)] there is an exact sequence $0\rightarrow\Lambda\rightarrow T_0\rightarrow\cdots\rightarrow T_n\rightarrow 0$ such that $T_i\in\add_\Lambda T$. 
\end{itemize}
\end{definition}

A tilting $\Lambda$-module $T$ induces a derived equivalence $\calD^\rmb(\mc\Lambda)\cong\calD^\rmb(\mc\End_\Lambda(T))$ (see e.g., \cite{Hap,Ric}). 
In this paper, we mainly consider a special tilting module called \emph{higher APR tilting module}, which is named after the work due to Auslander-Platzeck-Reiten \cite{APR}. 
In order to define this tilting module, we prepare some terminologies from higher dimensional Auslander-Reiten theory developed in \cite{Iya1}. 
Let $\Lambda$ be a finite dimensional algebra. 
We denote by $\Tr:\underline{\mc}\Lambda\xrightarrow{\simeq}\underline{\mc}\Lambda^{\rm op}$ the\emph{ Auslander-Bridger transpose} (see e.g., \cite{ASS,ARS}), 
and denote by $\Omega:\underline{\mc}\Lambda\rightarrow\underline{\mc}\Lambda$ (resp. $\Omega^-:\overline{\mc}\Lambda\rightarrow\overline{\mc}\Lambda$) the syzygy (resp. cosyzygy) functor. 
We consider functors: 
\[
\tau_n\coloneqq\rmD\Tr\Omega^{n-1}:\underline{\mc}\Lambda\xrightarrow{\Omega^{n-1}}\underline{\mc}\Lambda\xrightarrow{\Tr}\underline{\mc}\Lambda^{\rm op}\xrightarrow{\rmD}\overline{\mc}\Lambda
\]
\[
\tau_n^-\coloneqq\Tr\rmD\Omega^{-(n-1)}:\overline{\mc}\Lambda\xrightarrow{\Omega^{-(n-1)}}\overline{\mc}\Lambda\xrightarrow{\rmD}\overline{\mc}\Lambda^{\rm op}\xrightarrow{\Tr}\underline{\mc}\Lambda. 
\]
These $\tau_n$ and $\tau_n^-$ are called the \emph{$n$-Auslander-Reiten translations}. 
We remark that these are the usual Auslander-Reiten translations (see e.g., \cite{ASS,ARS}) when $n=1$. 

\medskip

In the following, we assume that $\gldim\Lambda\le n$. Let $P$ be a simple projective $\Lambda$-module 
satisfying $\Ext^i_\Lambda(\rmD\Lambda,P)=0$ for any $0\le i\le n-1$. 
We decompose $\Lambda=P\oplus U$. 
We then have the following exact sequence (see \cite[Proof of Theorem~3.2]{IO}): 
$$
0\rightarrow P\xrightarrow{f_0} U_1\xrightarrow{f_1} \cdots\xrightarrow{f_{n-1}} U_n\xrightarrow{f_n}\tau_n^-P\rightarrow 0, 
$$
where $U_i\in\add_\Lambda U$. Using this sequence, we define an $m$-APR tilting module as follows. 

\begin{definition}[{see \cite{MY,IO}}]
\label{mAPR_def}
With the notation as above, let $K_m\coloneqq \Image f_m$ for $m=0,\dots,n$. (We remark that $K_0=P$ and $K_n=\tau_n^-P$.) 
We call $U\oplus K_m$ the \emph{$m$-APR tilting module} associated with $P$. 
\end{definition}

\begin{remark}
\label{mAPR_rem}
\begin{enumerate}[(1)]
\item By the arguments in \cite[Proof of Theorem~3.2]{IO}, we see that $U\oplus K_m$ is actually a tilting $\Lambda$-module of projective dimension $m$. 
\item If $n=1$, then the $1$-APR tilting module $U\oplus K_1$ is simply the APR tilting module studied in \cite{APR}. 
\item In this paper,  we are interested in $n$-representation infinite algebras, especially the case $n=2$. 
For an $n$-representation infinite algebra $\Lambda$ and a simple projective $\Lambda$-module $P$, we always have the condition $\Ext^i_\Lambda(\rmD\Lambda,P)=0$ for any $0\le i\le n-1$. 
Thus, we can define the $m$-APR tilting module for any simple projective module. 
Furthermore, $m$-APR tilting preserves $n$-representation infiniteness \cite[Theorem~3.1]{MY}. 
\end{enumerate}
\end{remark}

We are now ready to proof Theorem~\ref{derivedequ_pm}. 

\begin{proof}[Proof of Theorem~\ref{derivedequ_pm}]
Since $D_i$ and $D_j$ correspond to the same interior lattice point, they are mutation equivalent (see Theorem~\ref{pm_mutation_equiv}). 
Since $\sfD_i$ is an internal perfect matching, $Q_{\sfD_i}$ is acyclic, thus there is a strict source of $(Q,\sfD_i)$. 
We denote this strict source by $k\in Q_0$. Then, we may assume that $\lambda_k^+(\sfD_i)=\sfD_j$. 
Then, we decompose $\calA_{\sfD_i}=P_k\oplus U$ where $P_k$ is a simple projective $\calA_{\sfD_i}$-module corresponding to $k$. 
Since $\calA_{\sfD_i}$ is a $2$-representation infinite algebra, we have the $2$-APR tilting $\calA_{\sfD_i}$-module $T_k\coloneqq (\tau^-_2P_k)\oplus U$  associated with $P_k$ (see Remark~\ref{mAPR_rem}). 
Furthermore, by \cite[Theorem~3.11]{IO} (see also Observation~\ref{obs_mutationPM}), we have $\End_{\calA_{\sfD_i}}(T_k)\cong\calA_{\sfD_j}$. 
Since $T_k$ is a tilting module, this induces $\calD^\rmb(\mc\calA_{\sfD_i})\cong\calD^\rmb(\mc\calA_{\sfD_j})$. 
\end{proof}

On the other hand, even if $2$-representation infinite algebras $\calA_{\sfD_i}$ and $\calA_{\sfD_j}$ associated with internal perfect matchings $\sfD_i, \sfD_j$ are derived equivalent, they do not necessarily correspond to the same interior lattice point as shown in the following example. 

\begin{example}
We consider the quotient singularity $R$ associated with the finite abelian group $G=\frac{1}{7}(1,2,4)\subset\SL(3,\kk)$, which is generated by ${\rm diag}(\omega,\omega^2,\omega^4)$ with $\omega^7=1$. 
In this case, a consistent dimer model associated with $R$ is unique and hexagonal (see \cite{UY,IN}). 
In particular, it takes the form as shown in the leftmost figure below, where the red line stands for the fundamental domain. 
We remark that the associated quiver (the center of the following figure) coincides with the McKay quiver of $G$. 
The rightmost figure is the PM polygon $\Delta$, and this has three interior lattice points (these correspond to elements in $G$ with ``age $1$", see \cite{ItoReid}). 
Furthermore, the triangulation given in the rightmost figure induces the crepant resolution of $\Spec R$ which is isomorphic to $\calM_\theta(Q)$ with the generic parameter $\theta=(-6,1,\dots,1)$, and this also coincides with the $G$-Hilbert scheme. 

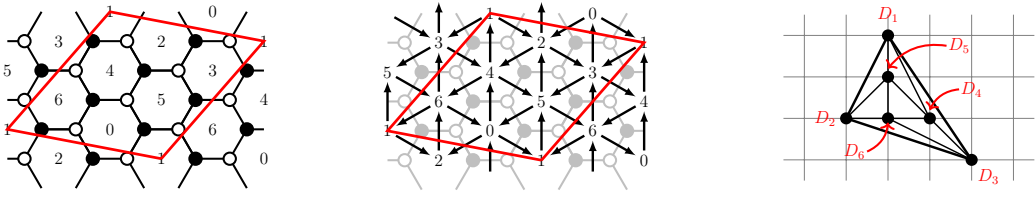
\newcommand{\dimerquotient}{
\newcommand{\edgewidth}{0.06cm} 
\newcommand{\nodewidth}{0.06cm} 
\newcommand{\noderad}{0.18} 

\node (V5) at (0,0) {{\LARGE$5$}}; 
\coordinate (e05) at (0:1cm); \coordinate (e15) at (60:1cm); \coordinate (e25) at (120:1cm); 
\coordinate (e35) at (180:1cm); \coordinate (e45) at (240:1cm); \coordinate (e55) at (300:1cm); 

\path (e15) ++(1, 0) coordinate (V3); \node at (V3) {{\LARGE$3$}}; 
\path (V3) ++(0:1cm) coordinate (e03); \path (V3) ++(60:1cm) coordinate (e13); \path (V3) ++(120:1cm) coordinate (e23);
\path (V3) ++(180:1cm) coordinate (e33); \path (V3) ++(240:1cm) coordinate (e43); \path (V3) ++(300:1cm) coordinate (e53);

\path (e55) ++(1, 0) coordinate (V6); \node at (V6) {{\LARGE$6$}}; 
\path (V6) ++(0:1cm) coordinate (e06); \path (V6) ++(60:1cm) coordinate (e16); \path (V6) ++(120:1cm) coordinate (e26);
\path (V6) ++(180:1cm) coordinate (e36); \path (V6) ++(240:1cm) coordinate (e46); \path (V6) ++(300:1cm) coordinate (e56);

\path (e23) ++(-1, 0) coordinate (V2); \node at (V2) {{\LARGE$2$}}; 
\path (V2) ++(0:1cm) coordinate (e02); \path (V2) ++(60:1cm) coordinate (e12); \path (V2) ++(120:1cm) coordinate (e22);
\path (V2) ++(180:1cm) coordinate (e32); \path (V2) ++(240:1cm) coordinate (e42); \path (V2) ++(300:1cm) coordinate (e52);

\path (e46) ++(-1, 0) coordinate (V1); \node at (V1) {{\LARGE$1$}};
\path (V1) ++(0:1cm) coordinate (e01); \path (V1) ++(60:1cm) coordinate (e11); \path (V1) ++(120:1cm) coordinate (e21);
\path (V1) ++(180:1cm) coordinate (e31); \path (V1) ++(240:1cm) coordinate (e41); \path (V1) ++(300:1cm) coordinate (e51);

\path (e25) ++(-1, 0) coordinate (V4); \node at (V4) {{\LARGE$4$}};
\path (V4) ++(0:1cm) coordinate (e04); \path (V4) ++(60:1cm) coordinate (e14); \path (V4) ++(120:1cm) coordinate (e24);
\path (V4) ++(180:1cm) coordinate (e34); \path (V4) ++(240:1cm) coordinate (e44); \path (V4) ++(300:1cm) coordinate (e54);

\path (e45) ++(-1, 0) coordinate (V0); \node at (V0) {{\LARGE$0$}};
\path (V0) ++(0:1cm) coordinate (e00); \path (V0) ++(60:1cm) coordinate (e10); \path (V0) ++(120:1cm) coordinate (e20);
\path (V0) ++(180:1cm) coordinate (e30); \path (V0) ++(240:1cm) coordinate (e40); \path (V0) ++(300:1cm) coordinate (e50);

\path (e13) ++(1, 0) coordinate (V1a); \node at (V1a) {{\LARGE$1$}};
\path (V1a) ++(0:1cm) coordinate (e01a); \path (V1a) ++(60:1cm) coordinate (e11a); \path (V1a) ++(120:1cm) coordinate (e21a);
\path (V1a) ++(180:1cm) coordinate (e31a); \path (V1a) ++(240:1cm) coordinate (e41a); \path (V1a) ++(300:1cm) coordinate (e51a);

\path (e53) ++(1, 0) coordinate (V4a); \node at (V4a) {{\LARGE$4$}};
\path (V4a) ++(0:1cm) coordinate (e04a); \path (V4a) ++(60:1cm) coordinate (e14a); \path (V4a) ++(120:1cm) coordinate (e24a);
\path (V4a) ++(180:1cm) coordinate (e34a); \path (V4a) ++(240:1cm) coordinate (e44a); \path (V4a) ++(300:1cm) coordinate (e54a);

\path (e56) ++(1, 0) coordinate (V0a); \node at (V0a) {{\LARGE$0$}};
\path (V0a) ++(0:1cm) coordinate (e00a); \path (V0a) ++(60:1cm) coordinate (e10a); \path (V0a) ++(120:1cm) coordinate (e20a);
\path (V0a) ++(180:1cm) coordinate (e30a); \path (V0a) ++(240:1cm) coordinate (e40a); \path (V0a) ++(300:1cm) coordinate (e50a);

\path (e24) ++(-1, 0) coordinate (V3a); \node at (V3a) {{\LARGE$3$}};
\path (V3a) ++(0:1cm) coordinate (e03a); \path (V3a) ++(60:1cm) coordinate (e13a); \path (V3a) ++(120:1cm) coordinate (e23a);
\path (V3a) ++(180:1cm) coordinate (e33a); \path (V3a) ++(240:1cm) coordinate (e43a); \path (V3a) ++(300:1cm) coordinate (e53a);

\path (e44) ++(-1, 0) coordinate (V6a); \node at (V6a) {{\LARGE$6$}};
\path (V6a) ++(0:1cm) coordinate (e06a); \path (V6a) ++(60:1cm) coordinate (e16a); \path (V6a) ++(120:1cm) coordinate (e26a);
\path (V6a) ++(180:1cm) coordinate (e36a); \path (V6a) ++(240:1cm) coordinate (e46a); \path (V6a) ++(300:1cm) coordinate (e56a);

\path (e40) ++(-1, 0) coordinate (V2a); \node at (V2a) {{\LARGE$2$}};
\path (V2a) ++(0:1cm) coordinate (e02a); \path (V2a) ++(60:1cm) coordinate (e12a); \path (V2a) ++(120:1cm) coordinate (e22a);
\path (V2a) ++(180:1cm) coordinate (e32a); \path (V2a) ++(240:1cm) coordinate (e42a); \path (V2a) ++(300:1cm) coordinate (e52a);

\draw [line width=\edgewidth]  (e00)--(e10)--(e20)--(e30)--(e40)--(e50)--(e00) ;
\draw [line width=\edgewidth]  (e51)--(e01)--(e11)--(e21)--(e31)--(e41) ;
\draw [line width=\edgewidth]  (e02)--(e12);
\draw [line width=\edgewidth]  (e22)--(e32)--(e42)--(e52)--(e02) ;
\draw [line width=\edgewidth]  (e03)--(e13)--(e23)--(e33)--(e43)--(e53)--(e03) ; 
\draw [line width=\edgewidth]  (e04)--(e14)--(e24)--(e34)--(e44)--(e54)--(e04) ;
\draw [line width=\edgewidth]  (e05)--(e15)--(e25)--(e35)--(e45)--(e55)--(e05) ;
\draw [line width=\edgewidth]  (e06)--(e16)--(e26)--(e36)--(e46)--(e56)--(e06) ;

\draw [line width=\edgewidth]  (e21a)--(e31a) ;
\draw [line width=\edgewidth]  (e30a)--(e40a) ;
\draw [line width=\edgewidth]  (e03a)--(e13a) ;
\draw [line width=\edgewidth]  (e23a)--(e33a)--(e43a)--(e53a)--(e03a) ;
\draw [line width=\edgewidth]  (e06a)--(e16a)--(e26a)--(e36a)--(e46a)--(e56a)--(e06a) ;
\draw [line width=\edgewidth]  (e52a)--(e02a)--(e12a)--(e22a)--(e32a)--(e42a) ; 

\draw [line width=\edgewidth] (e33a)-- +(-0.5,0) ; \draw [line width=\edgewidth] (e36a)-- +(-0.5,0) ; \draw [line width=\edgewidth] (e32a)-- +(-0.5,0) ;
\draw [line width=\edgewidth] (e03)-- +(0.5,0) ; \draw [line width=\edgewidth] (e06)-- +(0.5,0) ; 

\path (e22) ++(-1, 0) coordinate (V1b); \node at (V1b) {{\LARGE$1$}};
\path (e22a) ++(-1, 0) coordinate (V1c); \node at (V1c) {{\LARGE$1$}};
\path (e26a) ++(-1, 0) coordinate (V5a); \node at (V5a) {{\LARGE$5$}};
\path (e12) ++(1,0) coordinate (V0b); \node (n0b) at (V0b) {{\LARGE$0$}};

\filldraw [line width=\edgewidth, fill=black] (e05) circle [radius=\noderad]; \filldraw [line width=\edgewidth, fill=black] (e25) circle [radius=\noderad]; 
\filldraw [line width=\edgewidth, fill=black] (e45) circle [radius=\noderad];
\filldraw [line width=\edgewidth, fill=white] (e15) circle [radius=\noderad]; \filldraw [line width=\edgewidth, fill=white] (e35) circle [radius=\noderad]; 
\filldraw [line width=\edgewidth, fill=white] (e55) circle [radius=\noderad]; 
\filldraw [line width=\edgewidth, fill=black] (e03) circle [radius=\noderad]; \filldraw [line width=\edgewidth, fill=black] (e23) circle [radius=\noderad]; 
\filldraw [line width=\edgewidth, fill=white] (e13) circle [radius=\noderad]; \filldraw [line width=\edgewidth, fill=white] (e53) circle [radius=\noderad]; 
\filldraw [line width=\edgewidth, fill=black] (e24) circle [radius=\noderad]; \filldraw [line width=\edgewidth, fill=black] (e44) circle [radius=\noderad];
\filldraw [line width=\edgewidth, fill=white] (e14) circle [radius=\noderad]; \filldraw [line width=\edgewidth, fill=white] (e34) circle [radius=\noderad]; 
\filldraw [line width=\edgewidth, fill=black] (e40) circle [radius=\noderad];
\filldraw [line width=\edgewidth, fill=white] (e30) circle [radius=\noderad]; \filldraw [line width=\edgewidth, fill=white] (e50) circle [radius=\noderad]; 
\filldraw [line width=\edgewidth, fill=black] (e06) circle [radius=\noderad]; \filldraw [line width=\edgewidth, fill=black] (e46) circle [radius=\noderad]; 
\filldraw [line width=\edgewidth, fill=white] (e56) circle [radius=\noderad]; 
\filldraw [line width=\edgewidth, fill=black] (e26a) circle [radius=\noderad]; \filldraw [line width=\edgewidth, fill=black] (e46a) circle [radius=\noderad]; 
\filldraw [line width=\edgewidth, fill=white] (e36a) circle [radius=\noderad]; 
\filldraw [line width=\edgewidth, fill=white] (e33a) circle [radius=\noderad]; \filldraw [line width=\edgewidth, fill=white] (e32a) circle [radius=\noderad]; 

\draw[line width=0.08cm, red]  (V1)--(V1a)--(V1b)--(V1c)--(V1) ; 
}

\begin{center}
\begin{tikzpicture}

\node at (-5,0) {
\scalebox{0.45}{
\begin{tikzpicture}
\dimerquotient
\end{tikzpicture}
} } ; 

\node at (0,0) {
\scalebox{0.45}{
\begin{tikzpicture}[sarrow/.style={black, -latex, very thick}, ssarrow/.style={black, latex-, very thick}]
\newcommand{\edgewidth}{0.06cm} 
\newcommand{\nodewidth}{0.06cm} 
\newcommand{\noderad}{0.18} 
\newcommand{\arrowwidth}{0.07cm} 

\coordinate (e05) at (0:1cm); \coordinate (e15) at (60:1cm); \coordinate (e25) at (120:1cm); 
\coordinate (e35) at (180:1cm); \coordinate (e45) at (240:1cm); \coordinate (e55) at (300:1cm); 
\node (V5) at (0,0) {}; 
\node (n5) at (V5) {{\LARGE$5$}}; 
\coordinate (e05) at (0:1cm); \coordinate (e15) at (60:1cm); \coordinate (e25) at (120:1cm); 
\coordinate (e35) at (180:1cm); \coordinate (e45) at (240:1cm); \coordinate (e55) at (300:1cm); 
\path (e15) ++(1, 0) coordinate (V3); \node (n3) at (V3) {{\LARGE$3$}}; 
\path (V3) ++(0:1cm) coordinate (e03); \path (V3) ++(60:1cm) coordinate (e13); \path (V3) ++(120:1cm) coordinate (e23);
\path (V3) ++(180:1cm) coordinate (e33); \path (V3) ++(240:1cm) coordinate (e43); \path (V3) ++(300:1cm) coordinate (e53);

\path (e55) ++(1, 0) coordinate (V6); \node (n6) at (V6) {{\LARGE$6$}}; 
\path (V6) ++(0:1cm) coordinate (e06); \path (V6) ++(60:1cm) coordinate (e16); \path (V6) ++(120:1cm) coordinate (e26);
\path (V6) ++(180:1cm) coordinate (e36); \path (V6) ++(240:1cm) coordinate (e46); \path (V6) ++(300:1cm) coordinate (e56);

\path (e23) ++(-1, 0) coordinate (V2); \node (n2) at (V2) {{\LARGE$2$}}; 
\path (V2) ++(0:1cm) coordinate (e02); \path (V2) ++(60:1cm) coordinate (e12); \path (V2) ++(120:1cm) coordinate (e22);
\path (V2) ++(180:1cm) coordinate (e32); \path (V2) ++(240:1cm) coordinate (e42); \path (V2) ++(300:1cm) coordinate (e52);

\path (e46) ++(-1, 0) coordinate (V1); \node (n1) at (V1) {{\LARGE$1$}};
\path (V1) ++(0:1cm) coordinate (e01); \path (V1) ++(60:1cm) coordinate (e11); \path (V1) ++(120:1cm) coordinate (e21);
\path (V1) ++(180:1cm) coordinate (e31); \path (V1) ++(240:1cm) coordinate (e41); \path (V1) ++(300:1cm) coordinate (e51);

\path (e25) ++(-1, 0) coordinate (V4); \node (n4) at (V4) {{\LARGE$4$}};
\path (V4) ++(0:1cm) coordinate (e04); \path (V4) ++(60:1cm) coordinate (e14); \path (V4) ++(120:1cm) coordinate (e24);
\path (V4) ++(180:1cm) coordinate (e34); \path (V4) ++(240:1cm) coordinate (e44); \path (V4) ++(300:1cm) coordinate (e54);

\path (e45) ++(-1, 0) coordinate (V0); \node (n0) at (V0) {{\LARGE$0$}};
\path (V0) ++(0:1cm) coordinate (e00); \path (V0) ++(60:1cm) coordinate (e10); \path (V0) ++(120:1cm) coordinate (e20);
\path (V0) ++(180:1cm) coordinate (e30); \path (V0) ++(240:1cm) coordinate (e40); \path (V0) ++(300:1cm) coordinate (e50);

\path (e13) ++(1, 0) coordinate (V1a); \node (n1a) at (V1a) {{\LARGE$1$}};
\path (V1a) ++(0:1cm) coordinate (e01a); \path (V1a) ++(60:1cm) coordinate (e11a); \path (V1a) ++(120:1cm) coordinate (e21a);
\path (V1a) ++(180:1cm) coordinate (e31a); \path (V1a) ++(240:1cm) coordinate (e41a); \path (V1a) ++(300:1cm) coordinate (e51a);

\path (e53) ++(1, 0) coordinate (V4a); \node (n4a) at (V4a) {{\LARGE$4$}};
\path (V4a) ++(0:1cm) coordinate (e04a); \path (V4a) ++(60:1cm) coordinate (e14a); \path (V4a) ++(120:1cm) coordinate (e24a);
\path (V4a) ++(180:1cm) coordinate (e34a); \path (V4a) ++(240:1cm) coordinate (e44a); \path (V4a) ++(300:1cm) coordinate (e54a);

\path (e56) ++(1, 0) coordinate (V0a); \node (n0a) at (V0a) {{\LARGE$0$}};
\path (V0a) ++(0:1cm) coordinate (e00a); \path (V0a) ++(60:1cm) coordinate (e10a); \path (V0a) ++(120:1cm) coordinate (e20a);
\path (V0a) ++(180:1cm) coordinate (e30a); \path (V0a) ++(240:1cm) coordinate (e40a); \path (V0a) ++(300:1cm) coordinate (e50a);

\path (e24) ++(-1, 0) coordinate (V3a); \node (n3a) at (V3a) {{\LARGE$3$}};
\path (V3a) ++(0:1cm) coordinate (e03a); \path (V3a) ++(60:1cm) coordinate (e13a); \path (V3a) ++(120:1cm) coordinate (e23a);
\path (V3a) ++(180:1cm) coordinate (e33a); \path (V3a) ++(240:1cm) coordinate (e43a); \path (V3a) ++(300:1cm) coordinate (e53a);

\path (e44) ++(-1, 0) coordinate (V6a); \node (n6a) at (V6a) {{\LARGE$6$}};
\path (V6a) ++(0:1cm) coordinate (e06a); \path (V6a) ++(60:1cm) coordinate (e16a); \path (V6a) ++(120:1cm) coordinate (e26a);
\path (V6a) ++(180:1cm) coordinate (e36a); \path (V6a) ++(240:1cm) coordinate (e46a); \path (V6a) ++(300:1cm) coordinate (e56a);

\path (e40) ++(-1, 0) coordinate (V2a); \node (n2a) at (V2a) {{\LARGE$2$}};
\path (V2a) ++(0:1cm) coordinate (e02a); \path (V2a) ++(60:1cm) coordinate (e12a); \path (V2a) ++(120:1cm) coordinate (e22a);
\path (V2a) ++(180:1cm) coordinate (e32a); \path (V2a) ++(240:1cm) coordinate (e42a); \path (V2a) ++(300:1cm) coordinate (e52a);

\draw [line width=\edgewidth, lightgray]  (e00)--(e10)--(e20)--(e30)--(e40)--(e50)--(e00) ;
\draw [line width=\edgewidth, lightgray]  (e51)--(e01)--(e11)--(e21)--(e31)--(e41) ;
\draw [line width=\edgewidth, lightgray]  (e02)--(e12);
\draw [line width=\edgewidth, lightgray]  (e22)--(e32)--(e42)--(e52)--(e02) ;
\draw [line width=\edgewidth, lightgray]  (e03)--(e13)--(e23)--(e33)--(e43)--(e53)--(e03) ; 
\draw [line width=\edgewidth, lightgray]  (e04)--(e14)--(e24)--(e34)--(e44)--(e54)--(e04) ;
\draw [line width=\edgewidth, lightgray]  (e05)--(e15)--(e25)--(e35)--(e45)--(e55)--(e05) ;
\draw [line width=\edgewidth, lightgray]  (e06)--(e16)--(e26)--(e36)--(e46)--(e56)--(e06) ;

\draw [line width=\edgewidth, lightgray]  (e21a)--(e31a) ;
\draw [line width=\edgewidth, lightgray]  (e30a)--(e40a) ;
\draw [line width=\edgewidth, lightgray]  (e03a)--(e13a) ;
\draw [line width=\edgewidth, lightgray]  (e23a)--(e33a)--(e43a)--(e53a)--(e03a) ;
\draw [line width=\edgewidth, lightgray]  (e06a)--(e16a)--(e26a)--(e36a)--(e46a)--(e56a)--(e06a) ;
\draw [line width=\edgewidth, lightgray]  (e52a)--(e02a)--(e12a)--(e22a)--(e32a)--(e42a) ; 

\draw [line width=\edgewidth, lightgray] (e33a)-- +(-0.5,0) ; \draw [line width=\edgewidth, lightgray] (e36a)-- +(-0.5,0) ; \draw [line width=\edgewidth, lightgray] (e32a)-- +(-0.5,0) ;
\draw [line width=\edgewidth, lightgray] (e03)-- +(0.5,0) ; \draw [line width=\edgewidth, lightgray] (e06)-- +(0.5,0) ; 

\path (e22) ++(-1, 0) coordinate (V1b); \node (n1b) at (V1b) {{\LARGE$1$}};
\path (e22a) ++(-1, 0) coordinate (V1c); \node (n1c) at (V1c) {{\LARGE$1$}};
\path (e26a) ++(-1, 0) coordinate (V5a); \node (n5a) at (V5a) {{\LARGE$5$}};

\filldraw [line width=\nodewidth, draw=lightgray, fill=lightgray] (e05) circle [radius=\noderad]; \filldraw [line width=\nodewidth, draw=lightgray, fill=lightgray] (e25) circle [radius=\noderad]; 
\filldraw [line width=\nodewidth, draw=lightgray, fill=lightgray] (e45) circle [radius=\noderad];
\filldraw [line width=\nodewidth, draw=lightgray, fill=white] (e15) circle [radius=\noderad]; \filldraw [line width=\nodewidth, draw=lightgray, fill=white] (e35) circle [radius=\noderad]; 
\filldraw [line width=\nodewidth, draw=lightgray, fill=white] (e55) circle [radius=\noderad]; 
\filldraw [line width=\nodewidth, draw=lightgray, fill=lightgray] (e03) circle [radius=\noderad]; \filldraw [line width=\nodewidth, draw=lightgray, fill=lightgray] (e23) circle [radius=\noderad]; 
\filldraw [line width=\nodewidth, draw=lightgray, fill=white] (e13) circle [radius=\noderad]; \filldraw [line width=\nodewidth, draw=lightgray, fill=white] (e53) circle [radius=\noderad]; 
\filldraw [line width=\nodewidth, draw=lightgray, fill=lightgray] (e24) circle [radius=\noderad]; \filldraw [line width=\nodewidth, draw=lightgray, fill=lightgray] (e44) circle [radius=\noderad];
\filldraw [line width=\nodewidth, draw=lightgray, fill=white] (e14) circle [radius=\noderad]; \filldraw [line width=\nodewidth, draw=lightgray, fill=white] (e34) circle [radius=\noderad]; 
\filldraw [line width=\nodewidth, draw=lightgray, fill=lightgray] (e40) circle [radius=\noderad];
\filldraw [line width=\nodewidth, draw=lightgray, fill=white] (e30) circle [radius=\noderad]; \filldraw [line width=\nodewidth, draw=lightgray, fill=white] (e50) circle [radius=\noderad]; 
\filldraw [line width=\nodewidth, draw=lightgray, fill=lightgray] (e06) circle [radius=\noderad]; \filldraw [line width=\nodewidth, draw=lightgray, fill=lightgray] (e46) circle [radius=\noderad]; 
\filldraw [line width=\nodewidth, draw=lightgray, fill=white] (e56) circle [radius=\noderad]; 
\filldraw [line width=\nodewidth, draw=lightgray, fill=lightgray] (e26a) circle [radius=\noderad]; \filldraw [line width=\nodewidth, draw=lightgray, fill=lightgray] (e46a) circle [radius=\noderad]; 
\filldraw [line width=\nodewidth, draw=lightgray, fill=white] (e36a) circle [radius=\noderad]; 
\filldraw [line width=\nodewidth, draw=lightgray, fill=white] (e33a) circle [radius=\noderad]; \filldraw [line width=\nodewidth, draw=lightgray, fill=white] (e32a) circle [radius=\noderad]; 

\draw[sarrow, line width=\arrowwidth] (n0)--(n1); \draw[sarrow, line width=\arrowwidth] (n0)--(n4); \draw[sarrow, line width=\arrowwidth] (n0)--(n2a);  
\draw[sarrow, line width=\arrowwidth] (n5)--(n0); \draw[sarrow, line width=\arrowwidth] (n5)--(n2); \draw[sarrow, line width=\arrowwidth] (n5)--(n6);
\draw[sarrow, line width=\arrowwidth] (n6a)--(n0); \draw[sarrow, line width=\arrowwidth] (n6a)--(n3a); \draw[sarrow, line width=\arrowwidth] (n6a)--(n1c);
\draw[sarrow, line width=\arrowwidth] (n4)--(n5); \draw[sarrow, line width=\arrowwidth] (n4)--(n6a); \draw[sarrow, line width=\arrowwidth] (n4)--(n1b); 
\draw[sarrow, line width=\arrowwidth] (n2)--(n3); \draw[sarrow, line width=\arrowwidth] (n2)--(n4); \draw[sarrow, line width=\arrowwidth] (n1)--(n5); 
\draw[sarrow, line width=\arrowwidth] (n3)--(n5); \draw[sarrow, line width=\arrowwidth] (n3)--(n4a); 
\draw[sarrow, line width=\arrowwidth] (n6)--(n1); \draw[sarrow, line width=\arrowwidth] (n6)--(n3); \draw[sarrow, line width=\arrowwidth] (n6)--(n0a);
\draw[sarrow, line width=\arrowwidth] (n1b)--(n2); \draw[sarrow, line width=\arrowwidth] (n1b)--(n3a); \draw[sarrow, line width=\arrowwidth] (n4a)--(n6); 
\draw[sarrow, line width=\arrowwidth] (n3a)--(n4); \draw[sarrow, line width=\arrowwidth] (n3a)--(n5a); \draw[sarrow, line width=\arrowwidth] (n1c)--(n2a); 
\draw[sarrow, line width=\arrowwidth] (n1c)--(n5a); \draw[sarrow, line width=\arrowwidth] (n5a)--(n6a); \draw[sarrow, line width=\arrowwidth] (n2a)--(n6a); 
\draw[sarrow, line width=\arrowwidth] (n4a)--(n1a); \draw[sarrow, line width=\arrowwidth] (n1a)--(n3); \draw[sarrow, line width=\arrowwidth] (n0a)--(n4a); 

\path (e12) ++(1,0) coordinate (V0b); \node (n0b) at (V0b) {{\LARGE$0$}};
\draw[sarrow, line width=\arrowwidth] (n3)--(n0b); \draw[sarrow, line width=\arrowwidth] (n0b)--(n2); \draw[sarrow, line width=\arrowwidth] (n0b)--(n1a);
\draw[ssarrow, line width=\arrowwidth] (n6)--++(0,-1.2); \draw[ssarrow, line width=\arrowwidth] (n0)--++(0,-1.2); \draw[sarrow, line width=\arrowwidth] (n2)--++(0,1.2); 
\draw[sarrow, line width=\arrowwidth] (n3a)--++(0,1.2); \draw[ssarrow, line width=\arrowwidth] (n3a)--++(150:1.5cm);
\draw[line width=0.08cm, red]  (V1)--(V1a)--(V1b)--(V1c)--(V1) ; 
\end{tikzpicture}
} } ;

\node at (5,0) {
\scalebox{0.55}{
\begin{tikzpicture}
\coordinate (00) at (0,0); \coordinate (10) at (1,0); \coordinate (01) at (0,1); \coordinate (-10) at (-1,0); \coordinate (1-1) at (1,-1);  
\coordinate (-20) at (-2,0); \coordinate (-11) at (-1,1); \coordinate (-12) at (-1,2); \coordinate (1-1) at (1,-1);   

\draw [step=1, gray] (-3.5,-1.5) grid (2.5,2.5);
\draw [line width=0.06cm] (1-1)--(-12)--(-20)--(1-1) ;  

\draw [line width=0.05cm, fill=black] (00) circle [radius=0.12] ; \draw [line width=0.05cm, fill=black] (-10) circle [radius=0.12] ; 
\draw [line width=0.05cm, fill=black] (-20) circle [radius=0.12] ; \draw [line width=0.05cm, fill=black] (1-1) circle [radius=0.12] ; 
\draw [line width=0.05cm, fill=black] (-11) circle [radius=0.12] ; \draw [line width=0.05cm, fill=black] (-12) circle [radius=0.12] ; 

\draw [line width=0.035cm] (-12)--(-10) ; \draw [line width=0.035cm] (-20)--(00) ; \draw [line width=0.035cm] (1-1)--(-11) ; 
\draw [line width=0.035cm] (1-1)--(-10) ; \draw [line width=0.035cm] (-20)--(-11) ; \draw [line width=0.035cm] (-12)--(00) ; 

\node [red] at (-1,2.5) {\Large$D_1$} ; \node [red] at (-2.5,0) {\Large$D_2$} ;
\node [red] at (1.4,-1.4) {\Large$D_3$} ; 
\node [red] at (1,0.7) {\Large$D_4$} ; \draw[->, line width=0.05cm,red] (0.7,0.7) to [bend right] (0,0.2);  
\node [red] at (0.7,1.75) {\Large$D_5$} ; \draw[->, line width=0.05cm,red] (0.4,1.75) to [bend right] (-1,1.2); 
\node [red] at (-1.8,-0.8) {\Large$D_6$} ; \draw[->, line width=0.05cm,red] (-1.5,-0.8) to [bend right] (-1,-0.1);  

\end{tikzpicture}
} } ; 

\end{tikzpicture}
\end{center}

The following figures are $\theta$-stable perfect matchings. 
In particular, $D_1,D_2,D_3$ are corner perfect matchings, and $D_4,D_5,D_6$ are internal ones. 
We easily see that $\calA_{\sfD_4}\cong\calA_{\sfD_5}\cong\calA_{\sfD_6}$, but these do not correspond to the same lattice point. 

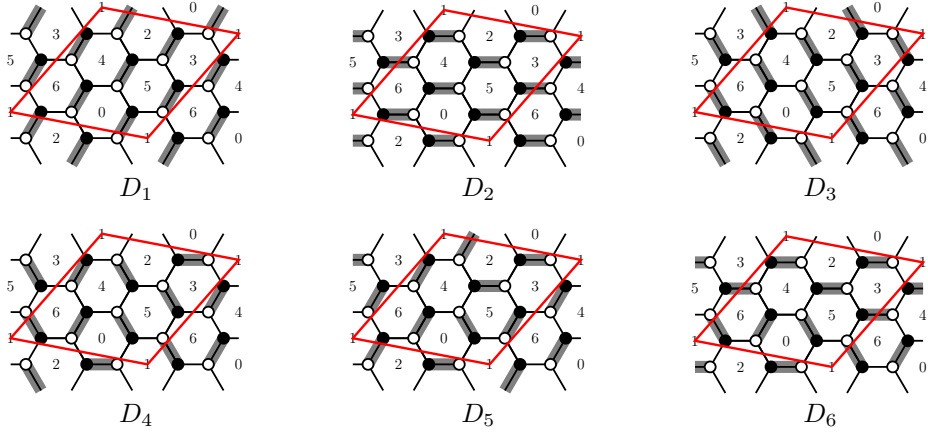
\begin{center}
\begin{tikzpicture}
\node at (0,-1.38) {$D_1$}; \node at (4.5,-1.38) {$D_2$}; \node at (9,-1.38) {$D_3$}; 
\node at (0,-4.38) {$D_4$}; \node at (4.5,-4.38) {$D_5$}; \node at (9,-4.38) {$D_6$}; 

\newcommand{\pmwidth}{0.4cm} 
\newcommand{\pmcolor}{gray} 

\node (PM1) at (0,0) {
\scalebox{0.4}{
\begin{tikzpicture}
\draw[line width=\pmwidth, color=\pmcolor] (e20)--(e30); \draw[line width=\pmwidth, color=\pmcolor] (e21)--(e31);
\draw[line width=\pmwidth, color=\pmcolor] (e22)--(e32); \draw[line width=\pmwidth, color=\pmcolor] (e23)--(e33);
\draw[line width=\pmwidth, color=\pmcolor] (e24)--(e34); \draw[line width=\pmwidth, color=\pmcolor] (e25)--(e35);
\draw[line width=\pmwidth, color=\pmcolor] (e26)--(e36); 
\draw[line width=\pmwidth, color=\pmcolor] (e20a)--(e30a); \draw[line width=\pmwidth, color=\pmcolor] (e21a)--(e31a);
\draw[line width=\pmwidth, color=\pmcolor] (e22a)--(e32a); \draw[line width=\pmwidth, color=\pmcolor] (e23a)--(e33a);
\draw[line width=\pmwidth, color=\pmcolor] (e24a)--(e34a); \draw[line width=\pmwidth, color=\pmcolor] (e26a)--(e36a); 
\draw[line width=\pmwidth, color=\pmcolor] (e51)--(e01); \draw[line width=\pmwidth, color=\pmcolor] (e52a)--(e02a); 

\dimerquotient
\end{tikzpicture}
} } ; 

\node (PM2) at (4.5,0) {
\scalebox{0.4}{
\begin{tikzpicture}
\draw[line width=\pmwidth, color=\pmcolor] (e10)--(e20); \draw[line width=\pmwidth, color=\pmcolor] (e11)--(e21);
\draw[line width=\pmwidth, color=\pmcolor] (e13)--(e23);
\draw[line width=\pmwidth, color=\pmcolor] (e14)--(e24); \draw[line width=\pmwidth, color=\pmcolor] (e15)--(e25);
\draw[line width=\pmwidth, color=\pmcolor] (e16)--(e26); 
\draw[line width=\pmwidth, color=\pmcolor] (e40)--(e50); \draw[line width=\pmwidth, color=\pmcolor] (e46)--(e56); 
\draw[line width=\pmwidth, color=\pmcolor] (e12a)--(e22a); \draw[line width=\pmwidth, color=\pmcolor] (e16a)--(e26a); 
\draw [line width=\pmwidth, color=\pmcolor] (e33a)-- +(-0.5,0) ; \draw [line width=\pmwidth, color=\pmcolor] (e36a)-- +(-0.5,0) ; \draw [line width=\pmwidth, color=\pmcolor] (e32a)-- +(-0.5,0) ;
\draw [line width=\pmwidth, color=\pmcolor] (e03)-- +(0.5,0) ; \draw [line width=\pmwidth, color=\pmcolor] (e06)-- +(0.5,0) ; 

\dimerquotient
\end{tikzpicture}
} } ; 

\node (PM3) at (9,0) {
\scalebox{0.4}{
\begin{tikzpicture}
\draw[line width=\pmwidth, color=\pmcolor] (e00)--(e10); \draw[line width=\pmwidth, color=\pmcolor] (e01)--(e11);
\draw[line width=\pmwidth, color=\pmcolor] (e02)--(e12); \draw[line width=\pmwidth, color=\pmcolor] (e03)--(e13);
\draw[line width=\pmwidth, color=\pmcolor] (e04)--(e14); \draw[line width=\pmwidth, color=\pmcolor] (e05)--(e15);
\draw[line width=\pmwidth, color=\pmcolor] (e06)--(e16); 
\draw[line width=\pmwidth, color=\pmcolor] (e02a)--(e12a); \draw[line width=\pmwidth, color=\pmcolor] (e03a)--(e13a);
\draw[line width=\pmwidth, color=\pmcolor] (e06a)--(e16a); 
\draw[line width=\pmwidth, color=\pmcolor] (e30a)--(e40a); \draw[line width=\pmwidth, color=\pmcolor] (e31)--(e41); 
\draw[line width=\pmwidth, color=\pmcolor] (e32a)--(e42a); \draw[line width=\pmwidth, color=\pmcolor] (e33a)--(e43a);
\draw[line width=\pmwidth, color=\pmcolor] (e36a)--(e46a); 

\dimerquotient
\end{tikzpicture}
} } ; 

\node (PM4) at (0,-3) {
\scalebox{0.4}{
\begin{tikzpicture}
\draw[line width=\pmwidth, color=\pmcolor] (e00)--(e10); \draw[line width=\pmwidth, color=\pmcolor] (e20)--(e30); 
\draw[line width=\pmwidth, color=\pmcolor] (e40)--(e50); 
\draw[line width=\pmwidth, color=\pmcolor] (e13)--(e23); \draw[line width=\pmwidth, color=\pmcolor] (e33)--(e43); 
\draw[line width=\pmwidth, color=\pmcolor] (e53)--(e03); \draw[line width=\pmwidth, color=\pmcolor] (e33a)--(e43a); 
\draw[line width=\pmwidth, color=\pmcolor] (e04)--(e14); \draw[line width=\pmwidth, color=\pmcolor] (e24)--(e34); 
\draw[line width=\pmwidth, color=\pmcolor] (e36)--(e46); \draw[line width=\pmwidth, color=\pmcolor] (e56)--(e06); 
\draw[line width=\pmwidth, color=\pmcolor] (e36a)--(e46a); \draw[line width=\pmwidth, color=\pmcolor] (e32a)--(e42a); 

\dimerquotient
\end{tikzpicture}
} } ; 

\node (PM5) at (4.5,-3) {
\scalebox{0.4}{
\begin{tikzpicture}
\draw[line width=\pmwidth, color=\pmcolor] (e00)--(e10); \draw[line width=\pmwidth, color=\pmcolor] (e20)--(e30); 
\draw[line width=\pmwidth, color=\pmcolor] (e40)--(e50); 
\draw[line width=\pmwidth, color=\pmcolor] (e01)--(e51); \draw[line width=\pmwidth, color=\pmcolor] (e22)--(e32); 
\draw[line width=\pmwidth, color=\pmcolor] (e22a)--(e32a); 
\draw[line width=\pmwidth, color=\pmcolor] (e13)--(e23); \draw[line width=\pmwidth, color=\pmcolor] (e53)--(e03); 
\draw[line width=\pmwidth, color=\pmcolor] (e24)--(e34); 
\draw[line width=\pmwidth, color=\pmcolor] (e15)--(e25); \draw[line width=\pmwidth, color=\pmcolor] (e05)--(e55); 
\draw[line width=\pmwidth, color=\pmcolor] (e06)--(e56); \draw[line width=\pmwidth, color=\pmcolor] (e26a)--(e36a); 
\draw [line width=\pmwidth, color=\pmcolor] (e33a)-- +(-0.5,0) ;

\dimerquotient
\end{tikzpicture}
} } ; 

\node (PM6) at (9,-3) {
\scalebox{0.4}{
\begin{tikzpicture}
\draw[line width=\pmwidth, color=\pmcolor] (e00)--(e10); \draw[line width=\pmwidth, color=\pmcolor] (e20)--(e30); 
\draw[line width=\pmwidth, color=\pmcolor] (e40)--(e50); 
\draw[line width=\pmwidth, color=\pmcolor] (e13)--(e23); 
\draw[line width=\pmwidth, color=\pmcolor] (e14)--(e24); 
\draw[line width=\pmwidth, color=\pmcolor] (e15)--(e25); 
\draw[line width=\pmwidth, color=\pmcolor] (e16)--(e26); \draw[line width=\pmwidth, color=\pmcolor] (e36)--(e46); 
\draw[line width=\pmwidth, color=\pmcolor] (e56)--(e06); 
\draw[line width=\pmwidth, color=\pmcolor] (e16a)--(e26a); \draw[line width=\pmwidth, color=\pmcolor] (e36a)--(e46a); 
\draw [line width=\pmwidth, color=\pmcolor] (e33a)-- +(-0.5,0) ; \draw [line width=\pmwidth, color=\pmcolor] (e32a)-- +(-0.5,0) ;
\draw [line width=\pmwidth, color=\pmcolor] (e03)-- +(0.5,0) ; 

\dimerquotient
\end{tikzpicture}
} } ; 

\end{tikzpicture}
\end{center}

\end{example}

\subsection{Mutations of QPs associated with dimer models} 
\label{subsec_mutation_QP}

As we mentioned in Subsection~\ref{subsec_pm}, a consistent dimer model associated with a given lattice polygon is not unique. 
In order to understand the relationship between consistent dimer models giving the same PM polygon, we introduce the mutations of QPs associated with dimer models. 
To define the mutations of QPs, we consider the complete path algebra. 
For a quiver $Q$, the \emph{complete path algebra} is defined as $\widehat{\kk Q}\coloneqq\prod_{r\ge 0}\kk Q_r$ 
where $\kk Q_r$ is the vector space with a basis $Q_r$. The multiplication is defined by the same way as the path algebra $\kk Q$. 
We note that $\widehat{\kk Q}$ is the $\fkm_Q$-adic completion of $\kk Q$, where $\fkm_Q\coloneqq \prod_{r\ge 1}\kk Q_r$.  
For a subset $V\subseteq\widehat{\kk Q}$, we define the \emph{$\fkm_Q$-adic closure} of $V$ as $\overline{V}\coloneqq \bigcap_{n\ge 0}(V+\fkm_Q^n)$. 
Then, we define the \emph{complete Jacobian algebra} $\widehat{\kk Q}/\overline{\calJ_Q}$, and we denote this by $\widehat{\calP}(Q,W_Q)$. 
Also, the \emph{truncated complete Jacobian algebra} is defined similar to the truncated Jacobian algebra. 

\begin{definition} 
We say that two potentials $W_1, W_2$ of a quiver $Q$ are \emph{cyclically equivalent} if $W_1-W_2$ lies in
the closure of the span of all elements with the form $a_1\cdots a_r-a_2\cdots a_ra_1$, where $a_1\cdots a_r$ is a cyclic path. 

Let $(Q, W), (Q^\prime, W^\prime)$ be QPs. 
We say that $(Q, W)$ and $(Q^\prime, W^\prime)$ are \emph{right equivalent} if $Q_0=Q_0^\prime$ and 
there exists a isomorphism $\varphi:\widehat{\kk Q}\rightarrow\widehat{\kk Q^\prime}$ such that $\varphi\mid_{Q_0}=id$, 
and $\varphi(W)$ and $W^\prime$ are cyclically equivalent. 
\end{definition}

We remark that if QPs $(Q, W), (Q^\prime, W^\prime)$ are right equivalent, then we have $\widehat{\calP}(Q, W)\cong\widehat{\calP}(Q^\prime, W^\prime)$ \cite[Proposition~3.3 and 3.7]{DWZ}, and hence we will consider the right equivalence classes of QPs. 

\begin{definition}
We say that a QP $(Q, W)$ is \emph{trivial} if $W$ consists of cycles of length two and $\widehat{\calP}(Q,W)$ is isomorphic to $\widehat{\kk Q_0}$. 
Also, we say that a QP $(Q, W)$ is \emph{reduced} if $W$ is a linear combination of cycles of length three or more. 
\end{definition}

By the splitting theorem \cite[Theorem~4.6]{DWZ}, any QP $(Q,W)$ is decomposed as a direct sum of a trivial QP $(Q_{\rm triv}, W_{\rm triv})$ and 
a reduced QP $(Q_{\rm red}, W_{\rm red})$: 
\[
(Q, W)\cong (Q_{\rm triv}, W_{\rm triv})\oplus(Q_{\rm red}, W_{\rm red}), 
\]
up to right equivalence where $Q_0=(Q_{\rm triv})_0=(Q_{\rm red})_0$, $Q_1=(Q_{\rm triv})_1\sqcup(Q_{\rm red})_1$, 
and $W=W_{\rm triv}+W_{\rm red}$. 

Using these notions, we can define the mutation of a QP $(Q, W)$ as in \cite[Section~5]{DWZ} (see also \cite[Subsection~1.2]{BIRS}), 
and this mutation is defined for each vertex $k\in Q_0$ not lying on a $2$-cycle and has no loops. 
However, we easily see that the mutation of the QP associated with a dimer model at some vertex is not the dual of a dimer model in general. 
In order to make the mutated QP the dual of a dimer model, we have to impose some restrictions as follows (see e.g., \cite[Lemma~4.6]{Nak}). 

\begin{lemma}
\label{lem_mutatable}
Let $\Gamma$ be a dimer model, and $(Q, W_Q)$ be the QP associated with $\Gamma$.
We assume that $k\in Q_0$ is not lying on $2$-cycles, and has no loops. 
Then, the mutation of $(Q, W_Q)$ at $k\in Q_0$ is the dual of a dimer model if and only if the vertex $k\in Q_0$ has exactly two incoming arrows $($equivalently, exactly two outgoing arrows$)$. 
\end{lemma}

We denote by $Q_0^\mu$ the set of vertices in $Q$ satisfying the equivalent conditions in Lemma~\ref{lem_mutatable}. 
Since a vertex $k\in Q_0^\mu$ has two incoming arrows and two outgoing arrows, we can divide them into two pairs 
$\{a_1,b_1\}, \{a_2,b_2\}$ where $a_1, a_2$ (resp. $b_1, b_2$) are incoming (resp. outgoing) arrows incident to $k$. 
We remark that there are two possibilities of a choice of such pairs, that is, they belong to clockwise small cycles or anti-clockwise small cycles. 
In the rest, we assume that $\{a_1,b_1\}, \{a_2,b_2\}$ belong to anti-clockwise small cycles. 
Then, we define the mutation of the QPs associated with dimer models as follows, 
which is a special case of the mutation of QPs introduced in \cite{DWZ}. 
(For more details, see e.g., \cite[Subsection~7.2]{Boc_NCCR}, \cite[Section~4]{Nak}.) 

\begin{definition}
\label{def_mut_QP}
Let $(Q, W_Q)$ be the QP associated with a dimer model $\Gamma$. Let $k\in Q_0^\mu$, and suppose that $a_1, a_2$ (resp. $b_1, b_2$) are incoming (resp. outgoing) arrows as above. 
Since a potential does not depend on the starting point of each cycle, we may assume that no cycles in $W_Q$ start at $k$. 
We first apply the following procedures: 
\begin{itemize}
\item [(A)] $Q^\prime$ is the quiver obtained from $Q$ as follows. 
 \begin{itemize}
 \item [(A-1)] We add new arrows $[a_ib_j]: \tl(a_i)\rightarrow \hd(b_j)$ for any $i, j=1,2$. 
 \item [(A-2)] Replace the arrow $a_i: \tl(a_i)\rightarrow k$ by the new arrow $a_i^*: k\rightarrow \tl(a_i)$ for $i=1,2$. 
 \item [(A-3)] Replace the arrow $b_j: k\rightarrow \hd(b_j)$ by the new arrow $b_j^*: \hd(b_j)\rightarrow k$ for $j=1,2$. 
 \end{itemize}
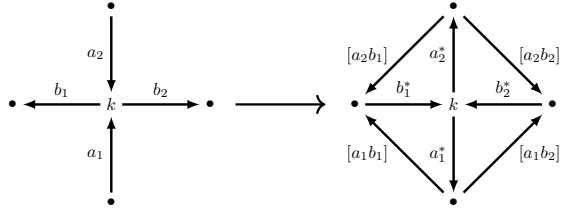
\[
\begin{tikzpicture} 
\node (original) at (0,0) 
{\scalebox{0.65}{
\begin{tikzpicture}[sarrow/.style={-latex, very thick}]
\node (vk) at (2,2){$k$};\node (hb1) at (0,2){$\bullet$}; \node(hb2) at (4,2){$\bullet$}; \node (ta1) at (2,0){$\bullet$};\node (ta2) at (2,4){$\bullet$};

\draw[sarrow, line width=0.05cm] (vk)--node[above] {$b_1$}(hb1);
\draw[sarrow, line width=0.05cm] (ta1)--node[midway,left] {$a_1$}(vk);
\draw[sarrow, line width=0.05cm] (ta2)--node[midway,left] {$a_2$}(vk); 
\draw[sarrow, line width=0.05cm] (vk)--node[above] {$b_2$}(hb2);
\end{tikzpicture} }}; 

\node (mutated) at (4.5,0) 
{\scalebox{0.65}{
\begin{tikzpicture}[sarrow/.style={-latex, very thick}]
\node (vk) at (2,2){$k$};\node (hb1) at (0,2){$\bullet$}; \node(hb2) at (4,2){$\bullet$}; \node (ta1) at (2,0){$\bullet$};\node (ta2) at (2,4){$\bullet$};

\draw[sarrow, line width=0.05cm] (hb1)--node[above] {$b_1^*$}(vk);
\draw[sarrow, line width=0.05cm] (vk)--node[midway,left] {$a_1^*$}(ta1);
\draw[sarrow, line width=0.05cm] (vk)--node[midway,left] {$a_2^*$}(ta2);
\draw[sarrow, line width=0.05cm] (hb2)--node[above] {$b_2^*$}(vk);
\draw[sarrow, line width=0.05cm] (ta1)--node[midway,left,xshift=-5pt] {$[a_1b_1]$}(hb1);
\draw[sarrow, line width=0.05cm] (ta2)--node[midway,left,xshift=-5pt] {$[a_2b_1]$}(hb1);
\draw[sarrow, line width=0.05cm] (ta1)--node[midway,right,xshift=5pt] {$[a_1b_2]$}(hb2);
\draw[sarrow, line width=0.05cm] (ta2)--node[midway,right,xshift=5pt] {$[a_2b_2]$}(hb2);
\end{tikzpicture} }}; 

\draw[->, line width=0.03cm] (original)--(mutated);
\end{tikzpicture}
\]

\item [(B)] $W_Q^\prime=[W_Q]+\Omega$ where $[W_Q]$ and $\Omega$ are obtained by the following rule. 
 \begin{itemize}
 \item [(B-1)] $[W_Q]$ is obtained by substituting $[a_ib_j]$ for each part $a_ib_j$ in $W_Q$ for any $i, j=1,2$. 
 \item [(B-2)] $\Omega=a_1^*[a_1b_1]b_1^*-a_1^*[a_1b_2]b_2^*+a_2^*[a_2 b_2]b_2^*-a_2^*[a_2 b_1]b_1^*$. 
 \end{itemize}
\end{itemize}
Then, we define the \emph{mutation $\mu_k(Q, W_Q)$ of the QP $(Q, W_Q)$ at $k\in Q_0^\mu$} as a reduced part of $(Q^\prime, W_Q^\prime)$, 
that is, $\mu_k(Q, W_Q)\coloneqq (Q^\prime_{\rm red}, (W_Q^\prime)_{\rm red})$. 
We sometimes denote the quiver part of $\mu_k(Q, W_Q)$ by $\mu_k(Q)$, in which case the potential is denoted by $W_{\mu_k(Q)}$, 
that is, $\mu_k(Q, W_Q)=(\mu_k(Q),W_{\mu_k(Q)})$. 
We also denote the dimer model obtained as the dual of $\mu_k(Q,W_Q)$ by $\mu_k(\Gamma)$. 
\end{definition}

We note that $\mu_{k}\left(\mu_k(Q, W_Q)\right)$ is right equivalent to $(Q, W_Q)$ (see \cite[Theorem~5.7]{DWZ}). 
We say that two dimer models are \emph{mutation equivalent} if the associated QPs are transformed into each other by repeating the mutations. 

\medskip

On the other hand, this mutation can be stated in terms of dimer models. 
To define the mutation of dimer models, we use the operation called \emph{spider move} (see e.g., \cite{GK,Boc_ABC}), which is the operation shown in Figure~\ref{fig_spider}. Note that this operation is an involution. 

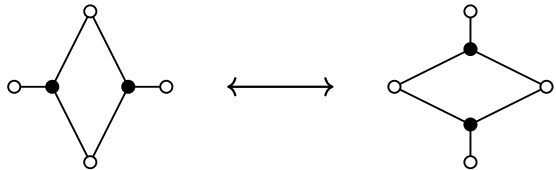
\begin{figure}[H]
\begin{center}
{\scalebox{1}{
\begin{tikzpicture} 
\newcommand{\edgewidth}{0.05cm} 
\newcommand{\nodewidth}{0.05cm} 
\newcommand{\noderad}{0.16} 

\draw[<->, line width=0.03cm] (1.8,0)--(3.2,0); 

\node (mutate_a) at (0,0)
{\scalebox{0.5}{
\begin{tikzpicture} 
\coordinate (B1) at (-1,0); \coordinate (B2) at (1,0); 
\coordinate (W1) at (-2,0); \coordinate (W2) at (0,-2); \coordinate (W3) at (2,0); 
\coordinate (W4) at (0,2); 

\draw[line width=\edgewidth] (W1)--(B1); \draw[line width=\edgewidth] (W2)--(B1);  
\draw[line width=\edgewidth] (W4)--(B1); 
\draw[line width=\edgewidth] (W2)--(B2); \draw[line width=\edgewidth] (W3)--(B2); 
\draw[line width=\edgewidth] (W4)--(B2);

\filldraw  [line width=\nodewidth, fill=black] (B1) circle [radius=\noderad] ; \filldraw  [line width=\nodewidth, fill=black] (B2) circle [radius=\noderad] ;
\draw [line width=\nodewidth, fill=white] (W1) circle [radius=\noderad] ; \draw [line width=\nodewidth, fill=white] (W2) circle [radius=\noderad] ;
\draw [line width=\nodewidth, fill=white] (W3) circle [radius=\noderad] ; \draw [line width=\nodewidth, fill=white] (W4) circle [radius=\noderad] ;
\end{tikzpicture} }} ;

\node (mutate_b) at (5,0)
{\scalebox{0.5}{
\begin{tikzpicture} 
\coordinate (B1) at (0,1); \coordinate (B2) at (0,-1); 
\coordinate (W1) at (-2,0); \coordinate (W2) at (0,-2); \coordinate (W3) at (2,0); 
\coordinate (W4) at (0,2); 

\draw[line width=\edgewidth] (W1)--(B1); \draw[line width=\edgewidth] (W3)--(B1);  
\draw[line width=\edgewidth] (W4)--(B1); 
\draw[line width=\edgewidth] (W1)--(B2); \draw[line width=\edgewidth] (W2)--(B2); 
\draw[line width=\edgewidth] (W3)--(B2);

\filldraw  [line width=\nodewidth, fill=black] (B1) circle [radius=\noderad] ; \filldraw  [line width=\nodewidth, fill=black] (B2) circle [radius=\noderad] ;
\draw [line width=\nodewidth, fill=white] (W1) circle [radius=\noderad] ; \draw [line width=\nodewidth, fill=white] (W2) circle [radius=\noderad] ;
\draw [line width=\nodewidth, fill=white] (W3) circle [radius=\noderad] ; \draw [line width=\nodewidth, fill=white] (W4) circle [radius=\noderad] ;
\end{tikzpicture} }} ;

\end{tikzpicture} 
}}
\caption{Spider move}
\label{fig_spider}
\end{center}
\end{figure}

\begin{definition}
\label{def_mutation_dimer}
Let $\Gamma$ be a dimer model. We pick a quadrangle face $k\in\Gamma_2$. 
Then, the \emph{mutation of a dimer model at $k$}, denoted by $\mu_k(\Gamma)$, is the operation consisting of the following procedures: 
\begin{enumerate}[(1)]
\item We consider two black nodes appearing on the boundary of $k$. 
If there exist black nodes that are not $3$-valent, we apply the split moves to those nodes and make them $3$-valent. 
\item Then, we apply the spider move to $k$ (see Figure~\ref{fig_spider}). 
\item If the resulting dimer model contains $2$-valent nodes, we remove them by applying the join moves. 
\end{enumerate}
\end{definition}

We easily check that the mutated dimer model $\mu_k(\Gamma)$ coincides with the one obtained as the dual of the mutated QP $\mu_k(Q,W_Q)$ in Definition~\ref{def_mut_QP}. 
Furthermore, this mutation preserves the consistency condition and the associated PM polygon. 

\begin{theorem}[{see e.g., \cite[Subsection~7.2]{Boc_NCCR}, \cite[Corollary~4.14]{Nak}}]
\label{mut_preserve_polygon}
Let $(Q,W_Q)$ be the QP associated with a consistent dimer model $\Gamma$. 
Then, for $k\in Q_0^\mu$ the dimer model $\mu_k(\Gamma)$ is also consistent. 
Furthermore, the PM polygon of $\mu_k(\Gamma)$ coincides with that of $\Gamma$. 
\end{theorem}

\subsection{Mutations of graded QPs associated with dimer models} 
In the previous subsection, we saw that using the mutation we have another consistent dimer model giving the same PM polygon. 
Thus, we can consider internal perfect matchings of the original dimer model and those of the mutated one corresponding to the same interior lattice point. 
In the following, we will consider $2$-representation infinite algebras arising from such internal perfect matchings. 
We remark that the main theorem of this subsection (Theorem~\ref{main_thm_derived_eq}) is also obtained via \cite[Theorem~7.2]{IU5} (see Remark~\ref{rem_main_derived_eq} for more details), but we give a representation theoretic proof because it has its own interests. 

\medskip

We recall that the degree $d_\sfD$ defined by a perfect matching $\sfD$ makes the Jacobian algebra a $\ZZ$-graded algebra. 
Thus, we call $(Q,W_Q,d_\sfD)$ a graded QP associated with $\sfD$. 
Then, we also define the \emph{mutation of a graded QP}, which is a special version of \cite[Section~6]{AO}, \cite[Definition~3.2]{Miz}. 

\begin{definition}
\label{def_mutation_gQP}
Let $(Q,W_Q,d_\sfD)$ be a graded QP associated with a consistent dimer model and a perfect matching $\sfD$ of $Q$. 
Let $k\in Q_0^\mu$. 
Then, we define the two kinds of degrees $d^\prime_\rmL$ and $d^\prime_\rmR$ as follows. 
 \begin{itemize}
 \item[] (The degree $d^\prime_\rmL$)
 \begin{itemize}
 \item[$\bullet$] $d_\rmL^\prime(a)=d_\sfD(a)$ for each arrow $a\in Q\cap Q^\prime$, 
 \item[$\bullet$] $d_\rmL^\prime(a_i^*)=-d_\sfD(a_i)+1$ where $a_i$ is an incoming arrow incident to $k$ $(i=1,2)$, 
 \item[$\bullet$] $d_\rmL^\prime(b_j^*)=-d_\sfD(b_j)$ where $b_j$ is an outgoing arrow incident to $k$ $(j=1,2)$, 
 \item[$\bullet$] $d_\rmL^\prime([a_ib_j])=d_\sfD(a_i)+d_\sfD(b_j)$ for the arrow $[a_ib_j]: \tl(a_i)\rightarrow \hd(b_j)$ where $i,j=1,2$. 
 \end{itemize}
 
 \smallskip
  \item[] (The degree $d^\prime_\rmR$)
   \begin{itemize}
 \item[$\bullet$] $d_\rmR^\prime(a)=d_\sfD(a)$ for each arrow $a\in Q\cap Q^\prime$, 
 \item[$\bullet$] $d_\rmR^\prime(a_i^*)=-d_\sfD(a_i)$ where $a_i$ is an incoming arrow incident to $k$ $(i=1,2)$, 
 \item[$\bullet$] $d_\rmR^\prime(b_j^*)=-d_\sfD(b_j)+1$ where $b_j$ is an outgoing arrow incident to $k$ $(j=1,2)$, 
 \item[$\bullet$] $d_\rmR^\prime([a_ib_j])=d_\sfD(a_i)+d_\sfD(b_j)$ for the arrow $[a_ib_j]: \tl(a_i)\rightarrow \hd(b_j)$ where $i,j=1,2$. 
 \end{itemize}
 \end{itemize}
We define the new graded QP $(Q^\prime,W_{Q^\prime},d^\prime)$ as $(Q^\prime,W_{Q^\prime})=\mu_k(Q,W_Q)$ 
and the new degree $d^\prime$ is determined as follows: 
\begin{itemize}
\item if at least one of $a_1, a_2$ is contained in $\sfD$, then we set $d^\prime=d^\prime_\rmL$, 
\item if at least one of $b_1, b_2$ is contained in $\sfD$, then we set $d^\prime=d^\prime_\rmR$,  
\item if none of $a_1,a_2,b_2,b_2$ are contained in $\sfD$, we can choose either $d^\prime_\rmL, d^\prime_\rmR$ as $d^\prime$. 
\end{itemize}

When $d^\prime=d_\rmL^\prime$ (resp. $d^\prime=d_\rmR^\prime$), 
we call $(Q^\prime,W_{Q^\prime},d^\prime)$ the \emph{left mutation} (resp. \emph{right mutation}) of a graded QP $(Q,W_Q,d_\sfD)$ at $k\in Q_0^\mu$, 
in which case the QP $(Q^\prime,W_{Q^\prime},d^\prime)$ is denoted by $\mu_k^\rmL(Q,W_Q,d_\sfD)$ (resp. $\mu_k^\rmR(Q,W_Q,d_\sfD)$). 
Furthermore, in any case, we easily see that the degree $d^\prime$ can be given by some perfect matching $\sfD^\prime$ of $\mu_k(Q)$, 
that is, $d^\prime=d_{\sfD^\prime}$. 
Thus, when $d^\prime=d_\rmL^\prime$ (resp. $d^\prime=d_\rmR^\prime$), 
we denote such a perfect matching $\sfD^\prime$ by $\mu_k^\rmL(\sfD)$ (resp. $\mu_k^\rmR(\sfD)$). 
Namely, $\mu_k^\rmL(Q,W_Q,d_\sfD) =(\mu_k(Q), W_{\mu_k(Q)},d_{\mu_k^\rmL(\sfD)})$ and $\mu_k^\rmR(Q,W_Q,d_\sfD) =(\mu_k(Q), W_{\mu_k(Q)},d_{\mu_k^\rmR(\sfD)})$. 
\end{definition}

Moreover, we easily have the following lemma by definition. 

\begin{lemma}
\label{mutated_deg_pm}
Let the notation be the same as Definition~{\rm\ref{def_mut_QP}} and {\rm\ref{def_mutation_gQP}}. 
Let $\sfD, \sfE$ be perfect matchings of $Q$. 
We assume that $a_1,a_2\in\sfD$ and $b_1,b_2\in\sfE$, that is, $k\in Q_0^\mu$ is a strict source of $(Q,\sfD)$ and is a strict sink of $(Q,\sfE)$. 
Then, $\mu_k^\rmL(\sfD)=\mu_k^\rmR(\sfE)$ holds, in which case we simply denote it by $\mu_k(\sfD)=\mu_k(\sfE)$. 
\end{lemma}

In the rest, we assume that the PM polygon of a consistent dimer model $\Gamma$ contains an interior lattice point. 
Thus, $\Gamma$ admits an internal perfect matching $D$ and there exists a strict source and sink of $(Q,\sfD)$ by Proposition~\ref{findim_internal}. 
We suppose that $k\in Q_0^\mu$ is a strict source of $(Q,\sfD)$. Then, $k$ is a strict sink of $(Q, \lambda_k^+(\sfD))$. 
Thus, we have the following diagram by Lemma~\ref{mutated_deg_pm}: 
\begin{center}
\begin{tikzcd}
(Q,W_Q,d_\sfD) \arrow[rd, "\mu_k^\rmL", bend left=20] \arrow[dd, shift right=0.2cm,"\lambda^+_k" '] & \\
&(\mu_k(Q,W_Q),d_{\mu_k(\sfD)}) \\
(Q,W_Q,d_{\lambda_k^+(\sfD)}) \arrow[ru, "\mu_k^\rmR" ', bend right=20] \arrow[uu, shift right=0.2cm, "\lambda^-_k" '] & 
\end{tikzcd}
\end{center}

\begin{lemma}
\label{lem_same_pt}
Let the notation be the same as above, especially $\sfD$ is an internal perfect matching of $Q$ and $k\in Q_0^\mu$ is a strict source of $(Q,\sfD)$. 
Let $D$ $($resp. $\mu_k(D)$$)$ be the perfect matching of $\Gamma$ $($resp. $\mu_k(\Gamma)$$)$ corresponding to $\sfD$ $($resp. $\mu_k(\sfD)$$)$. 
Then, the lattice point of the PM polygon $\Delta_{\mu_k(\Gamma)}$ corresponding to $\mu_k(D)$ is 
the same as the one of the PM polygon $\Delta_\Gamma$ corresponding to $D$. 
\end{lemma}

\begin{proof}
We consider the face on $\Gamma$ corresponding to $k\in Q_0^\mu$, and use the same notation $k\in\Gamma_2$ for this face. 
We denote the white (resp. black) nodes appearing in the boundary of $k$ by $W_1,W_2$ (resp. $B_1$, $B_2$). 
We make black nodes $B_1,B_2$ $3$-valent using split moves as shown in Figure~\ref{setting_fig_mutation}. 
(We remark that even if they are $3$-valent from the beginning, we apply the split move because it does not affect our purpose and it makes our argument simpler). 
We apply the spider move to $k$ and have Figure~\ref{fig_mut_pm}, where grayed edges are the ones contained in $D, \mu_k(D)$ respectively. 

\begin{figure}[H]
\begin{center}
\begin{tabular}{c}
\begin{minipage}{0.3\hsize}
\begin{center}
{\scalebox{0.5}{
\begin{tikzpicture} 
\newcommand{\edgewidth}{0.05cm} 
\newcommand{\nodewidth}{0.05cm} 
\newcommand{\noderad}{0.16} 

\coordinate (B1) at (-1,0); \coordinate (B2) at (1,0); 
\coordinate (W1) at (-2,0); \coordinate (W2) at (0,-2); \coordinate (W3) at (2,0); 
\coordinate (W4) at (0,2); 

\coordinate (B3) at (-3,0); \coordinate (B4) at (3,0); 

\draw[line width=\edgewidth] (W1)--(B1); \draw[line width=\edgewidth] (W2)--(B1);  
\draw[line width=\edgewidth] (W4)--(B1); 
\draw[line width=\edgewidth] (W2)--(B2); \draw[line width=\edgewidth] (W3)--(B2); 
\draw[line width=\edgewidth] (W4)--(B2);
\draw[line width=\edgewidth] (W1)--(B3); \draw[line width=\edgewidth] (W3)--(B4);

\filldraw  [line width=\nodewidth, fill=black] (B1) circle [radius=\noderad] ; \filldraw  [line width=\nodewidth, fill=black] (B2) circle [radius=\noderad] ;
\filldraw  [line width=\nodewidth, fill=black] (B3) circle [radius=\noderad] ; \filldraw  [line width=\nodewidth, fill=black] (B4) circle [radius=\noderad] ;

\draw [line width=\nodewidth, fill=white] (W1) circle [radius=\noderad] ; \draw [line width=\nodewidth, fill=white] (W2) circle [radius=\noderad] ;
\draw [line width=\nodewidth, fill=white] (W3) circle [radius=\noderad] ; \draw [line width=\nodewidth, fill=white] (W4) circle [radius=\noderad] ;
\node at (0,0) {{\LARGE$k$}} ; \node at (-1.2,0.45) {{\Large$B_1$}} ; \node at (1.2,0.45) {{\Large$B_2$}} ; 
\node at (-0.5,2) {{\Large$W_1$}} ; \node at (-0.5,-2) {{\Large$W_2$}} ; 
\end{tikzpicture} 
}}
\end{center}
\caption{}
\label{setting_fig_mutation}
\end{minipage}

\begin{minipage}{0.65\hsize}
\begin{center}
{\scalebox{1}{
\begin{tikzpicture} 
\newcommand{\edgewidth}{0.05cm} 
\newcommand{\nodewidth}{0.05cm} 
\newcommand{\noderad}{0.16} 

\newcommand{\pmwidth}{0.4cm} 
\newcommand{\pmcolor}{gray} 

\draw[->, line width=0.03cm] (2,0)--(3,0); 

\node (mutate_a) at (0,0)
{\scalebox{0.5}{
\begin{tikzpicture} 
\coordinate (B1) at (-1,0); \coordinate (B2) at (1,0); 
\coordinate (W1) at (-2,0); \coordinate (W2) at (0,-2); \coordinate (W3) at (2,0); 
\coordinate (W4) at (0,2); 
\coordinate (B3) at (-3,0); \coordinate (B4) at (3,0); 

\draw[line width=\pmwidth, color=\pmcolor] (W4)--(B2); \draw[line width=\pmwidth, color=\pmcolor] (W2)--(B1); 
\draw[line width=\pmwidth, color=\pmcolor] (W1)--(B3); \draw[line width=\pmwidth, color=\pmcolor] (W3)--(B4); 

\draw[line width=\edgewidth] (W1)--(B1); \draw[line width=\edgewidth] (W2)--(B1);  
\draw[line width=\edgewidth] (W4)--(B1); 
\draw[line width=\edgewidth] (W2)--(B2); \draw[line width=\edgewidth] (W3)--(B2); 
\draw[line width=\edgewidth] (W4)--(B2);
\draw[line width=\edgewidth] (W1)--(B3); \draw[line width=\edgewidth] (W3)--(B4);

\filldraw  [line width=\nodewidth, fill=black] (B1) circle [radius=\noderad] ; \filldraw  [line width=\nodewidth, fill=black] (B2) circle [radius=\noderad] ;
\filldraw  [line width=\nodewidth, fill=black] (B3) circle [radius=\noderad] ; \filldraw  [line width=\nodewidth, fill=black] (B4) circle [radius=\noderad] ;

\draw [line width=\nodewidth, fill=white] (W1) circle [radius=\noderad] ; \draw [line width=\nodewidth, fill=white] (W2) circle [radius=\noderad] ;
\draw [line width=\nodewidth, fill=white] (W3) circle [radius=\noderad] ; \draw [line width=\nodewidth, fill=white] (W4) circle [radius=\noderad] ;

\node at (0,0) {{\LARGE$k$}} ; 
\end{tikzpicture} }} ;

\node (mutate_b) at (5,0)
{\scalebox{0.5}{
\begin{tikzpicture} 
\coordinate (B1) at (0,1); \coordinate (B2) at (0,-1); 
\coordinate (W1) at (-2,0); \coordinate (W2) at (0,-2); \coordinate (W3) at (2,0); 
\coordinate (W4) at (0,2); 
\coordinate (B3) at (-3,0); \coordinate (B4) at (3,0); 

\draw[line width=\pmwidth, color=\pmcolor] (W2)--(B2); \draw[line width=\pmwidth, color=\pmcolor] (W4)--(B1); 
\draw[line width=\pmwidth, color=\pmcolor] (W1)--(B3); \draw[line width=\pmwidth, color=\pmcolor] (W3)--(B4); 

\draw[line width=\edgewidth] (W1)--(B1); \draw[line width=\edgewidth] (W3)--(B1);  
\draw[line width=\edgewidth] (W4)--(B1); 
\draw[line width=\edgewidth] (W1)--(B2); \draw[line width=\edgewidth] (W2)--(B2); 
\draw[line width=\edgewidth] (W3)--(B2);
\draw[line width=\edgewidth] (W1)--(B3); \draw[line width=\edgewidth] (W3)--(B4);

\filldraw  [line width=\nodewidth, fill=black] (B1) circle [radius=\noderad] ; \filldraw  [line width=\nodewidth, fill=black] (B2) circle [radius=\noderad] ;
\filldraw  [line width=\nodewidth, fill=black] (B3) circle [radius=\noderad] ; \filldraw  [line width=\nodewidth, fill=black] (B4) circle [radius=\noderad] ;

\draw [line width=\nodewidth, fill=white] (W1) circle [radius=\noderad] ; \draw [line width=\nodewidth, fill=white] (W2) circle [radius=\noderad] ;
\draw [line width=\nodewidth, fill=white] (W3) circle [radius=\noderad] ; \draw [line width=\nodewidth, fill=white] (W4) circle [radius=\noderad] ;

\node at (0,0) {{\LARGE$k$}} ; 
\end{tikzpicture} }} ;

\end{tikzpicture} 
}}
\end{center}
\caption{The spider move and the induced perfect matching}
\label{fig_mut_pm}

\end{minipage}
\end{tabular}\end{center}
\end{figure}

Here, we recall that $\Delta_\Gamma=\Delta_{\mu_k(\Gamma)}$ by Theorem~\ref{mut_preserve_polygon}. 
We take perfect matchings $\widetilde{D}$ of $\Gamma$ and $\widetilde{D}^\prime$ of $\mu_k(\Gamma)$, 
and assume that $\widetilde{D}$ and $\widetilde{D}^\prime$ correspond to the same point of $\Delta_\Gamma=\Delta_{\mu_k(\Gamma)}$. 
To prove our assertion, it is enough to show that $[\widetilde{D}-D]=[\widetilde{D}^\prime-\mu_k(D)]\in\rmH_1(\TT)$. 
In particular, we will consider the following corner perfect matchings $D_1, D_1^\prime$ as $\widetilde{D}, \widetilde{D}^\prime$ respectively. 

We consider the zigzag path $z$ shown in Figure~\ref{fig_keylem_1} (left), which passes through the boundary of $k$. 
By Proposition~\ref{zigzag_sidepolygon}, there exists a unique corner perfect matching containing all zig (or all zag) of $z$. 
The corner perfect matching containing all zig (resp. all zag) of $z$ takes the form as shown in Figure~\ref{fig_keylem_1} around the face $k$. 
By Proposition~\ref{zigzag_sidepolygon}, these corner perfect matchings are adjacent, and we denote these by $D_1,D_2$ respectively. 
After applying the spider move, we have the dimer model $\mu_k(\Gamma)$ and the zigzag path $z^\prime$ with the slope $[z^\prime]=[z]$ 
as shown in Figure~\ref{fig_keylem_2} (left). 
By Proposition~\ref{zigzag_sidepolygon}, 
there exist adjacent corner perfect matchings $D_1^\prime,D_2^\prime$ of $\mu_k(\Gamma)$ 
such that they correspond to the same vertex of $\Delta_\Gamma=\Delta_{\mu_k(\Gamma)}$ as $D_1,D_2$ respectively, 
and the difference of $D_1^\prime$ and $D_2^\prime$ forms $z^\prime$ (see Figure~\ref{fig_keylem_2}). 
(We note that $D_1^\prime=\mu^\rmR_k(D_1)$ and $D_2^\prime=\mu^\rmL_k(D_2)$). 

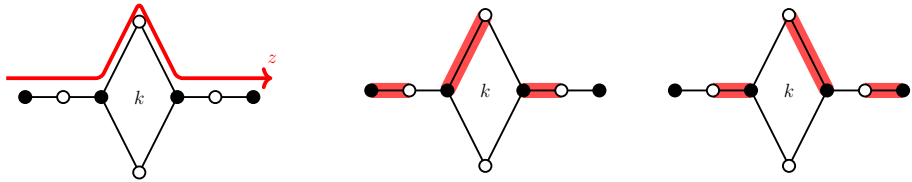
\begin{figure}[H]
\begin{center}
{\scalebox{1}{
\begin{tikzpicture} 
\newcommand{\edgewidth}{0.05cm} 
\newcommand{\nodewidth}{0.05cm} 
\newcommand{\noderad}{0.16} 

\newcommand{\pmwidth}{0.4cm} 
\newcommand{\pmcolor}{red!70} 

\node at (-4.5,0)
{\scalebox{0.5}{
\begin{tikzpicture} 
\coordinate (B1) at (-1,0); \coordinate (B2) at (1,0); 
\coordinate (W1) at (-2,0); \coordinate (W2) at (0,-2); \coordinate (W3) at (2,0); \coordinate (W4) at (0,2); 

\coordinate (B3) at (-3,0); \coordinate (B4) at (3,0); 

\draw[line width=\edgewidth] (W1)--(B1); \draw[line width=\edgewidth] (W2)--(B1);  
\draw[line width=\edgewidth] (W4)--(B1); 
\draw[line width=\edgewidth] (W2)--(B2); \draw[line width=\edgewidth] (W3)--(B2); 
\draw[line width=\edgewidth] (W4)--(B2);
\draw[line width=\edgewidth] (W1)--(B3); \draw[line width=\edgewidth] (W3)--(B4);

\filldraw  [line width=\nodewidth, fill=black] (B1) circle [radius=\noderad] ; \filldraw  [line width=\nodewidth, fill=black] (B2) circle [radius=\noderad] ;
\filldraw  [line width=\nodewidth, fill=black] (B3) circle [radius=\noderad] ; \filldraw  [line width=\nodewidth, fill=black] (B4) circle [radius=\noderad] ;

\draw [line width=\nodewidth, fill=white] (W1) circle [radius=\noderad] ; \draw [line width=\nodewidth, fill=white] (W2) circle [radius=\noderad] ;
\draw [line width=\nodewidth, fill=white] (W3) circle [radius=\noderad] ; \draw [line width=\nodewidth, fill=white] (W4) circle [radius=\noderad] ;

\path (B1) ++(90:0.5cm) coordinate (B1+); 
\path (B2) ++(90:0.5cm) coordinate (B2+); 
\path (-3.5,0) ++(90:0.5cm) coordinate (B3+); 
\path (3.5,0) ++(90:0.5cm) coordinate (B4+); 
\path (W4) ++(90:0.5cm) coordinate (W4+); 

\draw [->, rounded corners, line width=0.1cm, red] (B3+)--(B1+)--(W4+)--(B2+)--(B4+) ; 

\node[red] at (3.5,1) {{\LARGE$z$}} ; 
\node at (0,0) {{\LARGE$k$}} ;
\end{tikzpicture} 
}};

\node (PM_D1) at (0,0)
{\scalebox{0.5}{
\begin{tikzpicture} 
\coordinate (B1) at (-1,0); \coordinate (B2) at (1,0); 
\coordinate (W1) at (-2,0); \coordinate (W2) at (0,-2); \coordinate (W3) at (2,0); \coordinate (W4) at (0,2); 
\coordinate (B3) at (-3,0); \coordinate (B4) at (3,0); 

\draw[line width=\pmwidth, color=\pmcolor] (W4)--(B1); \draw[line width=\pmwidth, color=\pmcolor] (W3)--(B2); 
\draw[line width=\pmwidth, color=\pmcolor] (W1)--(B3); 

\draw[line width=\edgewidth] (W1)--(B1); \draw[line width=\edgewidth] (W2)--(B1);  
\draw[line width=\edgewidth] (W4)--(B1); 
\draw[line width=\edgewidth] (W2)--(B2); \draw[line width=\edgewidth] (W3)--(B2); 
\draw[line width=\edgewidth] (W4)--(B2);
\draw[line width=\edgewidth] (W1)--(B3); \draw[line width=\edgewidth] (W3)--(B4);

\filldraw  [line width=\nodewidth, fill=black] (B1) circle [radius=\noderad] ; \filldraw  [line width=\nodewidth, fill=black] (B2) circle [radius=\noderad] ;
\filldraw  [line width=\nodewidth, fill=black] (B3) circle [radius=\noderad] ; \filldraw  [line width=\nodewidth, fill=black] (B4) circle [radius=\noderad] ;

\draw [line width=\nodewidth, fill=white] (W1) circle [radius=\noderad] ; \draw [line width=\nodewidth, fill=white] (W2) circle [radius=\noderad] ;
\draw [line width=\nodewidth, fill=white] (W3) circle [radius=\noderad] ; \draw [line width=\nodewidth, fill=white] (W4) circle [radius=\noderad] ;

\node at (0,0) {{\LARGE$k$}} ; 
\end{tikzpicture} }} ;

\node (PM_D2) at (4,0)
{\scalebox{0.5}{
\begin{tikzpicture} 
\draw[line width=\pmwidth, color=\pmcolor] (W4)--(B2); \draw[line width=\pmwidth, color=\pmcolor] (W1)--(B1); 
\draw[line width=\pmwidth, color=\pmcolor] (W3)--(B4);  

\draw[line width=\edgewidth] (W1)--(B1); \draw[line width=\edgewidth] (W2)--(B1);  
\draw[line width=\edgewidth] (W4)--(B1); 
\draw[line width=\edgewidth] (W2)--(B2); \draw[line width=\edgewidth] (W3)--(B2); 
\draw[line width=\edgewidth] (W4)--(B2);
\draw[line width=\edgewidth] (W1)--(B3); \draw[line width=\edgewidth] (W3)--(B4);

\filldraw  [line width=\nodewidth, fill=black] (B1) circle [radius=\noderad] ; \filldraw  [line width=\nodewidth, fill=black] (B2) circle [radius=\noderad] ;
\filldraw  [line width=\nodewidth, fill=black] (B3) circle [radius=\noderad] ; \filldraw  [line width=\nodewidth, fill=black] (B4) circle [radius=\noderad] ;

\draw [line width=\nodewidth, fill=white] (W1) circle [radius=\noderad] ; \draw [line width=\nodewidth, fill=white] (W2) circle [radius=\noderad] ;
\draw [line width=\nodewidth, fill=white] (W3) circle [radius=\noderad] ; \draw [line width=\nodewidth, fill=white] (W4) circle [radius=\noderad] ;

\node at (0,0) {{\LARGE$k$}} ; 
\end{tikzpicture} }} ;

\end{tikzpicture} 
}}
\end{center}
\caption{\small The zigzag path $z$ around $k$ (left), and the corner perfect matchings $D_1$ (center) and $D_2$ (right) giving $z$}
\label{fig_keylem_1}
\end{figure}

\begin{figure}[H]
\begin{center}
{\scalebox{1}{
\begin{tikzpicture} 
\newcommand{\edgewidth}{0.05cm} 
\newcommand{\nodewidth}{0.05cm} 
\newcommand{\noderad}{0.16} 

\newcommand{\pmwidth}{0.4cm} 
\newcommand{\pmcolor}{red!70} 

\node at (-4.5,0)
{\scalebox{0.5}{
\begin{tikzpicture} 
\coordinate (B1) at (0,1); \coordinate (B2) at (0,-1); 
\coordinate (W1) at (-2,0); \coordinate (W2) at (0,-2); \coordinate (W3) at (2,0); 
\coordinate (W4) at (0,2); 
\coordinate (B3) at (-3,0); \coordinate (B4) at (3,0); 

\draw[line width=\edgewidth] (W1)--(B1); \draw[line width=\edgewidth] (W3)--(B1);  
\draw[line width=\edgewidth] (W4)--(B1); 
\draw[line width=\edgewidth] (W1)--(B2); \draw[line width=\edgewidth] (W2)--(B2); 
\draw[line width=\edgewidth] (W3)--(B2);
\draw[line width=\edgewidth] (W1)--(B3); \draw[line width=\edgewidth] (W3)--(B4);

\filldraw  [line width=\nodewidth, fill=black] (B1) circle [radius=\noderad] ; \filldraw  [line width=\nodewidth, fill=black] (B2) circle [radius=\noderad] ;
\filldraw  [line width=\nodewidth, fill=black] (B3) circle [radius=\noderad] ; \filldraw  [line width=\nodewidth, fill=black] (B4) circle [radius=\noderad] ;

\draw [line width=\nodewidth, fill=white] (W1) circle [radius=\noderad] ; \draw [line width=\nodewidth, fill=white] (W2) circle [radius=\noderad] ;
\draw [line width=\nodewidth, fill=white] (W3) circle [radius=\noderad] ; \draw [line width=\nodewidth, fill=white] (W4) circle [radius=\noderad] ;

\path (B2) ++(-90:0.3cm) coordinate (B2-); 
\path (-3.4,0) ++(-90:0.3cm) coordinate (B3-); \path (3.4,0) ++(-90:0.3cm) coordinate (B4-); 
\path (W1) ++(-90:0.3cm) coordinate (W1-); \path (W3) ++(-90:0.3cm) coordinate (W3-); 
\draw [->, rounded corners, line width=0.1cm, red] (B3-)--(W1-)--(B2-)--(W3-)--(B4-); 

\node at (0,0) {{\LARGE$k$}} ; 
\node[red] at (3.4,-0.8) {{\LARGE$z^\prime$}} ; 
\end{tikzpicture} 
}}; 

\node (PM_D1) at (0,0)
{\scalebox{0.5}{
\begin{tikzpicture} 
\coordinate (B1) at (0,1); \coordinate (B2) at (0,-1); 
\coordinate (W1) at (-2,0); \coordinate (W2) at (0,-2); \coordinate (W3) at (2,0); 
\coordinate (W4) at (0,2); 
\coordinate (B3) at (-3,0); \coordinate (B4) at (3,0); 

\draw[line width=\pmwidth, color=\pmcolor] (W4)--(B1); \draw[line width=\pmwidth, color=\pmcolor] (W3)--(B2); 
\draw[line width=\pmwidth, color=\pmcolor] (W1)--(B3); 

\draw[line width=\edgewidth] (W1)--(B1); \draw[line width=\edgewidth] (W3)--(B1);  
\draw[line width=\edgewidth] (W4)--(B1); 
\draw[line width=\edgewidth] (W1)--(B2); \draw[line width=\edgewidth] (W2)--(B2); 
\draw[line width=\edgewidth] (W3)--(B2);
\draw[line width=\edgewidth] (W1)--(B3); \draw[line width=\edgewidth] (W3)--(B4);

\filldraw  [line width=\nodewidth, fill=black] (B1) circle [radius=\noderad] ; \filldraw  [line width=\nodewidth, fill=black] (B2) circle [radius=\noderad] ;
\filldraw  [line width=\nodewidth, fill=black] (B3) circle [radius=\noderad] ; \filldraw  [line width=\nodewidth, fill=black] (B4) circle [radius=\noderad] ;

\draw [line width=\nodewidth, fill=white] (W1) circle [radius=\noderad] ; \draw [line width=\nodewidth, fill=white] (W2) circle [radius=\noderad] ;
\draw [line width=\nodewidth, fill=white] (W3) circle [radius=\noderad] ; \draw [line width=\nodewidth, fill=white] (W4) circle [radius=\noderad] ;

\node at (0,0) {{\LARGE$k$}} ; 
\end{tikzpicture} }} ;

\node (PM_D2) at (4,0)
{\scalebox{0.5}{
\begin{tikzpicture} 
\draw[line width=\pmwidth, color=\pmcolor] (W1)--(B2); \draw[line width=\pmwidth, color=\pmcolor] (W3)--(B4); 
\draw[line width=\pmwidth, color=\pmcolor] (W4)--(B1);  

\draw[line width=\edgewidth] (W1)--(B1); \draw[line width=\edgewidth] (W3)--(B1);  
\draw[line width=\edgewidth] (W4)--(B1); 
\draw[line width=\edgewidth] (W1)--(B2); \draw[line width=\edgewidth] (W2)--(B2); 
\draw[line width=\edgewidth] (W3)--(B2);
\draw[line width=\edgewidth] (W1)--(B3); \draw[line width=\edgewidth] (W3)--(B4);

\filldraw  [line width=\nodewidth, fill=black] (B1) circle [radius=\noderad] ; \filldraw  [line width=\nodewidth, fill=black] (B2) circle [radius=\noderad] ;
\filldraw  [line width=\nodewidth, fill=black] (B3) circle [radius=\noderad] ; \filldraw  [line width=\nodewidth, fill=black] (B4) circle [radius=\noderad] ;

\draw [line width=\nodewidth, fill=white] (W1) circle [radius=\noderad] ; \draw [line width=\nodewidth, fill=white] (W2) circle [radius=\noderad] ;
\draw [line width=\nodewidth, fill=white] (W3) circle [radius=\noderad] ; \draw [line width=\nodewidth, fill=white] (W4) circle [radius=\noderad] ;

\node at (0,0) {{\LARGE$k$}} ; \end{tikzpicture} }} ;

\end{tikzpicture} 
}}
\end{center}
\caption{\small The zigzag path $z^\prime$ with $[z]=[z^\prime]$ (left), and the corner perfect matchings $D_1^\prime$ (center) and $D_2^\prime$ (right) giving $z^\prime$}
\label{fig_keylem_2}
\end{figure}
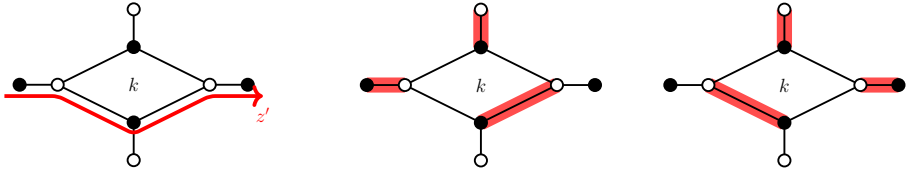

We then obtain Figure~\ref{fig_keylem_5} which shows the differences $D_1-D$ and $D_1^\prime-\mu_k(D)$. 
The $1$-cycle appearing in $D_1-D$ connects $W_4$ to $W_2$ with this order, and so does the $1$-cycle appearing in $D_1^\prime-\mu_k(D)$. 
Since the mutation is a local operation around $k$, the remaining parts of $\Gamma$ and $\mu_k(\Gamma)$ are preserved, 
thus we have $[D_1-D]=[D_1^\prime-\mu_k(D)]\in\rmH_1(\TT)$. 

\begin{figure}[H]
\begin{center}
{\scalebox{1}{
\begin{tikzpicture} 
\newcommand{\edgewidth}{0.05cm} 
\newcommand{\nodewidth}{0.05cm} 
\newcommand{\noderad}{0.16} 

\newcommand{\pmwidth}{0.4cm} 
\newcommand{\pmcolor}{gray} 
\newcommand{\pmcolorI}{red!70} 

\draw[->, line width=0.03cm] (2,0)--(3,0); 

\node (mutate_a) at (0,0)
{\scalebox{0.5}{
\begin{tikzpicture} 
\coordinate (B1) at (-1,0); \coordinate (B2) at (1,0); 
\coordinate (W1) at (0,2); \coordinate (W2) at (0,-2); \coordinate (W3) at (-2,0); 
\coordinate (W4) at (2,0); 
\coordinate (B3) at (-3,0); \coordinate (B4) at (3,0); 

\draw[line width=\pmwidth, color=\pmcolorI] (W1)--(B1); 
\draw[line width=\pmwidth, color=\pmcolorI] (W4)--(B2); 

\draw[line width=\pmwidth, color=\pmcolor] (W1)--(B2); 
\draw[line width=\pmwidth, color=\pmcolor] (W2)--(B1); 
\draw[line width=\pmwidth, color=\pmcolor] (W4)--(B4); 

\draw[line width=\edgewidth] (W3)--(B1); \draw[line width=\edgewidth] (W2)--(B1);  
\draw[line width=\edgewidth] (W1)--(B1); 
\draw[line width=\edgewidth] (W2)--(B2); \draw[line width=\edgewidth] (W4)--(B2); 
\draw[line width=\edgewidth] (W1)--(B2);
\draw[line width=\edgewidth] (W3)--(B3); \draw[line width=\edgewidth] (W4)--(B4);

\filldraw  [line width=\nodewidth, fill=black] (B1) circle [radius=\noderad] ; \filldraw  [line width=\nodewidth, fill=black] (B2) circle [radius=\noderad] ;
\filldraw  [line width=\nodewidth, fill=black] (B3) circle [radius=\noderad] ; \filldraw  [line width=\nodewidth, fill=black] (B4) circle [radius=\noderad] ;

\draw [line width=\nodewidth, fill=white] (W1) circle [radius=\noderad] ; \draw [line width=\nodewidth, fill=white] (W2) circle [radius=\noderad] ;
\draw [line width=\nodewidth, fill=white] (W3) circle [radius=\noderad] ; \draw [line width=\nodewidth, fill=white] (W4) circle [radius=\noderad] ;

\node at (0,0) {{\LARGE$k$}} ; 
\node at (-0.7,2) {{\Large$W_1$}} ; \node at (-0.7,-2) {{\Large$W_2$}} ; 
\node at (-2,0.5) {{\Large$W_3$}} ; \node at (2,0.5) {{\Large$W_4$}} ; 
\end{tikzpicture} }} ;

\node (mutate_b) at (5,0)
{\scalebox{0.5}{
\begin{tikzpicture} 
\coordinate (B1) at (0,1); \coordinate (B2) at (0,-1); 
\coordinate (W1) at (0,2); \coordinate (W2) at (0,-2); \coordinate (W3) at (-2,0); 
\coordinate (W4) at (2,0); 
\coordinate (B3) at (-3,0); \coordinate (B4) at (3,0); 

\draw[line width=\pmwidth, color=\pmcolorI] (W4)--(B2); 
\draw[line width=\pmwidth, color=\pmcolor] (W2)--(B2); 
\draw[line width=\pmwidth, color=\pmcolor] (W4)--(B4); 

\draw[line width=\edgewidth] (W3)--(B1); \draw[line width=\edgewidth] (W4)--(B1);  
\draw[line width=\edgewidth] (W1)--(B1); 
\draw[line width=\edgewidth] (W3)--(B2); \draw[line width=\edgewidth] (W2)--(B2); 
\draw[line width=\edgewidth] (W4)--(B2);
\draw[line width=\edgewidth] (W3)--(B3); \draw[line width=\edgewidth] (W4)--(B4);

\filldraw  [line width=\nodewidth, fill=black] (B1) circle [radius=\noderad] ; \filldraw  [line width=\nodewidth, fill=black] (B2) circle [radius=\noderad] ;
\filldraw  [line width=\nodewidth, fill=black] (B3) circle [radius=\noderad] ; \filldraw  [line width=\nodewidth, fill=black] (B4) circle [radius=\noderad] ;

\draw [line width=\nodewidth, fill=white] (W1) circle [radius=\noderad] ; \draw [line width=\nodewidth, fill=white] (W2) circle [radius=\noderad] ;
\draw [line width=\nodewidth, fill=white] (W3) circle [radius=\noderad] ; \draw [line width=\nodewidth, fill=white] (W4) circle [radius=\noderad] ;

\node at (0,0) {{\LARGE$k$}} ; 
\node at (-0.7,2) {{\Large$W_1$}} ; \node at (-0.7,-2) {{\Large$W_2$}} ; 
\node at (-2,0.5) {{\Large$W_3$}} ; \node at (2,0.5) {{\Large$W_4$}} ; 
\end{tikzpicture} }} ;

\end{tikzpicture} 
}}
\end{center}
\caption{The differences of perfect matchings $D_1-D$ (left) and  $D_1^\prime-\mu_k(D)$ (right)}
\label{fig_keylem_5}
\end{figure}
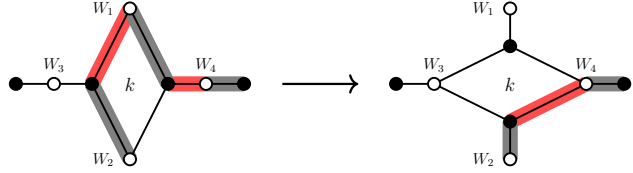

\end{proof}

We are now ready to show the following theorem. 

\begin{theorem}
\label{main_thm_derived_eq}
Let $\Gamma,\Gamma^\prime$ be consistent dimer models associated with the same $3$-dimensional Gorenstein toric singularity $R$. 
Let $(Q,W_Q),(Q^\prime,W_{Q^\prime})$ be QPs associated with $\Gamma,\Gamma^\prime$ respectively, and 
we suppose that $(Q,W_Q)$ and $(Q^\prime,W_{Q^\prime})$ are transformed into each other by repeating the mutations of QPs in the sense of Definition~\ref{def_mut_QP}. 
We assume that the toric diagram $\Delta$ of $R$, which coincides with both of the PM polygon of $\Gamma$ and $\Gamma^\prime$, contains an interior lattice point. 
Let $\sfD$ $($resp. $\sfD^\prime$$)$ be an internal perfect matching of $Q$ $($resp. $Q^\prime$$)$, and assume that $\sfD$ and $\sfD^\prime$ correspond to the same interior lattice point of $\Delta$. 
Then, we have an equivalence 
$$
\calD^\rmb(\mc \,\widehat{\calP}(Q,W_Q)_\sfD)\cong\calD^\rmb(\mc \,\widehat{\calP}(Q^\prime,W_{Q^\prime})_{\sfD^\prime}).
$$
\end{theorem}

\begin{proof}
We may assume that $(Q^\prime,W_{Q^\prime})=\mu_k(Q,W_Q)$ for some $k\in Q_0^\mu$. 
Let $\theta=(\theta_i)_{i\in Q_0}$ be a stability parameter satisfying $\theta_k<0$ and $\theta_i>0$ for all $i\neq k$. 
Then, we can check that $\theta$ is generic. 
Also, we see that if a representation $V$ is $\theta$-stable, then every vertex $i\in Q_0$ is reachable from $k$ passing through non-zero path. 
By Proposition~\ref{corresp_pm}, we can assign the $\theta$-stable perfect matching to each lattice point in $\Delta$. 
By Theorem~\ref{pm_mutation_equiv} and \ref{derivedequ_pm}, we may assume that the internal perfect matching $\sfD$ is $\theta$-stable. 
Then, $k$ must be a strict source of $(Q,\sfD)$ by a choice of $\theta$. 
Collectively, we have $\mu_k^\rmL(Q,W_Q,d_\sfD)=(Q^\prime,W_{Q^\prime},d_{\mu_k(\sfD)})$. 
Then, by \cite[Theorem~3.1]{Miz}, we have a $1$-APR tilting module $T_k$ over $\calA_\sfD=\widehat{\calP}(Q,W_Q)_\sfD$ such that 
$\End_{\calA_\sfD}(T_k)\cong\widehat{\calP}(Q^\prime,W_{Q^\prime})_{d_{\mu_k(\sfD)}}$. Thus, we have an equivalence 
\begin{equation}
\label{isom_for_mainthm}
\calD^\rmb(\mc \,\calA_\sfD)\cong\calD^\rmb(\mc \,\widehat{\calP}(Q^\prime,W_{Q^\prime})_{d_{\mu_k(\sfD)}}).
\end{equation}
By Lemma~\ref{lem_same_pt}, the interior lattice point corresponding to $\mu_k(\sfD)$ is the same as $\sfD$, and hence it is the same as $\sfD^\prime$. 
Thus, we have the assertion by combining Theorem~\ref{derivedequ_pm} and (\ref{isom_for_mainthm}). 
\end{proof}

\begin{remark}
\label{rem_main_derived_eq}
In our proof of Theorem~\ref{main_thm_derived_eq}, we need the assumption that $(Q,W_Q)$ and $(Q^\prime,W_{Q^\prime})$ are transformed into each other by the mutations, because the problem whether all consistent dimer models associated with the same lattice polygon are transformed into each other by  the mutations is still open in general (see \cite[pp396--397]{Boc_ABC}). 
We note that partial answers were given in several papers, see e.g., \cite{Boc_NCCR,GK,Nak}. 
For example, if the toric diagram $\Delta$ is a reflexive polygon (i.e., the origin is the unique interior lattice point of $\Delta$),  then we have the affirmative answer. 
On the other hand, if we consider the toric Deligne-Mumford stack $\calX_{\bf \Sigma}$, which was explained in the first part of this section, we can drop this assumption by using \cite[Theorem~7.2]{IU5}. 
Thus, Theorem~\ref{main_thm_derived_eq} is compatible with the conjecture. 
\end{remark}

We finish this paper with the following example. 

\begin{example}
\label{ex_mutation_dimer}
We consider the quiver with potential $(Q,W_Q)$ associated with the dimer model $\Gamma$ given in Example~\ref{ex_mut_pm1} (the left of Figure~\ref{ex_mut_gQP}), 
and the red arrows form the perfect matching $\sfD_7$ of $Q$ dual to the internal perfect matching $D_7$ of $\Gamma$, especially $0$ is a strict source of $(Q,\sfD_7)$. 
We apply the left mutation $\mu_0^\rmL$ to $(Q, W_Q, d_{\sfD_7})$, and have the right of Figure~\ref{ex_mut_gQP}. 
Here, the red arrows form the internal perfect matching $\mu_0(\sfD_7)$ of $\mu_0(Q)$, and $\sfD_7$ and $\mu_0(\sfD_7)$ correspond the same interior lattice point by Lemma~\ref{lem_same_pt}. 

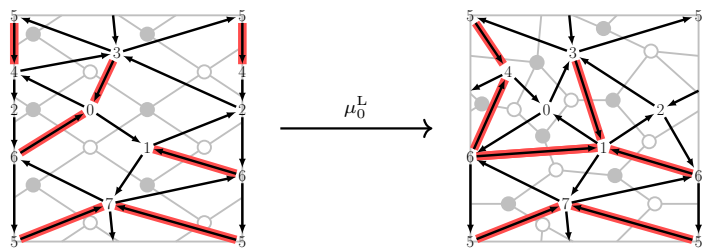
\begin{figure}[H]
\begin{center}
\begin{tikzpicture}
\newcommand{\edgewidth}{0.1cm} 
\newcommand{\nodewidth}{0.1cm} 
\newcommand{\noderad}{0.35} 
\newcommand{\nodecolor}{lightgray} 
\newcommand{\arrowwidth}{0.13cm} 

\newcommand{\pmwidth}{0.45cm} 
\newcommand{\pmcolor}{red!70} 

\coordinate (B1) at (1,3); \coordinate (B2) at (7,3); 
\coordinate (B3) at (1,7); \coordinate (B4) at (7,7); 
\coordinate (B5) at (1,11); \coordinate (B6) at (7,11); 

\coordinate (W1) at (4,1); \coordinate (W2) at (10,1); 
\coordinate (W3) at (4,5); \coordinate (W4) at (10,5); 
\coordinate (W5) at (4,9); \coordinate (W6) at (10,9); 
\node at (0,0)
{\scalebox{0.28}{
\begin{tikzpicture}[sarrow/.style={black, -latex, very thick}, ssarrow/.style={black, latex-, very thick}]

\draw[line width=0.1cm, \nodecolor]  (0,0) rectangle (12,12); 

\node [font=\fontsize{40pt}{60pt}\selectfont] (V0) at (4,7) {$0$} ; 
\node [font=\fontsize{40pt}{60pt}\selectfont] (V1) at (7,5) {$1$} ; 
\node [font=\fontsize{40pt}{60pt}\selectfont] (V2a) at (0,7) {$2$} ; 
\node [font=\fontsize{40pt}{60pt}\selectfont] (V2b) at (12,7) {$2$} ; 
\node [font=\fontsize{40pt}{60pt}\selectfont] (V3) at (5.4,10) {$3$} ; 
\node [font=\fontsize{40pt}{60pt}\selectfont] (V4a) at (0,9) {$4$} ; 
\node [font=\fontsize{40pt}{60pt}\selectfont] (V4b) at (12,9) {$4$} ; 
\node [font=\fontsize{40pt}{60pt}\selectfont] (V5a) at (0,0) {$5$} ; 
\node [font=\fontsize{40pt}{60pt}\selectfont] (V5b) at (12,0) {$5$} ; 
\node [font=\fontsize{40pt}{60pt}\selectfont] (V5c) at (12,12) {$5$} ; 
\node [font=\fontsize{40pt}{60pt}\selectfont] (V5d) at (0,12) {$5$} ; 
\node [font=\fontsize{40pt}{60pt}\selectfont] (V6a) at (0,4.5) {$6$} ; 
\node [font=\fontsize{40pt}{60pt}\selectfont] (V6b) at (12,3.5) {$6$} ; 
\node [font=\fontsize{40pt}{60pt}\selectfont] (V7) at (5,2) {$7$} ; 

\draw[line width=\edgewidth, \nodecolor] (B1)--(W1); \draw[line width=\edgewidth, \nodecolor] (B1)--(W3); 
\draw[line width=\edgewidth, \nodecolor] (B2)--(W2); \draw[line width=\edgewidth, \nodecolor] (B2)--(W3); 
\draw[line width=\edgewidth, \nodecolor] (B2)--(W4); 
\draw[line width=\edgewidth, \nodecolor] (B3)--(W3); \draw[line width=\edgewidth, \nodecolor] (B3)--(W5); 
\draw[line width=\edgewidth, \nodecolor] (B4)--(W3); \draw[line width=\edgewidth, \nodecolor] (B4)--(W4); 
\draw[line width=\edgewidth, \nodecolor] (B4)--(W5); \draw[line width=\edgewidth, \nodecolor] (B4)--(W6); 
\draw[line width=\edgewidth, \nodecolor] (B5)--(W5); \draw[line width=\edgewidth, \nodecolor] (B6)--(W6); 

\draw[line width=\edgewidth, \nodecolor] (W1)--(2.5,0); \draw[line width=\edgewidth, \nodecolor] (B5)--(2.5,12); 
\draw[line width=\edgewidth, \nodecolor] (W1)--(5.5,0); \draw[line width=\edgewidth, \nodecolor] (B6)--(5.5,12); 
\draw[line width=\edgewidth, \nodecolor] (W2)--(12,2.3334); \draw[line width=\edgewidth, \nodecolor] (B1)--(0,2.3334); 
\draw[line width=\edgewidth, \nodecolor] (W2)--(8.5,0); \draw[line width=\edgewidth, \nodecolor] (B6)--(8.5,12); 
\draw[line width=\edgewidth, \nodecolor] (W4)--(12,6.3334); \draw[line width=\edgewidth, \nodecolor] (B3)--(0,6.3334);
\draw[line width=\edgewidth, \nodecolor] (W6)--(12,7.6666); \draw[line width=\edgewidth, \nodecolor] (B3)--(0,7.6666); 
\draw[line width=\edgewidth, \nodecolor] (W6)--(12,10.3334); \draw[line width=\edgewidth, \nodecolor] (B5)--(0,10.3334); 

\draw  [line width=\nodewidth, fill=\nodecolor, \nodecolor] (B1) circle [radius=\noderad] ; 
\draw  [line width=\nodewidth, fill=\nodecolor, \nodecolor] (B2) circle [radius=\noderad] ;
\draw  [line width=\nodewidth, fill=\nodecolor, \nodecolor] (B3) circle [radius=\noderad] ; 
\draw  [line width=\nodewidth, fill=\nodecolor, \nodecolor] (B4) circle [radius=\noderad] ;
\draw  [line width=\nodewidth, fill=\nodecolor, \nodecolor] (B5) circle [radius=\noderad] ; 
\draw  [line width=\nodewidth, fill=\nodecolor, \nodecolor] (B6) circle [radius=\noderad] ;
\draw [line width=\nodewidth, \nodecolor, fill=white] (W1) circle [radius=\noderad] ; 
\draw [line width=\nodewidth, \nodecolor, fill=white] (W2) circle [radius=\noderad] ;
\draw [line width=\nodewidth, \nodecolor, fill=white] (W3) circle [radius=\noderad] ; 
\draw [line width=\nodewidth, \nodecolor, fill=white] (W4) circle [radius=\noderad] ;
\draw [line width=\nodewidth, \nodecolor, fill=white] (W5) circle [radius=\noderad] ; 
\draw [line width=\nodewidth, \nodecolor, fill=white] (W6) circle [radius=\noderad] ;

\draw[line width=\pmwidth, color=\pmcolor] (V3)--(V0); \draw[line width=\pmwidth, color=\pmcolor] (V6a)--(V0); 
\draw[line width=\pmwidth, color=\pmcolor] (V6b)--(V1); 
\draw[line width=\pmwidth, color=\pmcolor] (V5a)--(V7); \draw[line width=\pmwidth, color=\pmcolor] (V5b)--(V7); 
\draw[line width=\pmwidth, color=\pmcolor] (V5c)--(V4b); \draw[line width=\pmwidth, color=\pmcolor] (V5d)--(V4a); 

\draw [sarrow, line width=\arrowwidth] (V0)--(V1); \draw [sarrow, line width=\arrowwidth] (V0)--(V4a); 
\draw [sarrow, line width=\arrowwidth] (V3)--(V0); \draw [sarrow, line width=\arrowwidth] (V6a)--(V0); 
\draw [sarrow, line width=\arrowwidth] (V1)--(V2b); \draw [sarrow, line width=\arrowwidth] (V1)--(V7); 
\draw [sarrow, line width=\arrowwidth] (V6b)--(V1); \draw [sarrow, line width=\arrowwidth] (V2b)--(V3); 
\draw [sarrow, line width=\arrowwidth] (V2b)--(V6b); \draw [sarrow, line width=\arrowwidth] (V4a)--(V3); 
\draw [sarrow, line width=\arrowwidth] (V3)--(V5c); \draw [sarrow, line width=\arrowwidth] (V5a)--(V7); 
\draw [sarrow, line width=\arrowwidth] (V6a)--(V5a); \draw [sarrow, line width=\arrowwidth] (V7)--(V6a); 
\draw [sarrow, line width=\arrowwidth] (V7)--(V6b); \draw [sarrow, line width=\arrowwidth] (V2a)--(V6a); 
\draw [sarrow, line width=\arrowwidth] (V4a)--(V2a); \draw [sarrow, line width=\arrowwidth] (V4b)--(V2b); 
\draw [sarrow, line width=\arrowwidth] (V6b)--(V5b); 
\draw [sarrow, line width=\arrowwidth] (V5d)--(V4a); \draw [sarrow, line width=\arrowwidth] (V5c)--(V4b); 
\draw [sarrow, line width=\arrowwidth] (V3)--(V5d); \draw [sarrow, line width=\arrowwidth] (V5b)--(V7); 
\draw [sarrow, line width=\arrowwidth] (V7)--(5.2,0); \draw [sarrow, line width=\arrowwidth] (5.2,12)--(V3); 
\end{tikzpicture}
} };

\draw[->, line width=0.03cm] (2.2,0)--node[above] {{\scriptsize$\mu_0^\rmL$}} (3.8,0);

\node at (6,0) 
{\scalebox{0.28}{
\begin{tikzpicture}[sarrow/.style={black, -latex, very thick}, ssarrow/.style={black, latex-, very thick}]
\coordinate (B1) at (2,2); \coordinate (B2) at (7.6,3.5); 
\coordinate (B3) at (0.6,7.3); \coordinate (B4) at (7.5,7.5); 
\coordinate (B5) at (3.5,9.5); \coordinate (B6) at (7,11.3); 
\coordinate (B7) at (3.8,5.6); 

\coordinate (W1) at (3.8,0.9); \coordinate (W2) at (9.7,2); 
\coordinate (W3) at (4.5,3.8); \coordinate (W4) at (9.4,5.3); 
\coordinate (W5) at (5.3,7.5); \coordinate (W6) at (9.5,10.1); 
\coordinate (W7) at (2.3,7); 
\draw[line width=0.1cm, \nodecolor]  (0,0) rectangle (12,12); 
\node [font=\fontsize{40pt}{60pt}\selectfont] (V0) at (4,7) {$0$} ; 
\node [font=\fontsize{40pt}{60pt}\selectfont] (V1) at (7,5) {$1$} ; 
\node [font=\fontsize{40pt}{60pt}\selectfont] (V2b) at (10,7) {$2$} ; 
\node [font=\fontsize{40pt}{60pt}\selectfont] (V3) at (5.4,10) {$3$} ; 
\node [font=\fontsize{40pt}{60pt}\selectfont] (V4a) at (2,9) {$4$} ; 
\node [font=\fontsize{40pt}{60pt}\selectfont] (V5a) at (0,0) {$5$} ; 
\node [font=\fontsize{40pt}{60pt}\selectfont] (V5b) at (12,0) {$5$} ; 
\node [font=\fontsize{40pt}{60pt}\selectfont] (V5c) at (12,12) {$5$} ; 
\node [font=\fontsize{40pt}{60pt}\selectfont] (V5d) at (0,12) {$5$} ; 
\node [font=\fontsize{40pt}{60pt}\selectfont] (V6a) at (0,4.5) {$6$} ; 
\node [font=\fontsize{40pt}{60pt}\selectfont] (V6b) at (12,3.5) {$6$} ; 
\node [font=\fontsize{40pt}{60pt}\selectfont] (V7) at (5,2) {$7$} ; 

\draw[line width=\edgewidth, \nodecolor] (B1)--(W1); \draw[line width=\edgewidth, \nodecolor] (B1)--(W3); 
\draw[line width=\edgewidth, \nodecolor] (B2)--(W2); \draw[line width=\edgewidth, \nodecolor] (B2)--(W3); 
\draw[line width=\edgewidth, \nodecolor] (B2)--(W4); \draw[line width=\edgewidth, \nodecolor] (B4)--(W4); 
\draw[line width=\edgewidth, \nodecolor] (B4)--(W5); \draw[line width=\edgewidth, \nodecolor] (B4)--(W6); 
\draw[line width=\edgewidth, \nodecolor] (B5)--(W5); \draw[line width=\edgewidth, \nodecolor] (B6)--(W6); 

\draw[line width=\edgewidth, \nodecolor] (W3)--(B7); \draw[line width=\edgewidth, \nodecolor] (W5)--(B7); 
\draw[line width=\edgewidth, \nodecolor] (W7)--(B3); \draw[line width=\edgewidth, \nodecolor] (W7)--(B5); 
\draw[line width=\edgewidth, \nodecolor] (W7)--(B7); 

\draw[line width=\edgewidth, \nodecolor] (W1)--(3.725,0); \draw[line width=\edgewidth, \nodecolor] (B5)--(3.725,12); 

\draw[line width=\edgewidth, \nodecolor] (W1)--(5.4,0); \draw[line width=\edgewidth, \nodecolor] (B6)--(5.4,12); 
\draw[line width=\edgewidth, \nodecolor] (W2)--(12,2); \draw[line width=\edgewidth, \nodecolor] (B1)--(0,2); 
\draw[line width=\edgewidth, \nodecolor] (W2)--(7.7,0); \draw[line width=\edgewidth, \nodecolor] (B6)--(7.7,12); 
\draw[line width=\edgewidth, \nodecolor] (W4)--(12,6.925); \draw[line width=\edgewidth, \nodecolor] (B3)--(0,6.925); 
\draw[line width=\edgewidth, \nodecolor] (W6)--(12,9.85); \draw[line width=\edgewidth, \nodecolor] (B5)--(0,9.85); 
\draw[line width=\edgewidth, \nodecolor] (W6)--(12,7.842); \draw[line width=\edgewidth, \nodecolor] (B3)--(0,7.842); 

\draw  [line width=\nodewidth, fill=\nodecolor, \nodecolor] (B1) circle [radius=\noderad] ; 
\draw  [line width=\nodewidth, fill=\nodecolor, \nodecolor] (B2) circle [radius=\noderad] ;
\draw  [line width=\nodewidth, fill=\nodecolor, \nodecolor] (B3) circle [radius=\noderad] ; 
\draw  [line width=\nodewidth, fill=\nodecolor, \nodecolor] (B4) circle [radius=\noderad] ;
\draw  [line width=\nodewidth, fill=\nodecolor, \nodecolor] (B5) circle [radius=\noderad] ; 
\draw  [line width=\nodewidth, fill=\nodecolor, \nodecolor] (B6) circle [radius=\noderad] ;
\draw  [line width=\nodewidth, fill=\nodecolor, \nodecolor] (B7) circle [radius=\noderad] ;
\draw [line width=\nodewidth, \nodecolor, fill=white] (W1) circle [radius=\noderad] ; 
\draw [line width=\nodewidth, \nodecolor, fill=white] (W2) circle [radius=\noderad] ;
\draw [line width=\nodewidth, \nodecolor, fill=white] (W3) circle [radius=\noderad] ; 
\draw [line width=\nodewidth, \nodecolor, fill=white] (W4) circle [radius=\noderad] ;
\draw [line width=\nodewidth, \nodecolor, fill=white] (W5) circle [radius=\noderad] ; 
\draw [line width=\nodewidth, \nodecolor, fill=white] (W6) circle [radius=\noderad] ;
\draw [line width=\nodewidth, \nodecolor, fill=white] (W7) circle [radius=\noderad] ;

\draw[line width=\pmwidth, color=\pmcolor] (V3)--(V1); \draw[line width=\pmwidth, color=\pmcolor] (V6a)--(V1); 
\draw[line width=\pmwidth, color=\pmcolor] (V6a)--(V1); \draw[line width=\pmwidth, color=\pmcolor] (V6b)--(V1); 
\draw[line width=\pmwidth, color=\pmcolor] (V5a)--(V7); \draw[line width=\pmwidth, color=\pmcolor] (V5b)--(V7); 
\draw[line width=\pmwidth, color=\pmcolor] (V5d)--(V4a); \draw[line width=\pmwidth, color=\pmcolor] (V6a)--(V4a); 

\draw [ssarrow, line width=\arrowwidth] (V0)--(V1); \draw [ssarrow, line width=\arrowwidth] (V0)--(V4a); 
\draw [ssarrow, line width=\arrowwidth] (V3)--(V0); \draw [ssarrow, line width=\arrowwidth] (V6a)--(V0); 

\draw [sarrow, line width=\arrowwidth] (V3)--(V1); \draw [sarrow, line width=\arrowwidth] (V6a)--(V1); 
\draw [sarrow, line width=\arrowwidth] (V6a)--(V4a); 
\draw [sarrow, line width=\arrowwidth] (V1)--(V2b); \draw [sarrow, line width=\arrowwidth] (V1)--(V7); 
\draw [sarrow, line width=\arrowwidth] (V6b)--(V1); \draw [sarrow, line width=\arrowwidth] (V2b)--(V3); 
\draw [sarrow, line width=\arrowwidth] (V2b)--(V6b); 
\draw [sarrow, line width=\arrowwidth] (V3)--(V5c); \draw [sarrow, line width=\arrowwidth] (V5a)--(V7); 
\draw [sarrow, line width=\arrowwidth] (V6a)--(V5a); \draw [sarrow, line width=\arrowwidth] (V7)--(V6a); 
\draw [sarrow, line width=\arrowwidth] (V7)--(V6b); 
\draw [sarrow, line width=\arrowwidth] (V6b)--(V5b); 
\draw [sarrow, line width=\arrowwidth] (V5d)--(V4a); 
\draw [sarrow, line width=\arrowwidth] (V3)--(V5d); \draw [sarrow, line width=\arrowwidth] (V5b)--(V7); 
\draw [sarrow, line width=\arrowwidth] (V7)--(5.2,0); \draw [sarrow, line width=\arrowwidth] (5.2,12)--(V3); 
\draw [sarrow, line width=\arrowwidth] (V4a)--(0,8); \draw [sarrow, line width=\arrowwidth] (12,8)--(V2b); 
\end{tikzpicture}
} }; 
\end{tikzpicture}
\end{center}
\caption{The left mutation of the graded QP $(Q,W_Q,d_{\sfD_7})$ at $0\in Q_0^\mu$}
\label{ex_mut_gQP}
\end{figure}

By Theorem~\ref{pm_mutation_equiv} and \ref{main_thm_derived_eq}, for any internal perfect matching $\sfD$ (resp. $\sfD^\prime$) of $Q$ (resp. $\mu_0(Q)$) that is mutation equivalent to $\sfD_7$ (resp. $\mu_0(\sfD_7)$), 
$\widehat{\calP}(Q,W_Q)_\sfD$ and $\widehat{\calP}(\mu_0(Q,W_Q))_{\sfD^\prime}$ are derived equivalent. 
\end{example}

\subsection*{Acknowledgement} 
The author would like to thank Osamu Iyama for valuable comments and suggestions, especially for pointing out \ref{rem_conifold} and \ref{existence_tilting}. 
The author also thank Yuya Mizuno for stimulating discussions concerning $n$-representation infinite algebras. 
The author would like to thank the anonymous referees for numerous valuable comments and suggestions, 
especially the ones for improving the proof of Theorem~\ref{pm_mutation_equiv}. 

The author is supported by World Premier International Research Center Initiative (WPI initiative) MEXT Japan, JSPS Grant-in-Aid for Young Scientists (B) 17K14159, and JSPS Grant-in-Aid for Early-Career Scientists 20K14279. 


\end{document}